\pdfoutput=1
\documentclass[a4paper, 11pt]{article}

\usepackage[final]{showkeys} 

\usepackage{header} 
\usepackage{abbrev} 
\usepackage{abbrev_fk}

\graphicspath{{./images/}}

\newcommand\nlra{\nleftrightarrow}
\newcommand{\ba}{\boldsymbol{\alpha}}
\newcommand{\bb}{\boldsymbol{\beta}}

\usepackage{enumitem}
\setlist[itemize]{itemsep=1pt, topsep=4pt}
\setlist[enumerate]{itemsep=1pt, topsep=4pt}

\title{Universality for the random-cluster model on isoradial graphs}
\author{
  Hugo Duminil-Copin \thanks{Institut des Hautes Études Scientifiques}\ \thanks{Université de Genève},
  Jhih-Huang Li \addtocounter{footnote}{-1}\footnotemark,
  Ioan Manolescu \thanks{Universit\'e de Fribourg}
}

\begin{document}

\maketitle

\vspace{2cm}

\begin{abstract}
  We show that the canonical random-cluster measure associated to isoradial graphs is critical for all $q \geq 1$.
  Additionally, we prove that the phase transition of the model is of the same type on all isoradial graphs:
  continuous for $1 \leq q \leq 4$ and discontinuous for $q > 4$.
  For $1 \leq q \leq 4$, the arm exponents (assuming their existence) are shown to be the same for all isoradial graphs. 
  In particular, these properties also hold on the triangular and hexagonal lattices.
  Our results also include the limiting case of quantum random-cluster models in $1+1$ dimensions. 
\end{abstract}

\vspace{1cm}

\newpage

\tableofcontents

\newpage

\section{Introduction}

The random-cluster model is a dependent percolation model that generalises Bernoulli percolation.
It was introduced by Fortuin and Kasteleyn in~\cite{ForKas72} to unify percolation theory, electrical network theory and the Potts model.
The spin correlations of the Potts model get rephrased as cluster connectivity properties of its random-cluster representation, and can therefore be studied using probabilistic techniques coming from percolation theory. 

The random-cluster model on the square lattice has been the object of intense study in the past few decades.
A duality relation enables to prove that the model undergoes a phase transition at the self-dual value $p_c=\tfrac{\sqrt{q}}{1+\sqrt{q}}$ of the edge-parameter~\cite{BefDum12} (see also~\cite{DumMan16,DumRaoTas16,DumRaoTas17}).
It can also be proved that the distribution of the size of finite clusters has exponential tails when the model is non-critical.
Also, the critical phase is now fairly well understood: the phase transition of the model is continuous if the cluster-weight belongs to $[1,4]$~\cite{DumSidTas13} and discontinuous if it is greater than $4$~\cite{DumGanHar16}.
When the cluster-weight is equal to 2, the random-cluster model is coupled with the Ising model, and is known to be conformally invariant~\cite{Smi10,CheDumHon14} (we also refer to~\cite{DumSmi12a} for a review).

A general challenge in statistical physics consists in understanding universality, 
i.e., that the behaviour of a certain model is not affected by small modifications of its definition.
This is closely related to the so-called conformal invariance of scaling limits: 
when we scale out the model at criticality, the resulting limit should be preserved under conformal transformations, including translations, rotations and Möbius maps.

The goal of this paper is to prove a form of universality for a certain class of random-cluster models. 
Specifically we aim to transfer results obtained for the square lattice to a larger class of graphs called {\em isoradial graphs}, i.e., planar graphs embedded in the plane in such a way that every face is inscribed in a circle of radius one.
A specific random-cluster model is associated to each such graph, where the edge-weight of every edge is an explicit function of its length. 
Moreover, the edge-weight is expected to compensate the inhomogeneity of the embedding and render the model conformally invariant in the limit.

Isoradial graphs were introduced by Duffin in~\cite{Duf68} in the context of discrete complex analysis, and later appeared in the physics literature in the work of Baxter~\cite{Bax78}, where they are called $Z$-invariant graphs.
The term isoradial was only coined later by Kenyon, who studied discrete complex analysis on these graphs~\cite{Ken02}.
Since then, isoradial graphs have been studied extensively; we refer to~\cite{CheSmi12, KS-quad-graphs, Mer01} for literature on the subject.
Several mathematical studies of statistical mechanics on isoradial graphs have appeared in recent years.
The connection between the dimer and Ising models on isoradial graphs was studied in~\cite{BouTil10,BouTil11}.
In \cite{CheSmi12}, the scaling limit of the Ising model and that of its associated random-cluster  model with $q = 2$ was shown to be the same on isoradial graphs as on the square lattice. 
For other values of $q \geq 1$, the existence of the scaling limit of the random-cluster  model is still out of reach.
However, for Bernoulli percolation (which corresponds to $q = 1$) a universality result for isoradial graphs was obtained in~\cite{GriMan13,GriMan13a,GriMan14}.
In the present paper, we generalise the result of~\cite{GriMan14} to all random-cluster models with $q \geq 1$. 

\subsection{Definition of the model}

\paragraph{Isoradial graphs}

An \emph{isoradial graph} $\bbG = (\bbV, \bbE)$ is a planar graph embedded in the plane in such a way that (i) every face is inscribed in a circle of radius 1 and (ii) the centre of each circumcircle is contained in the corresponding face.
We sometimes call the embedding \emph{isoradial}. 
Note that an isoradial graph is necessarily infinite.

Given an isoradial graph (which we call the {\em primal graph}), we can construct its {\em dual graph} $\bbG^* = (\bbV^*, \bbE^*)$ as follows: $\bbV^*$ is composed of the centres of circumcircles of faces of $\bbG$.
By construction, every face of $\bbG$ is associated with a dual vertex.
Then, $\bbE^*$ is the set of edges between dual vertices whose corresponding faces share an edge in $\bbG$.
Edges of $\bbE^*$ are in one-to-one correspondence with those of $\bbE$.
We denote the dual edge associated to $e \in \bbE$ by $e^*$.
When constructed like this, $\bbG^*$ is also an isoradial graph. 

The {\em diamond graph} $\bbG^\diamond$ associated to $\bbG$ (and $\bbG^*$) has vertex set $\bbV \cup \bbV^*$ and an edge between vertices $u \in \bbV$ and $v \in \bbV^*$ if $v$ is the centre of a face containing $u$. All edges of $\bbG^\diamond$ are of length $1$, and $\bbG^\diamond$ is a rhombic tiling of the plane. Conversely, each rhombic tiling of the plane corresponds to a primal/dual pair of isoradial graphs. 
It will be often convenient to think of isoradial graphs through their diamond graphs. 
See Figure~\ref{fig:isoradial} for an illustration.

\begin{figure}[htb]
  \centering
  \includegraphics[width=0.4\textwidth, page=1]{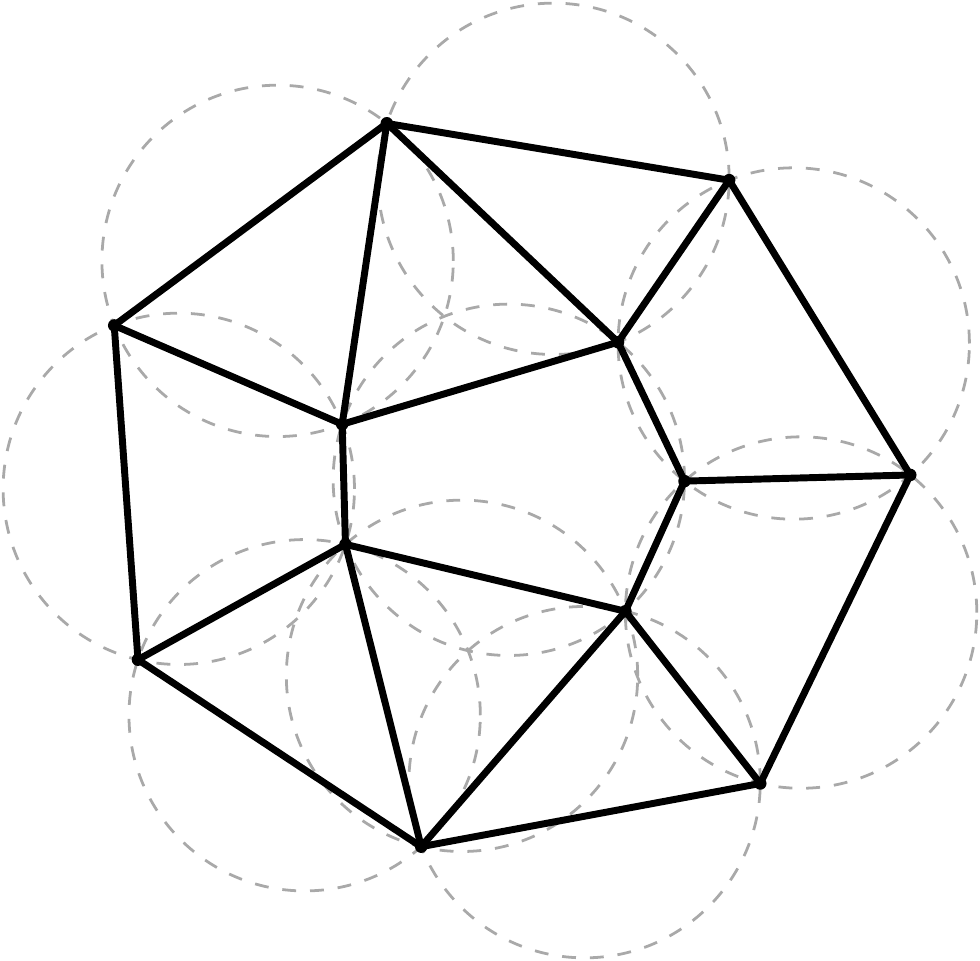}
  \hspace{1cm}
  \includegraphics[width=0.4\textwidth, page=2]{isoradial.pdf}
  \caption{The black graph is (a finite part of) an isoradial graph.
    All its finite faces can be inscribed into circumcircles of radius one. 
    The dual vertices (in white) have been drawn in such  a way that they are the centres of these circles; the dual edges are in dotted lines.
    The diamond graphs is drawn in gray in the right picture.}
  \label{fig:isoradial}
\end{figure}

The isoradial graphs $\bbG$ considered in this paper are assumed to be \emph{doubly-periodic}, in the sense that they are invariant under the action of a certain lattice $\Lambda \approx \bbZ \oplus \bbZ$.
In such case, $\bbG/ \Lambda$ is a finite graph embedded in the torus $\bbT := \bbR^2 / \Lambda$.
We will always translate $\bbG$ so that $0$ is a vertex of $\bbG$, which we call the origin.

\paragraph{The random-cluster model.}

Fix an isoradial graph $\bbG= (\bbV, \bbE)$. 
For $q \geq 1$ and $\beta > 0$, each edge $e \in \bbE$ is assigned a weight $p_e(\beta)$ given by
\begin{align}
  \text{if } 1 \leq q < 4, \quad &
  y_e(\beta) = \beta \sqrt{q} \, \tfrac{\sin( r (\pi - \theta_e))}{\sin( r \theta_e)}, 
  & \quad \text{where }  r = \tfrac{1}{\pi} \cos^{-1} \left( \tfrac{\sqrt{q}}{2} \right); \nonumber \\
  \text{if } q = 4, \quad & 
  y_e(\beta) = \beta \, \tfrac{2 (\pi - \theta_e)}{\theta_e}; \nonumber \\
  \text{if } q > 4, \quad  &
  y_e(\beta) = \beta \sqrt{q} \, \tfrac{\sinh( r (\pi - \theta_e))}{\sinh( r\theta_e)}, 
  & \quad \text{where } r = \tfrac{1}{\pi} \cosh^{-1} \left( \tfrac{\sqrt{q}}{2} \right),
  \label{eq:parameters}
\end{align}
where $y_e(\beta) = \frac{p_e(\beta)}{1 - p_e(\beta)}$ and $\theta_e \in (0, \pi)$ is the angle subtended by $e$.
That is, $\theta_e$ is the angle at the centre of the circle corresponding to any of the two faces bordered by $e$; see Figure~\ref{fig:subtended_angle}.
This family of values can be found in~\cite[Sec.~5.3 and Prop.~2]{Kenyon-dimer} and~\cite{BefDumSmi12}.
Note that the expression for $q = 4$ is the common limit $q \rightarrow 4$ of the expressions for $q<4$ and $q>4$.

\begin{figure}[htb]
  \centering
  \includegraphics[scale=1.2]{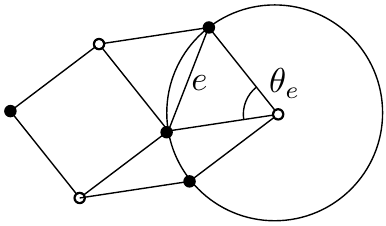}
  \caption{The edge $e \in \bbE$ and its subtended angle $\theta_e$.}
  \label{fig:subtended_angle}
\end{figure}

The \emph{random-cluster model} on a finite subgraph $G=(V,E)$ of $\bbG$ is defined as follows (see~also \cite{Gri06} for a manuscript on the subject).
A random-cluster configuration $\omega = (\omega_e : e \in E)$ is an element of $\{0,1\}^E$.
A configuration can be seen as a graph with vertex set $V$ and edge set $\{ e \in E : \omega_e = 1 \}$.
Write $k_0(\omega)$ for the number of connected components, also called \emph{clusters}, of the graph $\omega$.
For $q > 0$ and $\beta > 0$, the probability of a configuration $\omega$ is equal to
\begin{equation}
  \rcisolaw{\beta}{G}{0}(\omega) :=
  \frac{\dis q^{k_0(\omega)} \, \prod_{e\in E} y_e(\beta)^{\omega_e}}{Z^0(G,\beta,q)},
\end{equation}
where $Z^0(G,\beta,q)$ is a normalising constant called the partition function; it is chosen such that $\rcisolaw{\beta}{G}{0}$ is a probability measure. 
The measure $\rcisolaw{\beta}{G}{0}$ is called the random-cluster measure with \emph{free boundary conditions}.
Similarly, one defines the random-cluster measures $\rcisolaw{\beta}{G}{1}$ with \emph{wired boundary conditions} as follows. 
Let $\partial G$ be the set of vertices of $G$ with at least one neighbour outside of $G$: 
\begin{align*}
  \partial G = \big\{ u \in V:\, \exists v \in \bbV \setminus V \text{ such that } \{u,v\} \in \bbE\big\}.
\end{align*}
Write $k_1(\omega)$ for the number of connected components of $\omega$, when all connected components intersecting $\partial G$ are counted as $1$. 
Then, $\rcisolaw{\beta}{G}{1}$ is defined as $\rcisolaw{\beta}{G}{0}$, with $k_1$ instead of $k_0$. 

Other boundary conditions may be defined and stand for connections outside of $G$. 
They are represented by partitions of $\partial G$. 
\Rcms with such boundary conditions are defined as above,
with the number of connected components intersecting $\partial G$ being computed in a way that accounts for connections outside~$G$. 

For a configuration $\omega$ on $G$, its {\emph dual configuration} $\omega^*$ is the configuration on $G^*$ defined by 
$\omega^*(e^*) = 1 - \omega(e)$ for all $e \in E$.
If $\omega$ is chosen according to $\rcisolaw{\beta}{G}{\xi}$ for some boundary conditions $\xi$, 
then $\omega^*$ has law $\rcisolaw{\beta^{-1}}{G^*}{\xi^*}$, where $\xi^*$ are boundary conditions that depend on $\xi$. 
Most notably, if $\xi \in \{0,1\}$ then $\xi^* = 1-\xi$. 
Considering the above, one may be tempted to declare the models with $\beta = 1$ self-dual. 
Note however that the dual graph is generally different from the primal, and the tools associated with self-duality do not apply.

For $q \geq 1$, versions of these measures may be obtained for the infinite graph $\bbG$ 
by taking weak limits of measures on finite subgraphs $G$ of $\bbG$ that increase to $\bbG$ (see ~\cite[Sec.~4]{Gri06}).
The measures on $G$ should be taken with free or wired boundary conditions; 
the limiting measures are then denoted by $\rcisolaw{\beta}{\bbG}{0}$ and $\rcisolaw{\beta}{\bbG}{1}$, respectively,
and are called {\em infinite-volume measures} with free and wired boundary conditions.

We will be interested in connectivity properties of the (random) graph $\omega \in \{0,1\}^{\bbE}$.
For $A, B \subset \bbR^2$, we say that $A$ and $B$ are \emph{connected}, denoted by $A \leftrightarrow B$, if there exists a connected component of $\omega$ intersecting both $A$ and $B$ (here we see edges in $\omega$ as subsets of the plane).
Similarly, for a region $R \subset \bbR^2$, we say that $A$ and $B$ are \emph{connected in $R$}, denoted by $A \xleftrightarrow{R} B$, if there exists a connected component of $\omega \cap R$ 
intersecting both $A$ and $B$.
For $u \in \bbV$, write $u \leftrightarrow \infty$ if $u$ is in an infinite connected component of $\omega$. 

Let $B_n$ be the ball of radius $n$ for the Euclidean distance, and $\partial B_n$ its boundary.
Below, we will often identify a subset $S$ of the plane with the subgraph of $\bbG$ induced by the vertices (of $\bbG$) within it. 

\subsection{Results for the classical random-cluster model}

The square lattice embedded so that each face is a square of side-length $\sqrt 2$ is an isoradial graph. 
We will denote it abusively by $\bbZ^2$ and call it the {\em regular square lattice}. 
The edge-weight associated to each edge of $\bbZ^2$ by \eqref{eq:parameters} is $\frac{\sqrt q}{1 + \sqrt q}$. 
This was shown in \cite{BefDum12} to be the critical parameter for the random-cluster  model on the square lattice. 
Moreover, the phase transition of the model was shown to be continuous when $q \in [1,4]$ \cite{DumSidTas13} and discontinuous when $q > 4$ \cite{DumGanHar16}.
The following two theorems generalise these results to periodic isoradial graphs. 

\begin{thm} \label{thm:main}
  Fix a doubly-periodic isoradial graph $\bbG$ and $1 \leq q \leq 4$.
  Then,
  \begin{itemize}
  \item $\rcisolaw{1}{\bbG}{1} [ 0 \leftrightarrow \infty ] = 0$ and $\rcisolaw{1}{\bbG}{0} = \rcisolaw{1}{\bbG}{1}$;
  \item there exist $a, b > 0$ such that for all $n \geq 1$,
    \begin{align*}
      n^{-a} \leq \rcisolaw{1}{\bbG}{0} \big[ 0 \leftrightarrow \pd B_n \big] \leq n^{-b};
    \end{align*}
  \item for any $\rho > 0$, there exists $c = c(\rho)>0$ such that for all $n \geq 1$,
    \begin{align*}
      \rcisolaw{1}{R}{0} \big[ \calC_h(\rho n,n) \big] \geq c,
    \end{align*}
    where $R = [-(\rho+1)n, (\rho+1)n] \times [-2n, 2n]$ and $\calC_h(\rho n, n)$ is the event that there exists a path in $\omega \cap [-\rho n, \rho n] \times [-n, n]$ from $\{ -\rho n \} \times [-n, n]$ to $\{ \rho n \} \times [-n, n]$.
  \end{itemize}
\end{thm}

The last property is called the \emph{strong RSW property} (or simply RSW property) 
and may be extended as follows: for any boundary conditions $\xi$, 
\begin{align} \label{eq:strong_RSW}
  c \leq \rcisolaw{1}{R}{\xi} \big[ \calC_h(\rho n, n) \big] \leq 1-c,
\end{align}
for any $n\geq 1$ and some constant $c> 0$ depending only on $\rho$. 
In words, crossing probabilities remain bounded away from 0 and 1 uniformly in boundary conditions and in the size of the box (provided the aspect ratio is kept constant).
For this reason, in some works (e.g. \cite{GriMan14}) the denomination {\em box crossing property} is used.

The strong RSW property was known for Bernoulli percolation on the regular square lattice from the works of Russo and Seymour and Welsh~\cite{Russo-note,SW-percolation}, hence the name. 
The term strong refers to the uniformity in boundary conditions; weaker versions were developed in \cite{BefDum12} for the square lattice. 
Hereafter, we say the model has the {strong RSW property} if~\eqref{eq:strong_RSW} is satisfied.

The strong RSW property is indicative of a continuous phase transition and has numerous applications in describing the critical phase. 
In particular, it implies the first two points of Theorem~\ref{thm:main}.
It is also instrumental in the proofs of mixing properties and the existence of certain critical exponents and subsequential scaling limits of interfaces.
We refer to~\cite{DumSidTas13} for details.

%
\begin{thm}\label{thm:main2}
  Fix a doubly-periodic isoradial graph $\bbG$ and $q > 4$. Then, 
  \begin{itemize}
  \item $\rcisolaw{1}{\bbG}{1} [0 \leftrightarrow \infty ] >0$;
  \item there exists $c > 0$ such that for all $n \geq 1$, $\rcisolaw{1}{\bbG}{0} [ 0 \leftrightarrow \pd B_n ] \leq \exp(-cn).$
  \end{itemize}
\end{thm}

Note that the above result is also of interest for regular graphs such as the triangular and hexagonal lattices.
Indeed, the transfer matrix techniques developed in~\cite{DumGanHar16} are specific to the square lattice and do not easily extend to the triangular and hexagonal lattices.

The strategy of the proof for Theorems~\ref{thm:main} and~\ref{thm:main2} is the same as in~\cite{GriMan14}.
There, Theorem~\ref{thm:main} was proved for $q = 1$ (Bernoulli percolation).
The authors explained how to transfer the RSW property from the regular square lattice model to more general isoradial graphs by modifying the lattice step by step.
The main tool used for the transfer is the \stt.

In this article, we will follow the same strategy, with two additional difficulties:
\begin{itemize} 
\item The model has long-range dependencies, and one must proceed with care when handling boundary conditions.
\item For $q \leq 4$, the RSW property is indeed satisfied for the regular square lattice (this is the result of~\cite{DumSidTas13}), and may be transferred to other isoradial graphs. 
  This is not the case for $q > 4$, where a different property needs to be transported, and some tedious new difficulties arise.
\end{itemize}

The results above may be extended to isoradial graphs which are not periodic but satisfy the so-called \emph{bounded angles property} and an additional technical assumption termed the \emph{square-grid property} in~\cite{GriMan14}. We will not discuss this generalisation here and simply stick to the case of doubly-periodic graphs.
Interested readers may consult~\cite{GriMan14} for the exact conditions required for $\bbG$; the proofs below adapt readily.


A direct corollary of the previous two theorems is that isoradial random-cluster models are critical for $\beta = 1$.
This was already proved for $q>4$ in~\cite{BefDumSmi12} using different tools. 

\begin{cor} \label{cor:phase_transition}
  Fix $\bbG$ a doubly-periodic isoradial graph and $q \geq 1$. 
  Then, for any $\beta \ne 1$, one has $\rcisolaw{\beta}{\bbG}{1}=\rcisolaw{\beta}{\bbG}{0}$ and 
  \begin{itemize} 
  \item when $\beta < 1$, there exists $c_\beta>0$ such that for any $x, y \in \bbV$, 
    $$\rcisolaw{\beta}{\bbG}{1}[ x \leftrightarrow y] \leq \exp(-c_\beta \|x-y\|);$$
  \item when $\beta > 1$, $\rcisolaw{\beta}{\bbG}{0} [ x \leftrightarrow \infty ] > 0$ for any $x \in \bbV$.
  \end{itemize}
\end{cor}


\medbreak

For $1 \leq q \leq 4$, arm exponents at the critical point $\beta = 1$ are believed to exist and to be universal (that is they depend on $q$ and the dimension, but not on the structure of the underlying graph).
Below we define the arm events and effectively state the universality of the exponents, but do not claim their existence. 

Fix $k \in \{ 1 \} \cup 2\bbN$.
For $N > n$, define the \emph{$k$-arm event} $A_k(n, N)$ to be the event that there exists $k$ disjoint paths $\calP_1, \dots, \calP_k$ in counterclockwise order, contained in $[-N,N]^2 \setminus (-n,n)^2$, connecting $\pd [-n, n]^2$ to $\pd [-N, N]^2$,
with $\calP_1,\calP_3,\dots$ contained in $\omega$ and $\calP_2,\calP_4,\dots$ contained in $\omega^*$. 
Note that this event could be void if $n$ is too small compared to $k$; we will always assume $n$ is large enough to avoid such degenerate situations. 

For continuous phase transitions (that is for $q \in [1,4]$) it is expected that, 
\begin{align*}
  \rcisolaw{1}{R}{0}[A_k(n, N)] = \Big(\frac{n}{N}\Big)^{\alpha_k + o(1)},
\end{align*}
for some $\alpha_k >0$ called the \emph{$k$-arm exponent}. 
The RSW theory provides such polynomial upper and lower bounds, but the exponents do not match.

The one-arm exponent of the model describes the probability for the cluster of a given point to have large radius under the critical measure; 
the four-arm exponent is related to the probability for an edge to be pivotal for connection events.

\begin{thm}[Universality of arm exponents] \label{thm:universality}
  Fix $\bbG$ a doubly-periodic isoradial graph and $1 \leq q \leq 4$.
  Then, for any $k \in \{ 1 \} \cup 2 \bbN$, there exists a constant $c > 0$ such that, for all $N > n$ large enough,
  \begin{align*}
    c \, \rcisolaw{1}{\bbZ^2}{0} [ A_k(n, N) ]
    \leq \rcisolaw{1}{\bbG}{0} [ A_k(n, N) ]
    \leq c^{-1} \rcisolaw{1}{\bbZ^2}{0} [ A_k(n, N) ].
  \end{align*}
\end{thm}

\subsection{Results for the quantum random-cluster model}


The random-cluster model admits a quantum version, as described in~\cite[Sec.~9.3]{Gri10} for $q = 2$. 
Consider the set $\bbZ \times \bbR$ as a system of vertical axis. 
Let $\calC$ and $\calB$ be two independent Poisson point processes with parameters $\lambda$ and $\mu$ respectively, 
the first on $\bbZ \times \bbR$, the second on $(\tfrac12+\bbZ) \times \bbR$. 
Call the points of the former \emph{cuts} and those of the latter \emph{bridges}. 
For any realisation of the two processes, let $\omega$ be the subset of $\bbR^2$ formed of: 
\begin{itemize}
\item the set $\bbZ\times \bbR$ with the exception of the points in $\calB$;
\item a horizontal segment of length $1$ centered at every point of $\calC$. 
\end{itemize}
For a rectangle $R = [a,b] \times [c,b] \subset \bbR^2$ with $a,b\in \bbZ$, 
define the quantum \rcm on $R$ by weighing each configuration $\omega$ with respect to the number of clusters in $\omega$.
More precisely, we define $\phi_{\calQ, R, \lambda, \mu}$ to be the quantum \rcm with parameters $\lambda, \mu$ and $q > 0$ by
\begin{align*}
  \dd \quantumrc{0}{R} (\omega) \propto q^{k(\omega)} \dd \bbP_{\lambda, \mu}  (\omega)
\end{align*}
where $\bbP_{\lambda, \mu}$ is the joint law of the Poisson point processes $\calB$ and $\calC$, 
and $k(\omega)$ is the number of connected components of $\omega \cap R$ (notice that this number is a.s. finite).

Similarly, one may define measures with wired boundary conditions $\phi^1_{\calQ, R, \lambda, \mu}$ by altering the definition of $k$. 
Infinite-volume measures may be defined by taking limits over increasing rectangular regions $R$, as in the classical case. 

As will be discussed in Section~\ref{sec:quantum}, the quantum model may be seen as a limit of isoradial models on increasingly distorted embeddings of the square lattice. 
As a result, statements similar to Theorems~\ref{thm:main},~\ref{thm:main2} and Corollary~\ref{cor:phase_transition} apply to the quantum setting.
In particular, we identify the critical parameters as those with $\frac{\mu}{\lambda} = q$.
This critical value has already been computed earlier in~\cite{Pfeuty-QI, BG-QI-sharp} for the case of the quantum Ising model ($q = 2$).

\begin{thm} \label{thm:quantum}
  If $q \in [1,4]$ and $\mu/\lambda = q$, then
  \begin{itemize}
  \item $\quantumrc{1}{} [ 0 \leftrightarrow \infty ] = 0$ and $\quantumrc{0}{} = \quantumrc{1}{}$;
  \item there exist $a,b > 0$ such that for all $n \geq 1$,
    \begin{align*}
      n^{-a} \leq \quantumrc{0}{} \big[ 0 \leftrightarrow \pd B_n \big] \leq n^{-b};
    \end{align*} 
  \item for any $\rho > 0$, there exists $c = c(\rho)>0$ such that for all $n \geq 1$,
    \begin{align*}
      \quantumrc{0}{R} \big[ \calC_h(\rho n, n) \big] \geq c,
    \end{align*}
    where $R = [-(\rho+1)n, (\rho+1)n] \times [-2n, 2n]$ and $\calC_h(\rho n, n)$ is the event that there exists a path in $\omega \cap [-\rho n, \rho n] \times [-n, n]$ from $\{ -\rho n \} \times [-n, n]$ to $\{ \rho n \} \times [-n, n]$.
  \end{itemize}
  If $q > 4$ and $\mu/\lambda = q$, then    
  \begin{itemize}
  \item $\quantumrc{1}{} [ 0 \leftrightarrow \infty ] > 0$;
  \item there exists $c > 0$ such that for all $n \geq 1$, 
    $\quantumrc{0}{} [ 0 \leftrightarrow \pd B_n ] \leq \exp(-cn).$
  \end{itemize}
  Finally, if $\mu/\lambda \neq q$, then $\quantumrc{0}{} = \quantumrc{1}{}$ and
  \begin{itemize} 
  \item when $\mu/\lambda < q$, there exists $c_{\mu/\lambda} > 0$ such that for any $x, y \in \bbZ \times \bbR$, 
    $$\quantumrc{0}{} [ x \leftrightarrow y ] \leq \exp(-c_{\mu/\lambda}\|x-y\|).$$    
  \item when $\mu/\lambda > q$, $\quantumrc{0}{} [ 0 \leftrightarrow \infty ] > 0$.
  \end{itemize}
\end{thm}

Notice that multiplying both $\lambda$ and $\mu$ by a factor $\alpha$ is tantamount to dilating the configuration $\omega$ 
vertically by a factor of $1/\alpha$. 
Hence it is natural that only the ratio $\mu/\lambda$ plays a role in determining criticality. 

However, for $q \in [1,4]$, there are reasons to believe that for the specific values 
$$
\lambda = \frac{4r}{\sqrt{q(4-q)}} \quad \text{ and } \quad 
\mu = \frac{4r \sqrt q}{\sqrt{4-q}},
$$
the model is rotationally invariant at large scale, as will be apparent from the link to isoradial graphs.

\paragraph{Organisation of the paper}
Section~\ref{sec:star-triangle} contains background on the \stt and how it acts on isoradial graphs. 
It also sets up the strategy for gradually transforming the regular square lattice into general isoradial graphs.
This is done in two stages: first the regular square lattice is transformed into general isoradial square lattices, 
then into bi-periodic isoradial graphs. 
This two-stage process is repeated in each of the following two sections.

The proofs of Theorems~\ref{thm:main} and~\ref{thm:universality} are contained in Section~\ref{sec:q<4}, 
while that of Theorem~\ref{thm:main2} in Section~\ref{sec:q>4}.
The reason for this partition is that the tools in the case $1\leq q \leq 4$ and $q > 4$ are fairly different. 
Section~\ref{sec:quantum} contains the adaptation to the quantum case (Theorem~\ref{thm:quantum}).

Several standard computations involving the random-cluster model and the RSW technology are necessary. 
In order to not over-burden the paper, readers can refer to appendixes of~\cite{Li-thesis}.

\paragraph{Acknowledgments} 
This research is supported by the NCCR SwissMAP, the ERC AG COMPASP, the Swiss NSF and the IDEX Chair from Paris-Saclay.
The third author would like to thank his PhD advisor, G.~Grimmett, for introducing him to this topic. 

\section{\Stt} \label{sec:star-triangle}

In this section, we introduce the main tool of our article: the \emph{star-triangle transformation}, also known as the \emph{Yang-Baxter relation}.
This transformation was first discovered by Kennelly in 1899 in the context of electrical networks~\cite{Kennelly}.
Then, it was discovered to be a key relation in different models of statistical mechanics~\cite{Onsager, Baxter-exactly-solvable} indicative of the integrability of the system.

\subsection{Abstract \stt}

For a moment, we consider graphs as combinatorial objects without any embedding.
Consider the triangle graph $\triangle = (V, E)$ and the star graph $\starg = (V', E')$ shown in Figure~\ref{fig:simple_transformation}; the boundary vertices of both graphs are $\{A, B, C\}$.
Write $\Omega = \{ 0, 1 \}^E$ and $\Omega' = \{ 0, 1 \}^{E'}$ for the two spaces of percolation configurations associated to these two graphs.
Additionally, consider two triplets of parameters, $\boldsymbol{p} = (p_a, p_b, p_c) \in (0, 1)^3$ for the triangle and $\boldsymbol{p'} = (p_a', p_b', p_c') \in (0, 1)^3$ for the star, associated with the edges of the graph as indicated in Figure~\ref{fig:simple_transformation}.
For boundary conditions $\xi$ on $\{ A, B, C \}$, denote by $\trianglelaw$ (and $\starlaw$) the \rcm on $\triangle$ (and $\starg$, respectively) with cluster-weight $q$ and parameters $\boldsymbol{p}$ (and $\boldsymbol{p'}$, respectively). 
For practical reasons write 
$$
y_i = \frac{p_i}{1-p_i} \quad \text{ and } \quad y_i' = \frac{p_i'}{1-p_i'}.
$$

\begin{figure}[htb]
  \centering
  \includegraphics[scale=0.8]{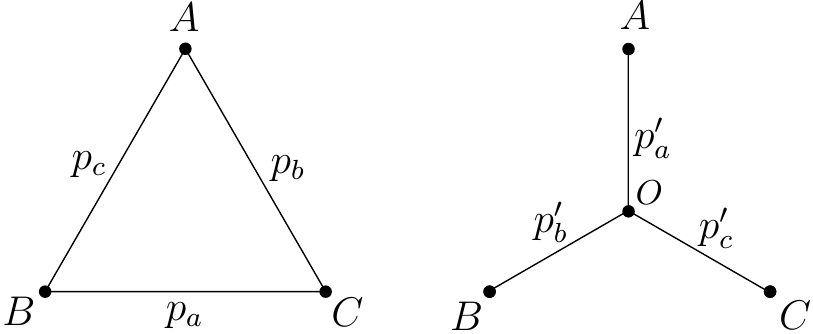}
  \caption{Triangle and star graphs with parameters indicated on edges.}
  \label{fig:simple_transformation}
\end{figure}

The two measures are related via the following relation. 

\begin{prop}[\Stt] \label{prop:stt}
  Fix a cluster weight $q \geq 1$ and suppose the following conditions hold:
  \begin{align}
    y_a y_b y_c + y_a y_b + y_b y_c + y_c y_a & = q,  \label{eq:1st_relation} \\
    y_i y_i' & = q, \qquad \forall i \in \{ a, b, c \}. \label{eq:2nd_relation}
  \end{align}
  Then, for any boundary conditions $\xi$, 
  the connections between the points $A,B,C$ inside the graphs $\triangle$ and $\starg$ have same law under $\trianglelaw$ and $\starlaw$, respectively. 
\end{prop}

\begin{rmk}
  In light of~\eqref{eq:2nd_relation}, the relation~\eqref{eq:1st_relation} is equivalent to
  \begin{equation} \label{eq:dual_relation}
    y_a' y_b' y_c' - q (y_a' + y_b' + y_c') = q^2.
  \end{equation}
\end{rmk}

The proof of the proposition is a straightforward computation of the probabilities of the different possible connections between $A$, $B$ and $C$ in the two graphs. 

\begin{proof}
  The probabilities of the different possible connections between $A$, $B$ and $C$ in $\triangle$ and $\starg$ with different boundary conditions are summarized in the following tables.
  For ease of notation, the probabilities are given up to a multiplicative constant; the multiplicative constant is the inverse of the sum of all the terms in each column. 
  Different tables correspond to different boundary conditions; each line to one connection event.
  We exclude symmetries of boundary conditions. 

  \smallskip 

  \noindent 
  \begin{center}
    \begin{tabular}{|c|cc|}
      \hline
      $\{ \{A, B\}, C\}$ & In $\triangle$ & In $\starg$  \\ 
      \hline
      all disconnected & $q$ & $q(q + y_a' + y_b' + y_c')$ \\
      $A \leftrightarrow B \nlra C$ & $y_c q$ & $y_a' y_b' q$ \\
      $B \leftrightarrow C \nlra A$ & $y_a$ & $y_b' y_c'$ \\
      $C \leftrightarrow A \nlra B$ & $y_b$ & $y_c' y_a'$  \\
      $A \leftrightarrow B \leftrightarrow C$ & $y_a y_b + y_b y_c + y_c y_a + y_a y_b y_c$ & $y_a' y_b' y_c'$ \\
      \hline
      \multicolumn{3}{c}{} \\
      \hline
      $\{ A, B, C\}$ & In $\triangle$ & In $\starg$  \\ 
      \hline
      all disconnected & $1$ & $q + y_a' + y_b' + y_c'$ \\
      $A \leftrightarrow B \nlra C$ & $y_c$ & $y_a' y_b'$ \\
      $B \leftrightarrow C \nlra A$ & $y_a$ & $y_b' y_c'$ \\
      $C \leftrightarrow A \nlra B$ & $y_b$ & $y_c' y_a'$  \\
      $A \leftrightarrow B \leftrightarrow C$ & $y_a y_b + y_b y_c + y_c y_a + y_a y_b y_c$ & $y_a' y_b' y_c'$ \\
      \hline
      \multicolumn{3}{c}{} \\
      \hline
      $\{ \{ A\}, \{ B\}, \{ C\} \}$ & In $\triangle$ & In $\starg$ \\ 
      \hline
      all disconnected & $q^2$ & $q^2 (y_a' + y_b' + y_c' + q)$ \\
      $A \leftrightarrow B \nlra C$ & $y_c q$ & $y_a' y_b' q$ \\
      $B \leftrightarrow C \nlra A$ & $y_a q$ & $y_b' y_c' q$ \\
      $C \leftrightarrow A \nlra B$ & $y_b q$ & $y_c' y_a' q$  \\
      $A \leftrightarrow B \leftrightarrow C$ & $y_a y_b + y_b y_c + y_c y_a + y_a y_b y_c$ & $y_a' y_b' y_c'$ \\
      \hline
    \end{tabular}
    \captionof{table}{Probabilities of different connection events with different boundary conditions.}
  \end{center}
  \smallskip

  It is straightforward to check that the corresponding entries in the two columns of each table are proportional, with ratio (right quantity divided by the left one) $q^2 / y_a y_b y_c$ each time.
\end{proof}

In light of Proposition~\ref{prop:stt}, the measures $\trianglelaw$ and $\starlaw$ may be coupled in a way that preserves connections. 
For the sake of future applications, we do this via two random maps $T$ and $S$ from $\{ 0, 1\}^{\smallstarg}$ to $\{ 0, 1 \}^{\triangle}$, and conversely. 
These random mappings are described in Figure~\ref{fig:simple_transformation_coupling};
when the initial configuration is such that the result is random, the choice of the resulting configuration is done independently of any other randomness. 

\begin{figure}[htb]
  \centering
  \includegraphics{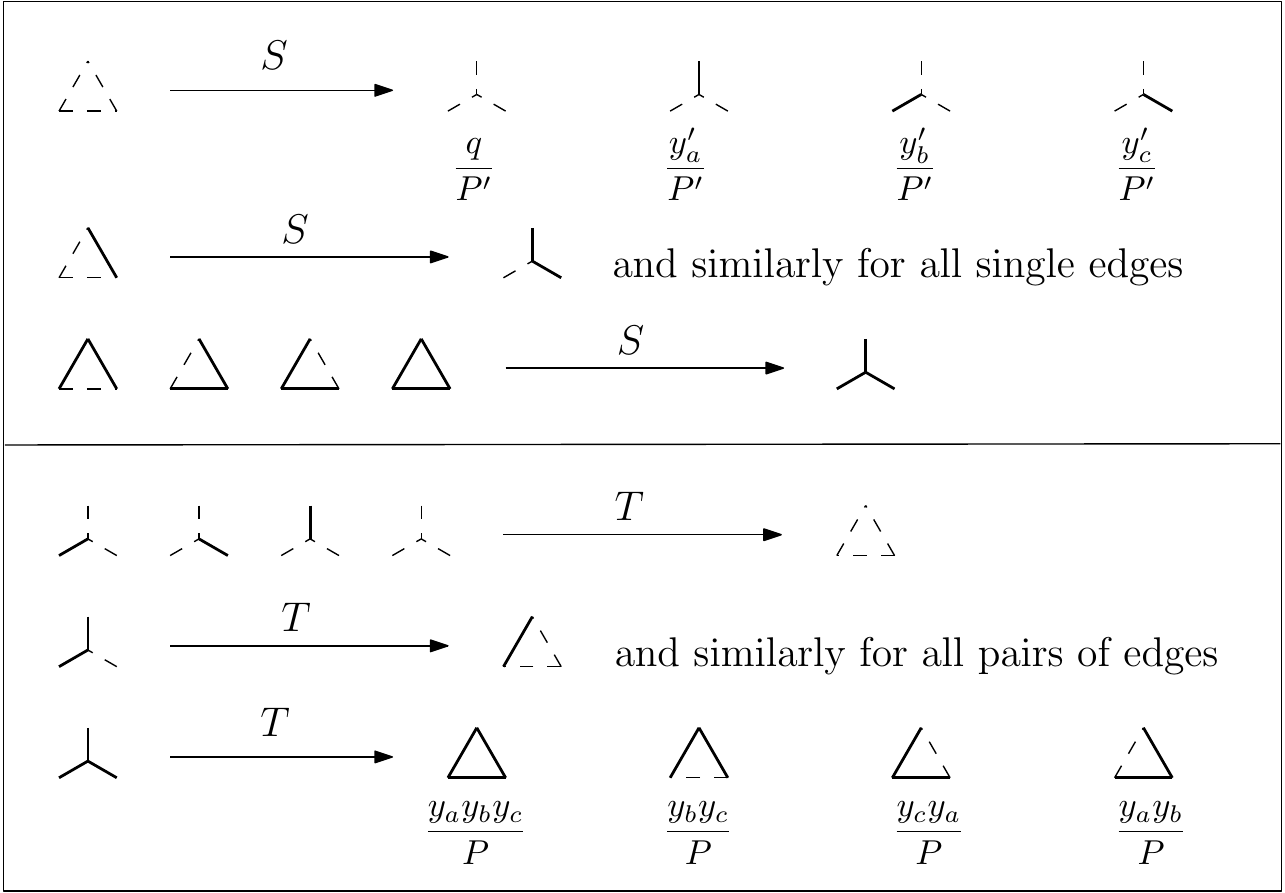}
  \caption{The random maps $T$ and $S$. Open edges are represented by thick segments, closed edges by dashed ones. 
    In the first and last lines, the outcome is random: it is chosen among four possibilities with probabilities indicated below. 
    The normalizing constants are $P' = q +y_a' + y_b' + y_c' = y_a' y_b' y_c' / q$ and $P = y_a y_b y_c + y_a y_b + y_b y_c + y_c y_a = q$.}
  \label{fig:simple_transformation_coupling}
\end{figure}

\begin{prop}[Star-triangle coupling] \label{prop:coupling}
  Fix $q \geq 1$, boundary conditions $\xi$ on $\{ A, B, C \}$ 
  and triplets $\boldsymbol{p} \in (0, 1)^3$ and $\boldsymbol{p'} \in (0, 1)^3$ 
  satisfying~\eqref{eq:1st_relation} and~\eqref{eq:2nd_relation}. 
  Let $\omega$ and $\omega'$ be configurations chosen according to $\trianglelaw$ and $\starlaw$, respectively. 
  Then,
  \begin{enumerate}
  \item $S(\omega)$ has the same law as $\omega'$,
  \item $T(\omega')$ has the same law as $\omega$,
  \item for $x, y \in \{ A, B, C \}$, $x \xleftrightarrow{\triangle,\; \omega} y$ 
    if and only if $x \xleftrightarrow{\smallstarg,\; S(\omega)} y$,
  \item for $x, y \in \{ A, B, C \}$, $x \xleftrightarrow{\smallstarg,\; \omega'} y$ 
    if and only if $x \xleftrightarrow{\triangle,\; T(\omega')} y$.
  \end{enumerate}
\end{prop}

\begin{proof}
  The points 3 and 4 are trivial by Figure~\ref{fig:simple_transformation_coupling}. 
  Points 1 and 2 follow by direct computation from the construction of $S$ and $T$, respectively, with the crucial remark that the randomness in $S$ and $T$ is independent of that of $\omega$ and $\omega'$, respectively. 
\end{proof}

\subsection{\Stt on isoradial graph} \label{sec:stt_rc_iso}


Next, we study the \stt for isoradial graphs.
We will see that when \stts are applied to isoradial graphs with the \rcm given by isoradiality when $\beta = 1$, what we get is exactly the \rcm on the resulting graph.

\begin{prop} \label{prop:stt_iso_preserved}
  Fix $q \geq 1$ and $\beta = 1$.
  Then, the random-cluster model is preserved under \stts in the following sense.
  \begin{itemize}
  \item For any triangle $ABC$ contained in an isoradial graph, the parameters $y_{AB}$, $y_{BC}$ and $y_{CA}$
    associated by~\eqref{eq:parameters} with the edges $AB$, $BC$ and $CA$, respectively, satisfy~\eqref{eq:1st_relation}. 
    Moreover, there exists a unique choice of point $O$ such that, if the triangle $ABC$ is replaced by the star $ABCO$, 
    the resulting graph is isoradial and the parameters associated with the edges $CO$, $AO$, $BO$ by~\eqref{eq:parameters}
    are related to $y_{AB}$, $y_{BC}$ and $y_{CA}$ as in~\eqref{eq:2nd_relation}.
  \item For any star $ABCO$ contained in an isoradial graph, the parameters $y_{OC}$, $y_{OA}$ and $y_{OB}$
    associated by~\eqref{eq:parameters} with the edges $CO$, $AO$ and $BO$, respectively, satisfy~\eqref{eq:1st_relation}. 
    Moreover, if the star $ABCO$ is replaced by the triangle $ABC$, 
    the resulting graph is isoradial and the parameters associated with the edges $AB$, $BC$, $CA$  by~\eqref{eq:parameters}
    are related to $y_{OC}$, $y_{OA}$ and $y_{OB}$ as in~\eqref{eq:2nd_relation}.
  \end{itemize}
\end{prop}

\begin{proof}
  We only give the proof of the first point; the second may be obtained by considering the dual graph. 
  Let $ABC$ be a triangle contained in an isoradial graph $G$.
  Write $a,b,c$ for the angles subtended to the edges $BC$, $AC$, $AB$, respectively. 
  Then, $a+b+c = 2\pi$. 
  A straightforward trigonometric computation shows that then 
  $y_a y_b y_c + y_a y_b + y_b y_c + y_c y_a - q = 0$.
  
  Permute the three rhombi of $G^\diamond$ corresponding to the edges $AB$, $BC$, $CA$ as described in Figure~\ref{fig:local_relation}
  and let $O$ be their common point after permutation. 
  Let $\tilde G$ be the graph obtained from $G$ by adding the vertex $O$ and connecting it to $A,B$ and $C$ 
  and removing the edges $AB$, $BC$, $CA$.
  Since $\tilde G$ has a diamond graph (as depicted in Figure~\ref{fig:local_relation}), it is isoradial. 
  Moreover, the angles subtended by the edges $OA$, $OB$ and $OC$ are $\pi-a$, $\pi-b$ and $\pi-c$, respectively.
  It follows from~\eqref{eq:parameters} that the parameters of the edges $OA$, $OB$ and $OC$ 
  are related to those of the edges  $AB$, $BC$ and $CA$ by~\eqref{eq:2nd_relation}.
\end{proof}

\begin{figure}[htb]
  \centering
  \includegraphics[scale=1.1]{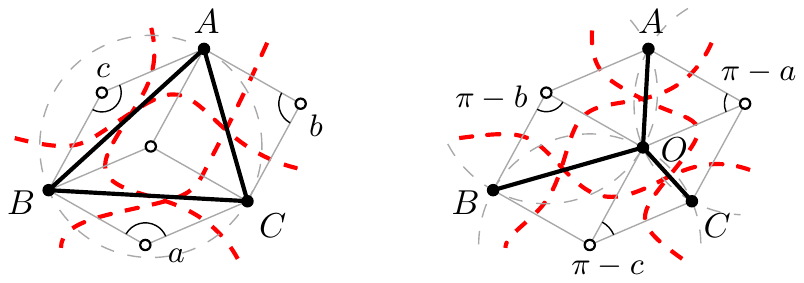}
  \caption{A local triangle subgraph with corresponding subtended angles $a$, $b$ and $c$. Note that $a + b + c = 2 \pi$.
    The order of crossing of the three tracks involved is changed.}
  \label{fig:local_relation}\label{fig:transfo_iso_tracks}
\end{figure}

Triangles and stars of isoradial graphs correspond to hexagons formed of three rhombi in the diamond graph. 
Thus, when three such rhombi are encountered in a diamond graph, 
they may be permuted as in Figure~\ref{fig:local_relation} using a \stt. 
We will call the three rhombi the {\em support} of the \stt. 

Let $\omega$ be a configuration on some isoradial graph $G$ and $\sigma$ a \stt that may be applied to $G$. 
When applying $\sigma$ to $G$, the coupling of Proposition~\ref{prop:coupling} yields a configuration that we will denote by $\sigma(\omega)$. 

Consider an open path $\gamma$ in $\omega$. 
Then, define $\sigma(\gamma)$ the image of $\gamma$ under $\sigma$ to be the open path of $\sigma(\omega)$ described as follows.
\begin{itemize}
\item If an endpoint of $\gamma$ is adjacent to the support of $\sigma$, then we set $\sigma(\gamma)$ to be $\gamma$ plus the additional possibly open edge if the latter has an endpoint on $\gamma$, which is given by the first line of Figure~\ref{fig:simple_transformation_coupling}.
\item If $\gamma$ does not cross (and is not adjacent to) the support of $\sigma$, we set $\sigma(\gamma) = \gamma$.
\item Otherwise, $\gamma$ intersects the support of $\sigma$ in one of the ways depicted in the first two lines of Figure~\ref{fig:path_transformation1}.
  Then, we set $\sigma(\gamma)$ to be identical to $\gamma$ outside the support of $\sigma$. 
  And in the support of \stt, since $\sigma$ preserves connections, the part of $\gamma$ inside may be replaced by an open path as in the same figure. 
  Notice the special case when $\gamma$ ends in the centre of a star and the corresponding edge is lost when applying $\sigma$ (third line of Figure~\ref{fig:path_transformation1}).
\end{itemize}

\begin{figure}[htb]
  \centering
  \includegraphics[scale=1]{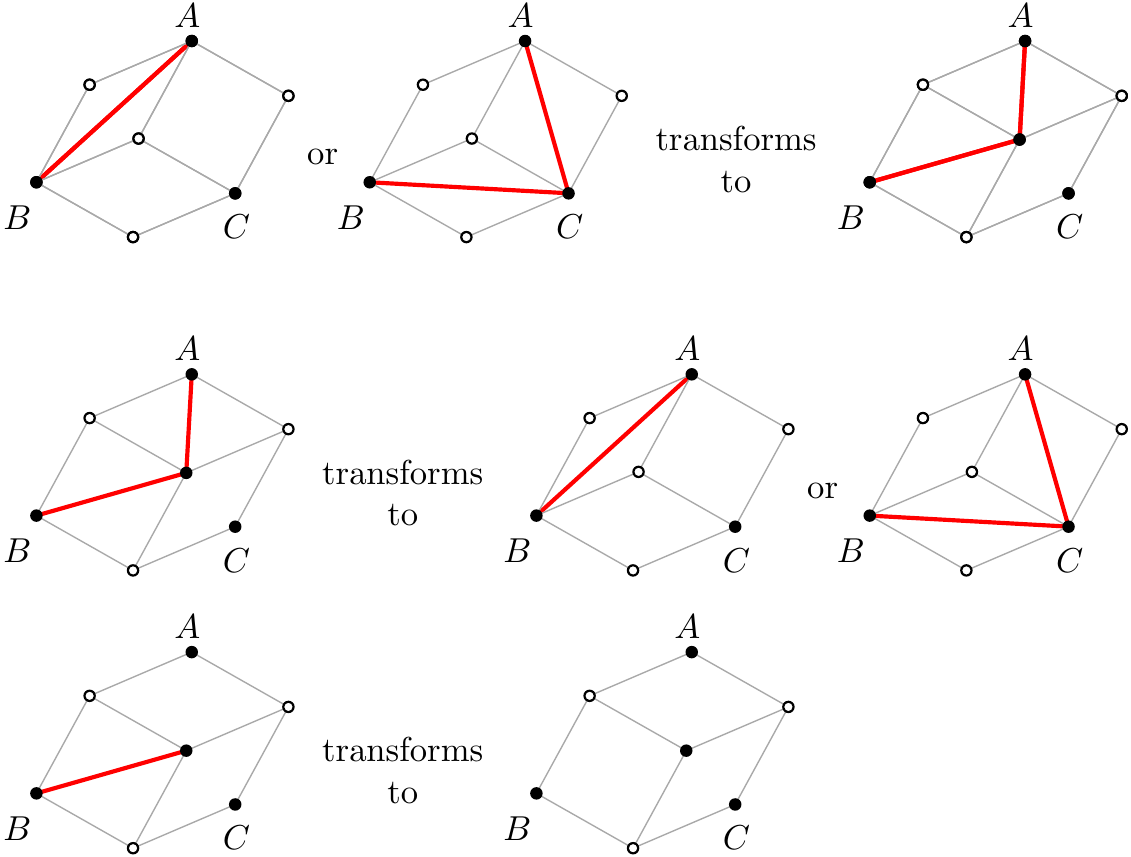}
  \caption{The effect of a \stt on an open path. In the second line, the second outcome is chosen only if the edge $AB$ is closed. 
    The fact that the result is always an open path is guaranteed by the coupling that preserves connections. 
    In the last case, the open path may loose one edge.}
  \label{fig:path_transformation1}
\end{figure}

\subsection{Details on isoradial graphs: train tracks and bounded-angles property}

Let $\bbG$ be an isoradial graph.
Recall that $\bbG^\diamond$ is the diamond graph associated with $\bbG$, whose faces are rhombi. 
Each edge $e$ of $\bbG$ corresponds to a face of $\bbG^\diamond$, and the angle $\theta_e$ associated to $e$ is one of the two angles of that face.
We say that $\bbG$ satisfies the {\em bounded-angles property} with parameter $\eps>0$ if all the angles $\theta_e$ of edges of $e \in \bbE$ are contained in $[\eps, \pi - \eps]$. Equivalently, edges of $\bbG$ have parameter $p_e$ bounded away from $0$ and $1$ uniformly. 
The property also implies that the graph distance on $\bbG^\diamond$ or $\bbG$ and the euclidean distance are quasi-isometric. 

Write $\calG(\eps)$ for the set of double-periodic isoradial graphs satisfying the bounded-angles property with parameter $\eps > 0$.

Define a \emph{train track} as a double-infinite sequence of faces $(r_i)_{i \in \bbZ}$ of $\bbG^\diamond$ 
such that the intersections $(r_i \cap r_{i+1})_{i \in \bbZ}$ are non-empty, distinct and parallel segments (Figure~\ref{fig:train_tracks}).

A train track as above may also be viewed as an arc in $\bbR^2$ which connects the midpoints of the edges $(r_i \cap r_{i+1})_{i \in \bbZ}$.
These edges are called the \emph{transverse segments} of the track, and the angle they form with the horizontal line is called the \emph{transverse angle} of the track. 

Write $\calT(\bbG)$ for the set of train tracks of $\bbG$.
Notice that $\calT(\bbG) = \calT(\bbG^*)$ since the diamond graph is the same for the primal and dual graphs.
Most commonly, $\calT(\bbG)$ is regarded only up to homeomorphism. 
Then, it only encodes the structure of $\bbG$; the embedding of $\bbG$ may be recovered from the values of the transverse angles of the tracks. 

\begin{figure}[htb]
  \centering
  \includegraphics[scale=0.7, page=3]{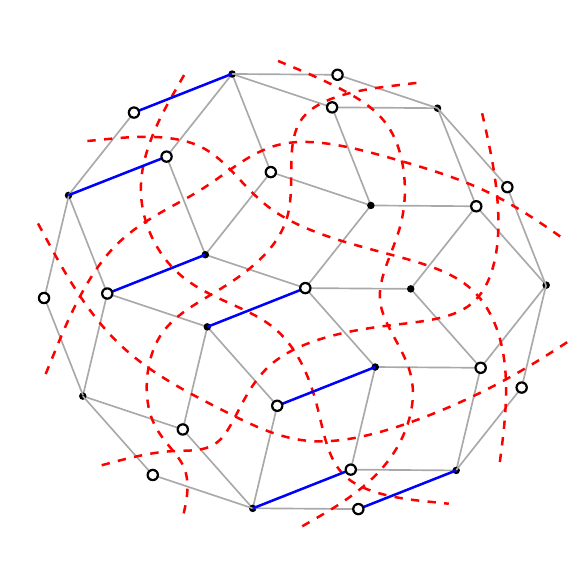}
  \caption{The train track representation (in dotted red lines) of the isoradial graph in Figure~\ref{fig:isoradial}.
    Transverse edges of a track are drawn in blue.}
  \label{fig:train_tracks}
\end{figure}

One can easily check that the rhombi forming a track are all distinct, thus a track does not intersect itself.
Furthermore, two distinct tracks can only have at most one intersection.
A converse theorem has been shown by Kenyon and Schlenker~\cite{KS-quad-graphs}.

Each face of $\bbG^\diamond$ corresponds to an intersection of two train tracks. 
A hexagon in $\bbG^\diamond$ (that is a star or triangle in $\bbG$) corresponds to the intersection of three train tracks, as in Figure~\ref{fig:transfo_iso_tracks}.
The effect of a \stt is to locally permute the three train tracks involved in the hexagon 
by ``pushing'' one track over the intersection of the other two.


\subsection{Switching between isoradial graphs} \label{sec:switching_isos}

As explained in the introduction, the strategy of the proof is to transform the regularly embedded square lattice into arbitrary doubly-periodic isoradial graphs using \stts.
This will enable us to transfer estimates on connection probabilities from the former to the latter. 
Below, we explain the several steps of the transformation. 

\subsubsection{From regular square lattice to isoradial square lattice} \label{sec:mix}

In this section we will consider isoradial embeddings of the square lattice.
As described in \cite{GriMan14}, a procedure based on \emph{track exchanges} transforms one isoradial embedding of the square lattice into a different one. In addition to \cite{GriMan14}, the effect of boundary conditions needs to be taken into account; a construction called {\em convexification} is therefore required.

Isoradial embeddings of the square lattice may be encoded by two doubly-infinite sequences of angles. Let $\ba = (\alpha_n)_{n \in \bbZ}$ and $\bb = (\beta_n)_{n \in \bbZ}$ be two sequences of angles in $[0,\pi)$ such that
\begin{align} \label{eq:bounded_angles}
  \sup\{\alpha_n : n \in \bbZ\} & < \inf \{\beta_n : n \in \bbZ\}, \nonumber \\
  \inf\{\alpha_n : n \in \bbZ\} & > \sup \{\beta_n : n \in \bbZ\} - \pi.
\end{align}
Then, define $\bbG_{\ba, \bb}$ to be the isoradial embedding of the square lattice with vertical train tracks $(s_{n})_{n \in \bbZ}$ with transverse angles $(\alpha_n)_{n \in \bbZ}$ and horizontal train tracks $(t_{n})_{n \in \bbZ}$ with transverse angles $(\beta_n)_{n \in \bbZ}$.
Condition \eqref{eq:bounded_angles} ensures that $\bbG_{\ba, \bb}$ satisfies the bounded-angles property for $\eps = \inf \{ \beta_n - \alpha_m, \alpha_n - \beta_m + \pi : m, n \in \bbZ \} > 0$.

\begin{figure}[htb]
  \centering
  \includegraphics[scale=1.1]{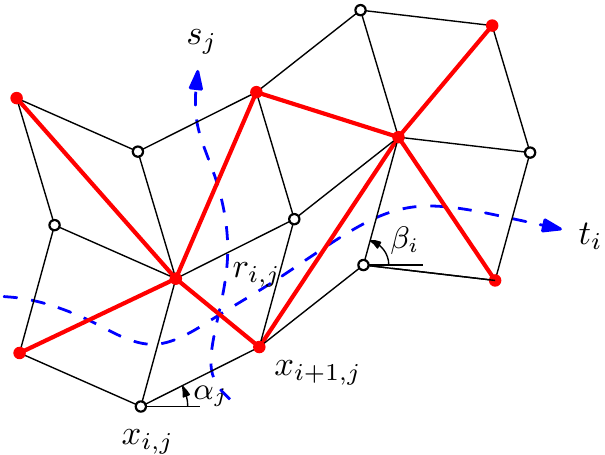}
  \caption{A piece of an isoradial embedding of a square lattice (the square lattice in red, the diamond graph in black and the tracks in blue).}
  \label{fig:isoradial_square}
\end{figure}

\medbreak

In the following, we mainly consider doubly-periodic isoradial graphs, 
hence periodic sequences $(\alpha_n)_{n \in \bbZ}$ and $(\beta_n)_{n \in \bbZ}$.
The bounded-angles property is then automatically ensured if it is satisfied for a period of $(\alpha_n)$ and $(\beta_n)$. 

The same notation may be used to denote ``rectangular'' finite subgraphs of isoradial square lattices.
Indeed, for finite sequences $\ba = (\alpha_n)_{M_- \leq n \leq M_+}$ and $\bb = (\beta_n)_{N_- \leq n \leq N_+}$,
define $G_{\ba,\bb}$ to be a (finite) isoradial square lattice with $M_+ - M_- +1$ vertical tracks and $N_+ - N_- +1$ horizontal tracks.
We will think of this graph as part of an infinite isoradial graph, thus we call the 
right boundary of $G_{\ba,\bb}$ the vertices to the right of $s_{M_+}$, 
the left boundary those to the left of $s_{M_-}$,
the top boundary the vertices above $t_{N_+}$ and the bottom boundary those below $t_{N_-}$.
The term rectangular refers to the diamond graph rather than to $G_{\ba,\bb}$; 
the boundary denominations are also used for $G_{\ba,\bb}^\diamond$. 

The {\em regular square lattice} is the embedding corresponding to sequences $\beta_n = \frac{\pi}2$ and $\alpha_n = 0$ for all $n \in \bbZ$.

\subtitle{Track exchange}
Let us start by describing a simple but essential operation composed of \stts, which we call \emph{track exchange}.
In the language of transfer matrices, this amounts to that the transfer matrices associated with two adjacent rows commute with each other, which is the usual formulation of the Yang-Baxter transformation.

Let $G$ be a finite rectangular subgraph of an isoradial square lattice and $t$ and $t'$ be two parallel adjacent horizontal train tracks. 
Suppose that we want to switch their positions using \stts.
That is, we would like to perform a series of \stts that changes the graph $G$ into an identical graph, 
with the exception of the train tracks $t$ and $t'$ that are exchanged (or equivalently that their transverse angles are exchanged). 
We will suppose here that the transverse angles of $t$ and $t'$ are distinct, otherwise the operation is trivial. 

Since $t$ and $t'$ do not intersect, no \stt may be applied to them.
Suppose however that $G$ contains one additional rhombus (gray in Figure~\ref{fig:track_exchange}) at either the left or right end of $t$ and $t'$ that corresponds to the intersection of these two tracks. (Depending on the transverse angles of the tracks, there is only one possible position for this rhombus.) 
Then, a series of \stts may be performed as in Figure~\ref{fig:track_exchange}.
In effect, these transformations ``slide'' the gray rhombus from one end of the tracks to the other, and exchange the two tracks in the process. 

\begin{figure}[htb]
  \centering
  \includegraphics[scale=0.75]{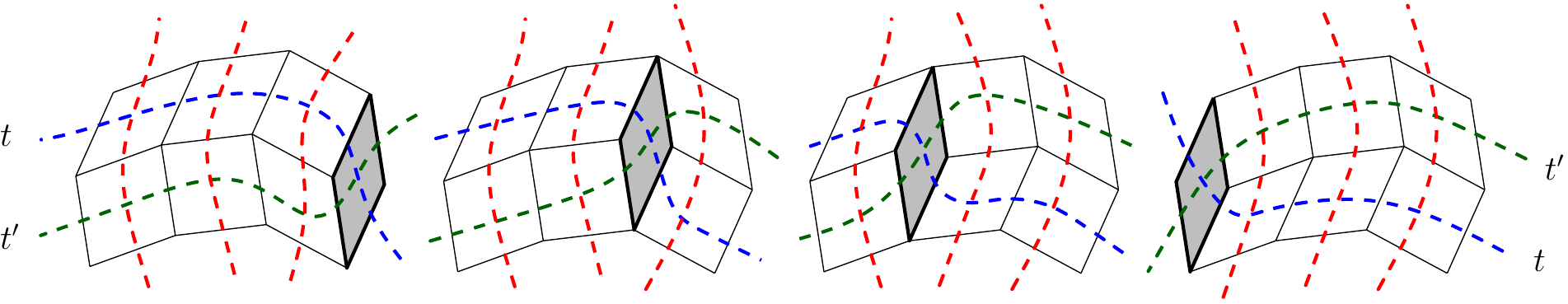}
  \caption{We move the gray rhombus from the right to the left by a sequence of \stt. Observe that these transformations only affect the tracks $t$ and $t'$, and that their ultimate effect is to exchange them.}
  \label{fig:track_exchange}
\end{figure}

As seen in Section~\ref{sec:stt_rc_iso}, each \stt of an isoradial graph preserves the \rcm and connection properties.
Thus, the procedure above, which we call a \emph{track exchange}, allows us to deduce connection properties of the resulting graph from those of the initial graph. 

In~\cite{GriMan14}, the gray rhombus was added before exchangin the tracks and removed afterwards. 
Thus, the track exchange could be perceived as a measure- and connection-preserving transformation between isoradial square lattices.
By repeating such track exchanges, blocks of tracks of a square lattice were exchanged, and RSW-type estimates were transported from one block to another.  

In the present context, adding a rhombus (and hence an edge) to a graph affects the \rcm of the entire graph. We therefore prefer to ``prepare'' the graph by adding all necessary gray rhombi for all the track exchanges to be performed at once. 
The operation is called the \emph{convexification} of a finite part of a square lattice. 

\subtitle{Convexification}
Consider a finite rectangular portion $G = G_{\ba,\bb}$ of an isoradial square lattice, with $\ba$ and $\bb$ two finite sequences of angles.
Suppose that $\bb = (\beta_n)_{0 \leq n \leq N}$ for some $N > 0$. 
We call the vertices below $t_0$ (in the present case the bottom boundary) the \emph{base level} of $G$. 

We say that $\tilde{G}$ is a \emph{convexification} of $G = G_{\ba, \bb}$ if
\begin{itemize}
\item $G$ is a subgraph of $\tilde{G}$ and $\tilde{G}$ has no other tracks than those of $G$;
\item the top and bottom boundary of $G^\diamond$ are also boundaries of $\tilde{G}^\diamond$;
\item as we follow the boundary of $\tilde{G}^\diamond$ in counterclockwise direction, the segment between the top and bottom boundaries (which we naturally call the left boundary) and that between the bottom and top boundaries (called the right boundary) are convex.
\end{itemize}

The second condition may be read as follows: in $\tilde G$, the vertical tracks $(s_n)$ only intersect the horizontal tracks $(t_n)$; however, additionally to $G$, $\tilde G$ may contain intersections between horizontal tracks. 

The third condition is equivalent to asking that all horizontal tracks of $G$ with distinct transverse angles intersect in $\tilde{G}$. Indeed, the left and right boundaries of $\tilde G^\diamond$ are formed of the transverse segments of the horizontal tracks of $G$, each track contributing once to each segment of the boundary. That both the left and right boundaries of  $\tilde G^\diamond$ are convex means that the transverse segments of two tracks $t_i$, $t_j$ with distinct transverse angles appear in alternative order along the boundary of $\tilde G^\diamond$, when oriented in counterclockwise direction. Hence, they necessarily intersect in $\tilde G$. The converse may also be easily checked.

\begin{figure}[htb]
  \centering
  \includegraphics[scale=0.85, page=1]{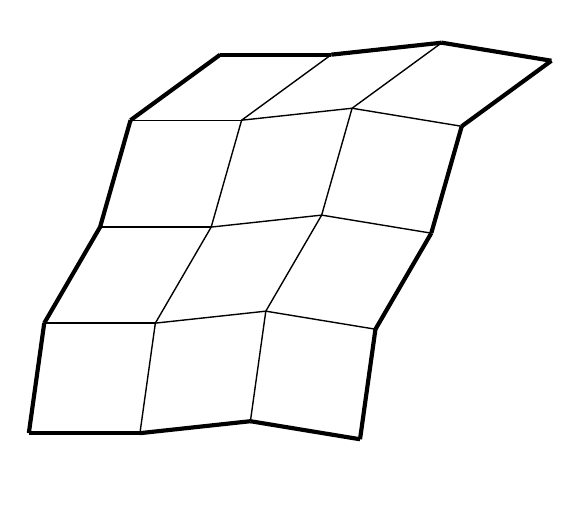}
  \hspace{0.5cm}
  \includegraphics[scale=0.85, page=2]{images/convexification.pdf}
  \caption{An isoradial square lattice and a convexification of it. Only the diamond graph is depicted.}
  \label{fig:convexification}
\end{figure}

Below, we will sometimes call $G$ the \emph{square lattice block} of $\tilde G$; $\tilde G \setminus G$ is naturally split into a left and a right part. 

The following two simple lemmas will come in useful when performing track exchanges. 

\begin{lem} \label{lem:convexification}
  For any adjacent horizontal tracks $t, t'$ of $G$ with distinct transverse angles, 
  there exists a convexification $\tilde G$ of $G$ 
  in which the rhombus corresponding to the intersection of $t$ and $t'$ is adjacent to $G^\diamond$. 
\end{lem}

\begin{lem} \label{lem:convexification2}
  For any two convexifications $\tilde G$ and $\tilde G'$ of $G$, there exists a sequence of \stts that transforms $\tilde G$ into $\tilde G'$ and that does not affect any rhombus of $G^\diamond$. 
\end{lem}

\begin{proof}[Proof of Lemma~\ref{lem:convexification}]
  We start by describing an algorithm that constructs a convexification of $G$. 
  Let $\langle , \rangle$ be the scalar product on $\bbR^2$.
  \begin{enumerate}
  \item Set $H = G$, which is the graph to be convexified.
  \item Orient the edges on the right boundary of $H^\diamond$ above the base level upwards and denote the corresponding unit vectors by $\bdryedge{0}, \dots, \bdryedge{N}$.
  \item If there exists $j$ such that $\langle\bdryedge{j+1} - \bdryedge{j},(1,0)  \rangle > 0$, fix such a value $j$ and proceed to Step 4. Otherwise, go to the Step 5.
  \item Add a rhombus to $H^\diamond$ whose boundary is given by $\bdryedge{j}, \bdryedge{j+1}, -\bdryedge{j}$ and $-\bdryedge{j+1}$ to the right of the edges $ \bdryedge{j}$, $\bdryedge{j+1}$. Set $H$ to be the graph thus obtained, and go back to Step~2. 
  \item Proceed the same for the left boundary of $G$. 
  \end{enumerate}
  Each rhombus added in Step 4 corresponds to an intersection of two  horizontal tracks of $G$. As such, only a finite number of such rhombi may be added, which shows that the algorithm necessarily terminates. Moreover, it is obvious to see that when it terminates, the resulting graph, which we denote by $\tilde G$, is indeed a convexification of $G$.

  The construction of $\tilde G$ does not ensure that the successive tracks $t$ and $t'$ intersect in $\tilde G$ adjacently to $G$.
  However, 
  we may choose $j$ corresponding to the index of $t$ the first time the algorithm arrives at Step 3 for either the right or left boundary.
  If such choice is made, the intersection of the tracks $t$ and $t'$ in the resulting graph $\tilde G$ will be adjacent to $G$. 
\end{proof}

\begin{proof}[Proof of Lemma~\ref{lem:convexification2}]
  By symmetry, it is sufficient to show that there exists a sequence of \stts that transforms the right part (call it $G_r$) of $\tilde G \backslash G$ into the right part of $\tilde G' \backslash G$ (which we call $G_r'$) without affecting any rhombus of $G^\diamond$.
  Notice that $G_r^\diamond$ and $(G'_r)^\diamond$ have the same boundary.
  Indeed, the left boundaries of $G_r^\diamond$ and $(G'_r)^\diamond$ coincide both with the right boundary of $G^\diamond$. 
  Their right boundaries are both formed of the segments of length $1$, with angles $\bb$, arranged in increasing order.
  Then,~\cite[Thm.~5]{Kenyon-tiling} ensures the existence of the transformations as required. 
\end{proof}

\medbreak

Consider a finite rectangular region $G$ of an isoradial square lattice and consider any of its convexification $\tilde G$.
Using the previous two lemmas, one can switch the transverse angles of any two neighbouring horizontal train tracks by a sequence of \stts.
A more precise statement is given below.

\begin{cor}
  Let $G = G_{\ba, \bb}$ be as above and let $t$ and $t'$ be two adjacent horizontal train tracks with distinct transverse angles. 
  Then, for any convexification $\tilde G$ of $G$, there exists a sequence of \stts $\sigma_1, \dots, \sigma_k$ that may be applied to $\tilde G$ with the following properties: 
  \begin{itemize}
  \item  there exists $0 \leq \ell < k$ such that the transformations $\sigma_1, \dots, \sigma_\ell$ only affect either the right or the left side of $\tilde G \setminus G$;
  \item in $(\sigma_\ell \circ \dots \circ \sigma_1) (\tilde G)$, the tracks $t$ and $t'$ intersect at a rhombus adjacent to $G$;
  \item the transformations $\sigma_{\ell + 1}, \dots, \sigma_k$ applied to $(\sigma_\ell \circ \dots \circ \sigma_1) (\tilde G)$ are ''sliding'' the intersection of $t$ and $t'$ from one side of $G$ to the other, as described in Figure~\ref{fig:track_exchange}.
  \end{itemize} 
  Write $\Sigma_{t, t'} = \sigma_k \circ \dots \circ \sigma_1$. 
  If $\tau$ denotes the transposition of the indices of tracks $t$ and $t'$, 
  then $\Sigma_{t, t'}(G)$ is a convexification of $G_{\ba, \tau\bb}$.
\end{cor}

\begin{proof}
  Suppose for simplicity that the tracks $t$ and $t'$ intersect in $\tilde G$ to the right of $G$ 
  (which is to say that the transverse angle of the lower track is greater than that of the above).
  
  Write $\tilde G'$ for a convexification of $G$ 
  in which the tracks $t$, $t'$ intersect in a rhombus adjacent to $G^\diamond$ (as given by Lemma~\ref{lem:convexification}). 
  It is obvious that the left side of $\tilde G'$ may be chosen identical to that of $\tilde G$, and we will work under this assumption. 
  
  Let $\sigma_1,\dots,\sigma_\ell$ be a sequence of \stts as that given by Lemma~\ref{lem:convexification2} that affects only the right side of $\tilde G$ and that transforms $\tilde G$ into $\tilde G'$. Let $\sigma_{\ell+1},\dots, \sigma_k$ be the series of \stts that slides the intersection of $t$ and $t'$ from right to left of $G$, as in Figure~\ref{fig:track_exchange}.
  Then, $\sigma_1,\dots,\sigma_k$ obviously satisfies the conditions above. 
\end{proof}

In the following, we will apply repeated line exchanges $\Sigma_{t_{i},t_{j}}$ to a convexification $\tilde G$ of some finite portion of a square lattice.
Thus, we will implicitly assume $\Sigma_{t_{i},t_{j}}$ is a series of \stts as in the lemma above, adapted to the convexification to which it is applied. 
When $t_i$ and $t_{j}$ have same transverse angles, we will simply write $\Sigma_{t_{i},t_{j}}$ for the empty sequence of transformations. 
We note that tracks are indexed with respect to the starting graph and are not reindexed when track exchanges are applied.
This is the reason why neighboring tracks do not necessarily have indices which differ by 1; thus, we call them $t_i$ and $t_j$ with the only constraint $i \neq j$.

All of the above may be summarised as follows.
A convexification of $G$ provides all the horizontal track intersections necessary to exchange any two horizontal tracks (that is the gray rhombus in Figure~\ref{fig:track_exchange} for any pair of horizontal tracks).
In order to exchange two adjacent horizontal tracks $t_i$ and $t_{j}$, the sequence of transformations $\Sigma_{t_i, t_j}$ starts from bringing the intersection of $t_i$ and $t_j$ next to $G$ (this is done through \stts that do not affect $G$), then slides it through $t_i$ and $t_j$. 

In certain arguments below, it will be more convenient to work with a ``double'' strip of square lattice $G = G_{\ba, \bb}$ where $\ba$ and $\bb$ are finite sequences of angles and $\bb = (\beta_n)_{-N \leq n \leq N}$ for some $N > 0$. 
We will then separately convexify the upper half $G_{\ba,(\beta_0,\dots, \beta_N)}$ and $G_{\ba,(\beta_{-N},\dots, \beta_{-1})}$ (as in Figure~\ref{fig:gmix}).
Track exchanges will only be between tracks above $t_0$ or strictly below $t_0$; 
the base (that is the vertices between $t_{-1}$ and $t_0$) will never be affected by track exchanges.

\subtitle{Construction of the mixed graph by gluing}
Consider two isoradial square lattices with same sequence $\ba$ of transverse angles for the vertical tracks. 
Write $\graph{1} = G_{\ba, \bb^{(1)}}$ and $\graph{2} = G_{\ba, \bb^{(2)}}$.
Additionally, suppose that they both belong to $\calG(\eps)$ for some $\eps > 0$.

Fix integers $N_1,N_2,M \in \bbN$. 
We create an auxiliary graph $G_\mix$, called the \emph{mixed graph}, 
by superimposing strips of $\graph{1}$ and $\graph{2}$ of width $2M+1$, then convexifying the result.
More precisely, let $\tilde \bb = (\beta^{(1)}_0, \dots, \beta^{(1)}_{N_1}, \beta^{(2)}_{0},\dots, \beta^{(2)}_{N_2})$ and $\tilde \ba = (\alpha_n)_{-M \leq n \leq M}$. 
Define $G_\mix$ to be a convexification of $G_{\tilde \ba,\tilde \bb}$.

Write $G^{(1)} = G_{\tilde \ba, \tilde \bb^{(1)}}$ and $G^{(2)} = G_{\tilde \ba, \tilde \bb^{(2)}}$, 
where 
$$
\tilde \bb^{(1)} = (\beta^{(1)}_0, \dots, \beta^{(1)}_{N_1})
\qquad \text{ and } \qquad
\tilde \bb^{(2)} = (\beta^{(2)}_0, \dots, \beta^{(2)}_{N_2}).
$$
These are both subgraphs of $G_\mix$; we call them the blocks of $\graph{1}$ and $\graph{2}$ inside $G_\mix$.

Next, we aim to switch these two blocks of $G_\mix$ using \stts.
That is, we aim to transform $G_\mix$ into a graph $G_\mix'$ obtained as above, with the sequence $\tilde \bb$ replaced by 
$(\beta^{(2)}_0, \dots, \beta^{(2)}_{N_2}, \beta^{(1)}_{0},\dots, \beta^{(1)}_{N_1})$.
There are two ways of doing this, each having its own advantages. 

One way is to use track exchanges to send the tracks  $t_{N_{1}+1},\dots, t_{N_1+ N_2+1}$ of $G_\mix$ all the way down, one by one. 
Using the notation of the previous section, this corresponds to the following sequence of track exchanges
\begin{align*}
  \Sigma^\downarrow
  =
  \Sigma^\downarrow_{N_1 + N_2 + 1} \circ \cdots \circ \Sigma^\downarrow_{N_1+1},
\end{align*}
where $\Sigma^\downarrow_k = \Sigma_{t_0, t_k} \circ \cdots \circ \Sigma_{t_{N_1},t_k}$
is a sequence of \stts sending the track $t_k$ to the bottom of the block $G^{(1)}$ in $G_\mix$.
This will be useful in the proof of Proposition~\ref{prop:horizontal_transport}, where we need to control the upward drift of an open path.

The other is to push the tracks $t_{N_{1}},\dots, t_{0}$ all the way up, one by one. 
It formally reads
\begin{align*}
  \Sigma^\uparrow 
  =
  \Sigma^\uparrow_0 \circ \cdots \circ \Sigma^\uparrow_{N_1},
\end{align*}
where $\Sigma^\uparrow_k = \Sigma_{t_k, t_{N_1 + N_2 +1}} \circ \cdots \circ \Sigma_{t_k, t_{N_1+1}}$
is a sequence of \stts sending the track $t_k$ to the top of the block $G^{(2)}$ in $G_\mix$.
This will be used to study the downward drift of an open path in Proposition~\ref{prop:vertical_transport}.

One may easily check that the sequences $\Sigma^\downarrow$ and $\Sigma^\uparrow$ may be applied to $G_\mix$.
That is that whenever a track exchange $\Sigma_{t,t'}$ is applied, 
the previous track exchanges are such that the tracks $t$ and $t'$ are adjacent. 
The two sequences of track exchanges are illustrated in Figure~\ref{fig:Sigma_up_down}.

\begin{figure}[htb]
  \centering
  \includegraphics[scale=0.75, page=1]{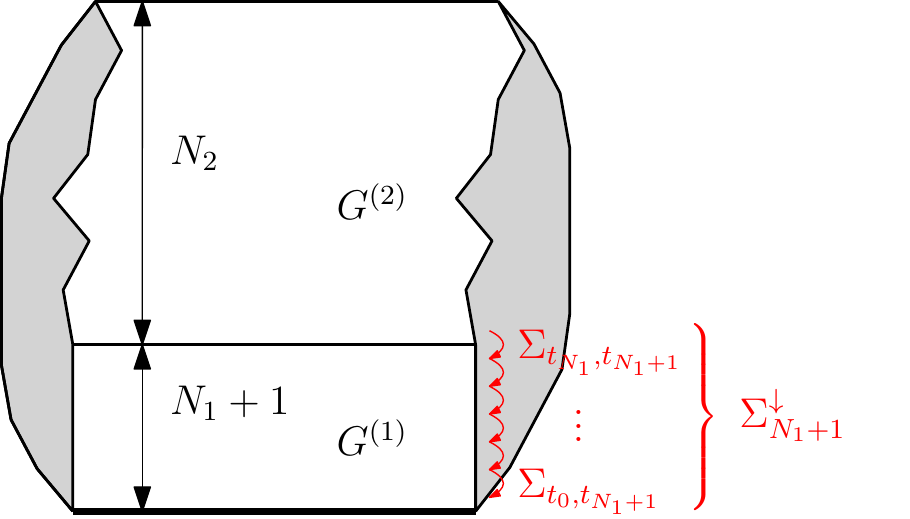}
  \includegraphics[scale=0.75, page=2]{images/Sigma_up_down.pdf}
  \caption{The graph $G_\mix$ is obtained by superimposing $G^{(1)}$ and $G^{(2)}$ then convexifying the result (in gray).
    \textbf{Left:} The sequence $\Sigma^\downarrow_{N_1+1}$ moves the track $t_{N_1+1}$ below the block $G^{(1)}$.
    \textbf{Right:} The sequence $\Sigma^\uparrow_{N_1}$ moves the track $t_{N_1}$ above the block $G^{(2)}$.}
  \label{fig:Sigma_up_down}
\end{figure}

The resulting graphs $\Sigma^\uparrow(G_\mix)$ and $\Sigma^\downarrow(G_\mix)$ 
both contain the desired block of isoradial square lattice, but their convexification may differ. 
However, by Lemma~\ref{lem:convexification2}, we may fix one convexification $G_\mix'$ of the resulting square lattice block and add \stts at the end of both $\Sigma^\uparrow$ and $\Sigma^\downarrow$ that only affect the convexification and such that 
$\Sigma^\uparrow(G_\mix) = \Sigma^\downarrow(G_\mix) = G_\mix'$.
Henceforth, we will always assume that both $\Sigma^\uparrow$ and $\Sigma^\downarrow$ contain these \stts. 

Since each \stt preserves the \rcm, we have
$$\Sigma^\uparrow \phi^\xi_{G_\mix} = \Sigma^\downarrow \phi^\xi_{G_\mix} = \phi^\xi_{G_\mix'}$$
for all boundary conditions $\xi$. 
Above, $\phi^\xi_{G_\mix}$ and $\phi^\xi_{G_\mix'}$ denote the \rcms with $\beta = 1$ 
and boundary conditions $\xi$ on $G_\mix$ and $G_\mix'$ respectively. 
The action of $\Sigma^\uparrow$ (and $\Sigma^\downarrow$) should be understood as follows. 
For a configuration $\omega$ chosen according to $\phi^\xi_{G_\mix}$, 
the sequence $\Sigma^\uparrow$ of \stts is applied to $\omega$
with the resulting configuration sampled as described in Figure~\ref{fig:simple_transformation_coupling}, 
independently for each \stt.
Then the final configuration follows $\phi^\xi_{G_\mix'}$. The same holds for~$\Sigma^\downarrow$.

The reader may note that we do not claim that $\Sigma^\uparrow (\omega)$ and $\Sigma^\downarrow(\omega)$ have the same law for any \emph{fixed} configuration $\omega$ on $G_\mix$;
this is actually not the case in general.

\begin{figure}[htb]
  \centering
  \includegraphics[scale=0.6, page=1]{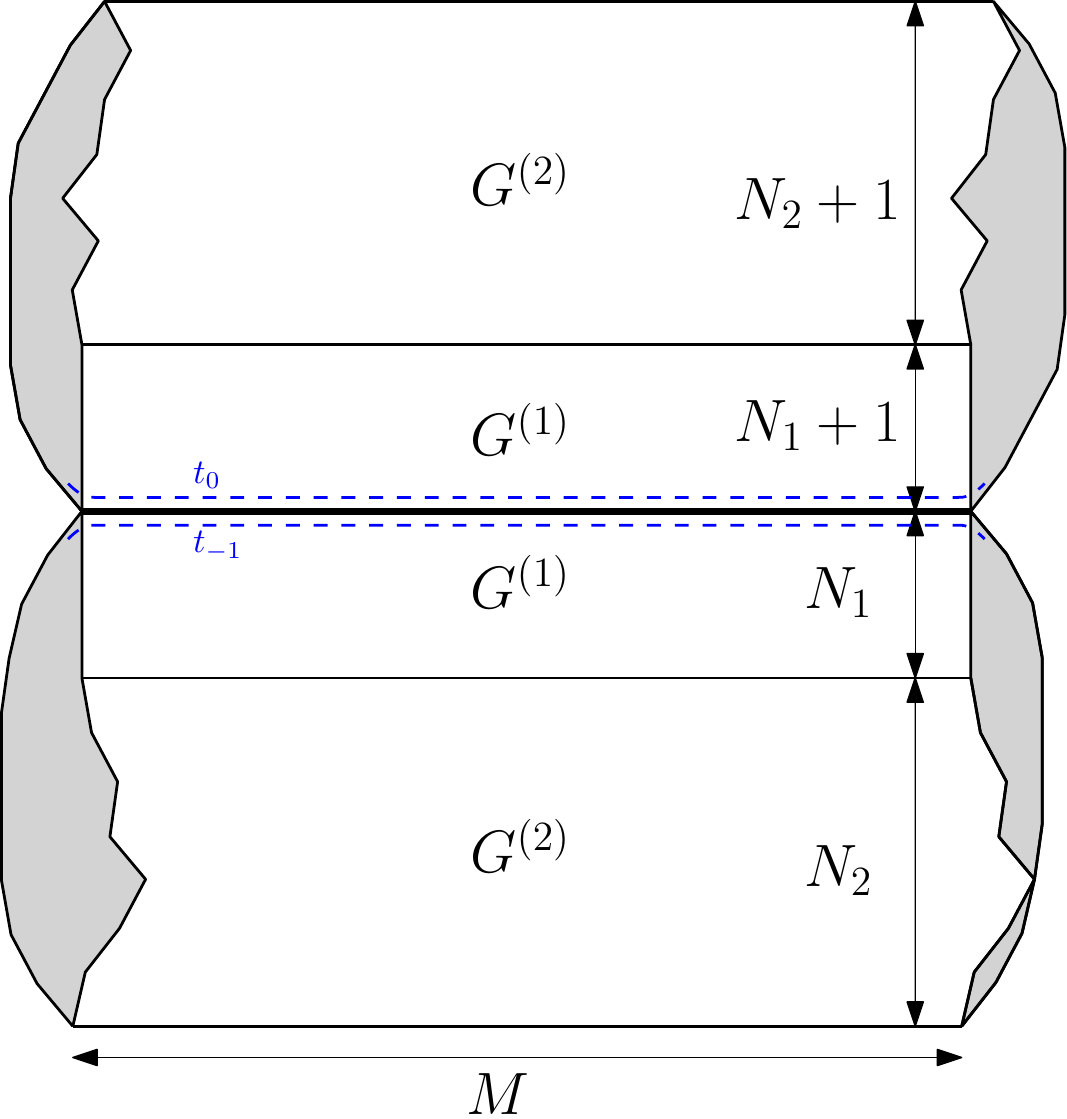}
  \hspace{1cm}
  \includegraphics[scale=0.6, page=2]{images/gmix3.pdf}
  \caption{\textbf{Left:} The graph $G_\mix$ constructed in both the upper and lower half plane. The convexification is drawn in gray
    \textbf{Right:} By exchanging tracks, the relative positions of $G^{(1)}$ and $G^{(2)}$ are switched,  
    the resulting graph is $\Sigma^\uparrow(G_\mix) = \Sigma^\downarrow(G_\mix) = G_\mix'$.
    Note that there is a slight assymetry in the upper-half and the lower-half planes.}
  \label{fig:gmix}
\end{figure}

In certain parts of the proofs that follow, we construct a mixture as described above, in both the upper and lower half-plane, as depicted in Figure~\ref{fig:gmix}. 
That is, we set 
$$
\tilde \bb = ( \beta^{(2)}_{-N_2},\dots, \beta^{(2)}_{-1},\beta^{(1)}_{-N_1}, \dots, \beta^{(1)}_{N_1}, \beta^{(2)}_{0},\dots, \beta^{(2)}_{N_2})
$$
and $\tilde \ba = (\alpha_n)_{-M \leq n \leq M}$ and define the base as the vertices of $G_{\tilde \ba, \tilde \bb}^\diamond$ between $t_{-1}$ and $t_0$. 
Then, set $G_\mix$ to be the separate convexification of the portions of $G_{\tilde \ba, \tilde \bb}$ above and below the base.
We will call $G_\mix$ the \emph{symmetric} version of the \emph{mixed graph}. 

The sequences $\Sigma^\uparrow$ and $\Sigma^\downarrow$ of track exchanges are defined in this case by 
performing the procedure described above separately on both sides of the base. 
For instance, $\Sigma^\uparrow$ is the sequence of \stts that pushes $t_{N_1}$ all the way to the top and $t_{-N_1}$ all the way to the bottom, then  $t_{N_1-1}$ and $t_{-N_1+1}$ all the way to the top and bottom respectively, etc.
Observe that the blocks below the base, and therefore the number of line exchanges applied, differ by one from those above due to the track $t_0$.

%
%

\subtitle{Local behaviour of an open path}
In the proofs of the coming sections we will utilize the line exchanges defined above to transport certain connection estimates from $\graph{1}$ to $\graph{2}$.
To that end, we will need to control the effect that the line exchanges have on open paths.
Recall that the coupling of Figure~\ref{fig:simple_transformation_coupling} is designed to preserve connections.
As such, any open path before a \stt has a corresponding open path in the resulting configuration. 

Let $G_\mix$ be a mixed graph and $t,t'$ be two adjacent horizontal tracks.
Let $\omega$ be a configuration on  $G_\mix$ and $\gamma$ be a simple path, open in $\omega$, 
and contained in the square lattice block of $G_\mix$. 
Then, the intersection of $\gamma$ with the tracks $t$ and $t'$ 
may be split into disjoint segments of two edges (or of one edge if the endpoint of $\gamma$ is on the line between $t$ and $t'$).
The effect of the transformations on $\gamma$ may therefore be understood simply by studying how each individual segment is affected. 
Each segment is actually only affected by at most three consecutive \stts of $\Sigma_{t,t'}$, and the effect of these is summarized in Figure~\ref{fig:path_transformations}.

A very similar analysis appears in~\cite[Sec.~5.3]{GriMan14}.
The only difference between Figure~\ref{fig:path_transformations} and~\cite[Fig.~5.5.]{GriMan14} 
is in the probabilities of secondary outcomes, which are adapted to the random-cluster model.
The exact values will be relevant in Section~\ref{sec:quantum}, when studying the quantum model.

\begin{figure}[htb]
  \centering
  \includegraphics[scale=0.9]{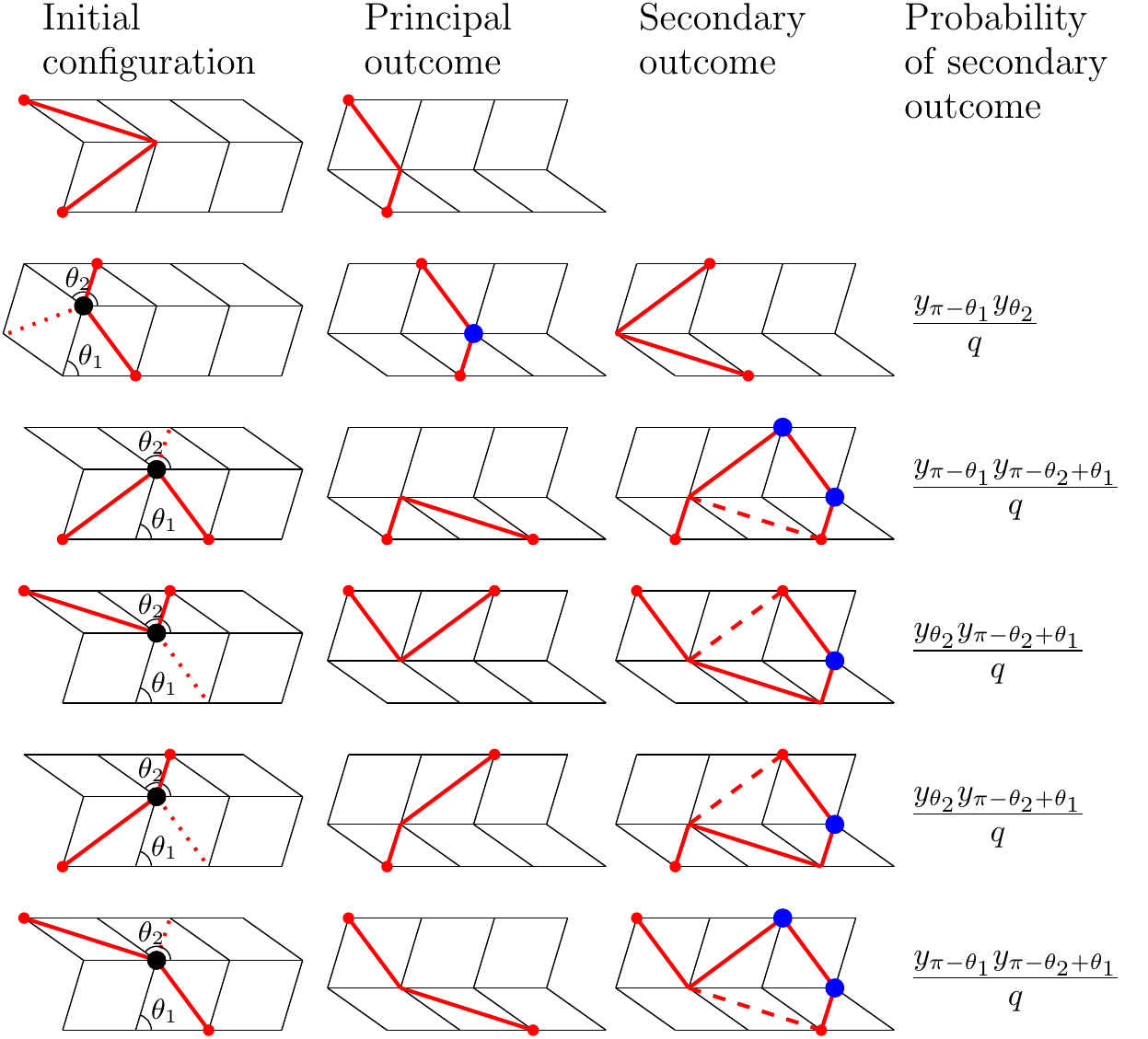}
  \caption{Path transformations. 
    The left column exhausts all the possible intersections of $\gamma$ (in thick red lines) with $t$ and $t'$. 
    The second column depicts the ``principal'' outcome, which arises with probability $1$ when there is no secondary outcome or when the dotted red edge in the initial diagram is closed.
    Otherwise, the resulting configuration is random: either the principal or the secondary outcome (third column) appear, the latter with the probability given in the last column.
    Dashed edges in the secondary outcome are closed.
    The randomness comes from a \stt, and hence is independent of any other randomness.}
  \label{fig:path_transformations}
\end{figure}

Finally, if an endpoint of $\gamma$ lies between the two adjacent horizontal tracks $t$ and $t'$, 
a special segment of length $1$ appears in the intersection of $\gamma$ with $t$ and $t'$. 
This segment obeys separate rules; in particular it may be contracted to a single point, as shown in Figure~\ref{fig:path_transfo_ep}.

\begin{figure}[htb]
  \centering
  \includegraphics[scale=0.92]{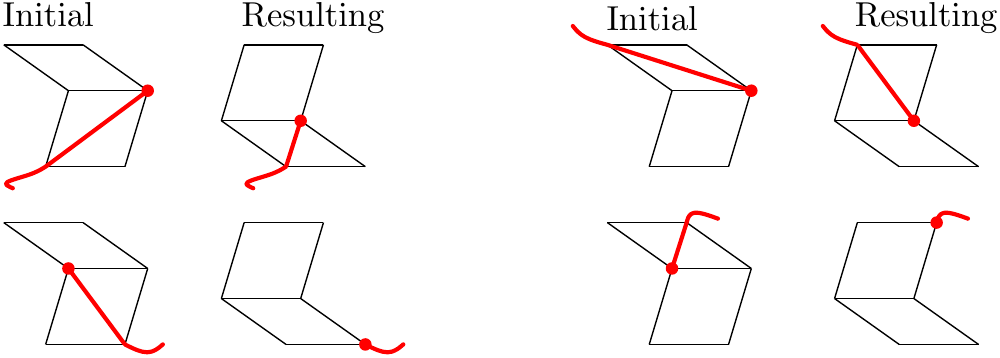}
  \caption{If an endpoint of a path lies between two tracks, the corresponding edge is sometimes contracted to a single point.}
  \label{fig:path_transfo_ep}
\end{figure}

\subsubsection{From isoradial square lattices to general graphs}

Let $\bbG$ be an isoradial graph. 
We call a \emph{grid} of $\bbG$ two  bi-infinite families of tracks $(s_{n})_{n \in \bbZ}$ and $(t_{n})_{n \in\bbZ}$ of $\bbG$ with the following properties.
\begin{itemize}
\item The tracks of each family do not intersect each other.
\item All tracks of $\bbG$ not in $(t_{n})_{n \in \bbZ}$ intersect all those of $(t_{n})_{n \in \bbZ}$.
\item All tracks of $\bbG$ not in $(s_{n})_{n \in \bbZ}$ intersect all those of $(s_{n})_{n \in \bbZ}$.
\item The intersections of $(s_{n})_{n \in \bbZ}$ with $t_0$ appear in order along $t_0$ (according to some arbitrary orientation of $t_0$) and the same holds for the intersections of $(t_{n})_{n \in \bbZ}$ with $s_0$. 
\end{itemize}
The tracks $(s_n)_{n \in \bbZ}$ and $(t_n)_{n \in \bbZ}$ are called \emph{vertical} and \emph{horizontal} respectively.
The vertices of $\bbG^\diamond$ below and adjacent to $t_0$ are called the base of $\bbG$.

In our setting, the existence of a grid is guaranteed by the following lemma.

\begin{lem} \label{lem:square_grid}
  Let $\bbG$ be an isoradial graph.
  Then,
  \begin{itemize}
  \item if $\bbG$ is doubly-periodic, it contains a grid;
  \item $\bbG$ is an embedding of the square lattice if and only if any of its grid contains all its tracks.
  \end{itemize}
\end{lem}

It may be worth mentioning that if $\bbG$ has a grid $(s_n)_{n \in \bbZ}$ and $(t_n)_{n \in \bbZ}$  and $\sigma_1, \dots, \sigma_K$ are \stts that may be applied to $\bbG$, then the tracks $(s_n)_{n \in \bbZ}$ and $(t_n)_{n \in \bbZ}$ of $(\sigma_K \circ \dots \circ \sigma_1) (\bbG)$ also form a grid of the transformed graph $(\sigma_K \circ \dots \circ \sigma_1) (\bbG)$.
Observe also that generally, grids are not unique.

\begin{proof}
  Let $\bbG$ be a doubly-periodic isoradial graph, 
  invariant under the translation by two linearly independent vectors $\tau_1,\tau_2 \in \bbR^2$.
  First notice that, by the periodicity of $\bbG$, 
  each track $t$ of $\bbG$ is also invariant under some translation $a \tau_1 + b \tau_2$ 
  for a certain pair $(a, b) \in \bbZ^2 \backslash \{ (0, 0) \}$. 
  Thus, $t$ stays within bounded distance of the line of direction $a \tau_1 + b \tau_2$,
  which we now call the \emph{asymptotic direction} of $t$.
  Call two tracks \emph{parallel} if they have the same asymptotic direction. 
  
  By the periodicity of $\bbG$, the set of all asymptotic directions of tracks of $\bbG$ is finite. 
  Thus, the tracks of $\bbG$ may be split into a finite number of sets of parallel tracks. 
  It is immediate that two tracks which are not parallel intersect. 
  Conversely, if two parallel tracks intersect, they must do so infinitely many times, due to periodicity. 
  This is impossible, since two tracks can intersect at most once. 
  In conclusion, tracks intersect if and only if they are not parallel. 
  
  Let $t_0$ and $s_0$ be two intersecting tracks of $\bbG$. Orient each of them in some arbitrary direction. 
  Write $\dots, t_{-1}, t_0, t_1, \dots$ for the tracks parallel to $t_0$, ordered by their intersections with $s_0$.
  Similarly, let $\dots, s_{-1}, s_0, s_1, \dots$ be the tracks parallel to $s_0$, in the order of their intersections with $t_0$.
  
  Then, the two families $(s_n)_{n \in \bbZ}$ and $(t_n)_{n \in \bbZ}$ defined above form a grid for $\bbG$:
  the tracks of each family do not intersect each other since they are parallel, 
  but intersect all other tracks, since these have distinct asymptotic directions. 
  
  The second point of the lemma is straightforward. 
\end{proof}

In an isoradial graph $\bbG$ with grid $(s_n)_{n \in \bbZ}$ and $(t_n)_{n \in \bbZ}$, write $\rect(i, j; k, \ell)$ for the region enclosed by $s_i$, $s_j$, $t_k$ and $t_\ell$, 
including the four boundary tracks.
We say that $\rect(i, j; k, \ell)$ has a \emph{square lattice structure} if it is the subgraph of some isoradial square lattice. 
This will be applied to local modifications of bi-periodic graphs,
thus inside $\rect(i,j;k,\ell)$ there may exist tracks not belonging to $(s_n)_{i \leq n \leq j}$ which do not intersect any of the tracks $(s_n)_{i \leq n \leq j}$.  
Such tracks would be vertical in a square lattice containing $\rect(i, j; k, \ell)$, but are not vertical in $\bbG$.
See the right-hand side of Figure~\ref{fig:square_black_pts} for an illustration.

\begin{figure}[htb]
  \centering
  \includegraphics[scale=0.94]{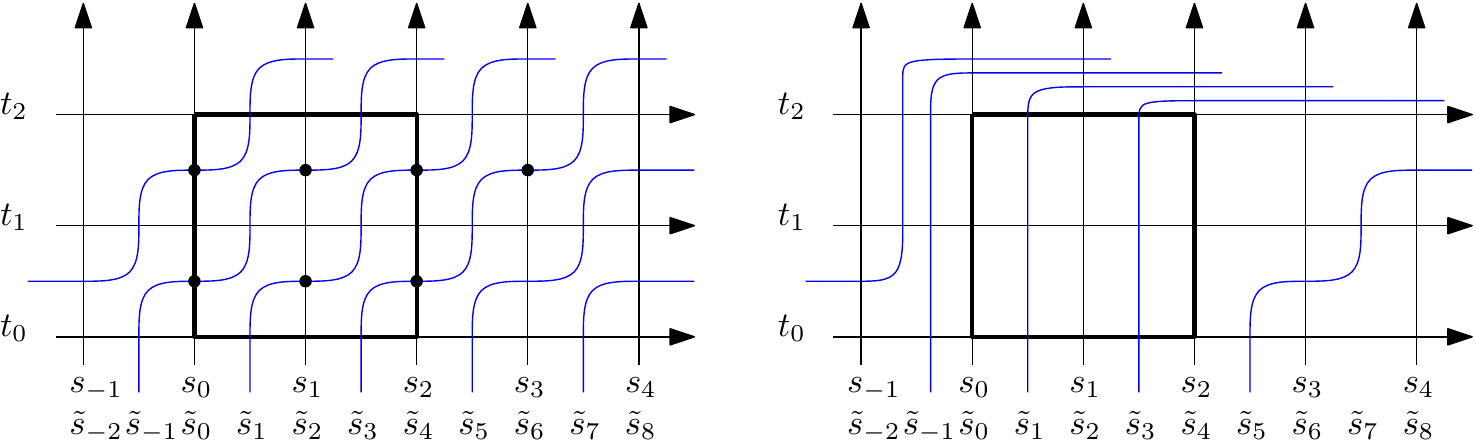}
  \caption{\textbf{Left:} The train tracks of a portion of a doubly-periodic isoradial graph $\bbG$.
    A grid of $\bbG$ is given by horizontal tracks $(s_n)$ and vertical tracks $(t_n)$.
    We denote by $(\tilde s_n)$ its non-horizontal tracks.
    We want to make the region $\rect(0, 2; 0; 2)$ have a square lattice structure by removing all the black points using \stts.
    \textbf{Right:} The black points are removed (from the top) from the region $\rect(0, 2; 0; 2)$, making a square structure appear inside.
    This region contains tracks $\tilde s_1$ and $\tilde s_3$ which would be vertical in a square lattice containing $\rect(0, 2; 0, 2)$ but are not vertical in the original graph $\bbG$ on the left.}
  \label{fig:square_black_pts}
\end{figure}

In the second stage of our transformation from the regular square lattice to arbitrary doubly-periodic isoradial graphs, we use \stts to transfer crossing estimates from isoradial square lattices to periodic graphs. To that end, given a doubly-periodic isoradial graph, 
we will use \stts to construct a large region with a square lattice structure. 
The proposition below is the key to these transformations. 

A \stt is said to act between two tracks $t$ and $t'$ if the three rhombi affected by the transformation are all between $t$ and $t'$, including potentially on $t$ and $t'$.

\begin{prop} \label{prop:bi_to_square}
  Let $\bbG$ be a doubly-periodic isoradial graph with grid $(s_n)_{n \in \bbZ}$, $(t_n)_{n \in \bbZ}$.
  There exists $d \geq 1$ such that for all $M, N \in \bbN$, there exists a finite sequence of \stts $(\sigma_k)_{1 \leq k \leq K}$, 
  each acting between $s_{-(M+dN)}$ and $s_{M+dN}$ and between $t_N$ and $t_0$, none of them affecting any rhombus of $t_0$ 
  and such that in the resulting graph $(\sigma_k \circ \cdots \circ \sigma_1)(\bbG)$, the region $\rect(-M, M; 0, N)$ has a square lattice structure.
\end{prop}

This is a version of~\cite[Lem.~7.1]{GriMan14} with a quantitative control over the horizontal position of the \stts involved.
Obviously, the lemma may be applied also below the base level $t_0$ by symmetry. 

\begin{proof}
  We only sketch this proof since it is very similar to the corresponding one in~\cite{GriMan14}.
  We will only refer below to the track system of $\bbG$; we call an intersection of two tracks a \emph{point}. 
  Fix $M, N \in \bbN$. 

  Index all non-horizontal tracks of $\bbG$ as $(\tilde s_n)_{n\in \bbZ}$, in the order of their orientation with $t_0$, such that $\tilde s_0 = s_0$.
  Then the vertical tracks $(s_n)_{n\in \bbZ}$ of $\bbG$ form a periodically distributed subset of $(\tilde s_n)_{n\in \bbZ}$.
  Let $M_+$ and $M_-$ be such that $\tilde s_{M_+} = s_M$ and $\tilde s_{M_-} = s_{-M}$.

  We will work with $\bbG$ and transformations of $\bbG$ by a finite number of \stts.
  The tracks of any such transformations are the same as those of $\bbG$, we therefore use the same indexing for them. 
  Call a \emph{black point} of $\bbG$, or of any transformation of $\bbG$, an intersection of a track $\tilde s_k$ with $M_- \leq k \leq M_+$ 
  with a non-horizontal track, contained between $t_0$ and $t_N$.
  See Figure~\ref{fig:square_black_pts} for an example.

  Observe that, if in a transformation $(\sigma_k \circ \cdots \circ \sigma_1)(\bbG)$ of $\bbG$ there are no black points, then $(\sigma_k \circ \cdots \circ \sigma_1)(\bbG)$ has the desired property. 
  The strategy of the proof is therefore to eliminate the black points one by one as follows. 

  Orient all non-horizontal tracks of $\bbG$ upwards (that is from their intersection with $t_0$ to that with $t_1$). 
  We say that a black point is \emph{maximal} if, along any of the two tracks whose intersection gives this black point, there is no other black point further. 
  One may then check (see the proof of~\cite{GriMan14}) that if black points exist, then at least one maximal one exists.
  Moreover, a maximal black point may be eliminated by a series of \stts involving its two intersecting tracks and the horizontal tracks between it and $t_N$. 
  Thus, black points may be eliminated one by one, until none of them is left (by the fact that $(s_n)_{n \in \bbZ}$ and $(t_n)_{n \in \bbZ}$ form a grid, only finitely many black points exist to begin with). 
  Call $\sigma_1, \dots, \sigma_K$ the successive \stts involved in this elimination. 
  Then $(\sigma_K \circ \dots \circ \sigma_1)(\bbG)$ has a square lattice structure in $\rect(-M, M; 0, N)$.

  We are left with the matter of controlling the region where \stts are applied. 
  Notice that $\sigma_1,\dots, \sigma_K$ each involve exactly one horizontal track $t_k$ with $0 < k \leq N$. 
  Thus, they all only involve rhombi between $t_0$ and $t_N$, but none of those along $t_0$. 

  Also observe that, due to the periodicity of $\bbG$, between $t_0$ and $t_N$, a track $\tilde s_k$ intersects only tracks $\tilde s_j$ with $|j - k|\leq c N$ for some constant $c$ depending only on the fundamental domain of $\bbG$.
  It follows, by the periodicity of the tracks $(s_n)_{n \in \bbZ}$ in  $(\tilde s_n)_{n \in \bbZ}$,
  that all black points are initially in  $\rect(-M-dN, M+dN; 0,N)$ for some constant $d \geq 0$ depending only on the fundamental domain of $\bbG$. 
  Finally, since all \stts $(\sigma_k)_{0 \leq k \leq K}$ involve one horizontal track and two others intersecting at a black point, 
  each $\sigma_k$ acts in the region of $(\sigma_{k-1} \circ \cdots \circ \sigma_1)(\bbG)$ delimited by $s_{-M-dM}$, $s_{M+dN}$, $t_0$ and~$t_N$.
\end{proof}

\section{Proofs for $1 \leq q \leq 4$} \label{sec:q<4}

Starting from now, fix  $q\in [1,4]$ and let $\bbG$ be a doubly-periodic graph with grid $(s_n)_{n \in \bbZ},(t_n)_{n \in \bbZ}$. 
Recall that $\bbG \in \calG(\eps)$ for some $\eps > 0$. All the constants below depend on the value of $\eps$. 
Write $\phi_\bbG^\xi := \rcisolaw{1}{\bbG}{\xi}$ for the \rcm with parameters $q$ and $\beta = 1$ 
and boundary conditions $\xi \in \{0,1\}$ on $\bbG$.

\subsection{Notations and properties} \label{sec:q<=4_notation}

For integers $i \leq j$ and $k \leq \ell$ recall that $\rect(i,j; k,\ell)$ is the subgraph of $\bbG$ contained between tracks $s_i$ and $s_j$ and between $t_k$ and $t_\ell$ (including the boundary tracks).
Write $\rect(i; k)$ for the centred rectangle $\rect(-i, i; k, k)$ and $\sq(n) = \rect(n; n)$.
The same notation applies to $\bbG^\diamond$ and $\bbG^*$.
We define $R$ and $\Lambda$ in the same way using Euclidean distances.
Note that $\rect$ and $\sq$ are domains with respect to a grid of $\bbG$ whereas $R$ and $\Lambda$ are Euclidean ones and they should all be seen as subregions of $\bbR^2$.

Similarly to the crossings events defined in the introduction, set 
\begin{itemize}
\item $\calC_h (i, j; k, \ell)$: the event that there exists an open path in $\rect(i, j; k, \ell)$
  with one endpoint left of the track $s_i$ and the other right of the track $s_j$.
  This is called a \emph{horizontal crossing} of $\rect(i, j; k, \ell)$.
\item $\calC_v (i, j; k, \ell)$: the event that there exists an open path in $\rect(i, j; k, \ell)$
  with one endpoint below $t_k$ and the other above $t_\ell$.
  This is called a \emph{vertical crossing} of $\rect(i, j; k, \ell)$.
\end{itemize}
The crossings $\calC_h$ and $\calC_v$ can also be defined for symmetric rectangular domains $\rect(m; n)$, in which case we write $\calC_h(m; n)$ and $\calC_v(m; n)$.
Also write $\calC_h^* (i, j; k, \ell)$, $\calC_v^* (i, j; k, \ell)$, $\calC_h^*(m; n)$ and $\calC_v^*(m; n)$ for the corresponding events for the dual model.

To abbreviate the notation, we will henceforth say that $\bbG$ satisfies the RSW property 
if the random-cluster model on $\bbG$ with $\beta = 1$ satisfies this property.
It will be easier to work with the crossing events defined above, rather than the one of the introduction, hence the following lemma. 

\begin{lem} \label{lem:equiv_RSW}
  Fix $\rho > 1$ and $\nu > 0$. 	
  Then, $\bbG$ has the RSW property if and only if 
  there exists $\delta := \delta_1(\rho,\nu) > 0$ such that for all $n \geq 1$,
  \begin{align}
    \phi^0_{\rect((\rho + \nu)n,(1+\nu)n)} \big[\calC_h(\rho n; n)  \big] & \geq \delta, &
    \phi^1_{\rect((\rho + \nu)n,(1+\nu)n)} \big[\calC_h^*(\rho n; n)  \big] & \geq \delta, \nonumber \\
    \phi^0_{\rect((1 + \nu)n,(\rho+\nu)n)} \big[\calC_v(n; \rho n) \big] & \geq \delta, &
    \phi^1_{\rect((1 + \nu)n,(\rho+\nu)n)} \big[\calC_v^*(n; \rho n) \big] & \geq \delta.&
    \tag{BXP($\rho$, $\nu$)} \label{eq:BXP3}
  \end{align}
\end{lem}

In other words, crossing estimates for Euclidean rectangles and rectangles in $\bbG^\diamond$ imply each other. 
Moreover, the aspect ratio $\rho$ and distance $\nu n$ to the boundary conditions is irrelevant;
indeed it is a by-product of the lemma that the conditions \eqref{eq:BXP3} with different values of $\rho > 1$ and $\nu > 0$ are equivalent (obviously with different values for $\delta > 0$).

In general, one would also require crossing estimates as those of \eqref{eq:BXP3} 
for translates of the rectangles $\rect(n;\rho n)$ and $\rect(\rho n; n)$. 
This is irrelevant here due to periodicity. 

The proof of the lemma is elementary. 
It emploies the quasi-isometry between Euclidean distance and the graph distance of $\bbG^\diamond$, the FKG inequality and the comparison between boundary conditions. A similar statement was proved in~\cite[Prop.~4.2]{GriMan14} for Bernoulli percolation. 
Since the boundary conditions matter, additional care is needed here, and the proof is slightly more technical.
Details are skipped here and are given in~\cite[App.~B]{Li-thesis}.

\begin{figure}[htb]
  \centering
  \includegraphics[scale=0.85]{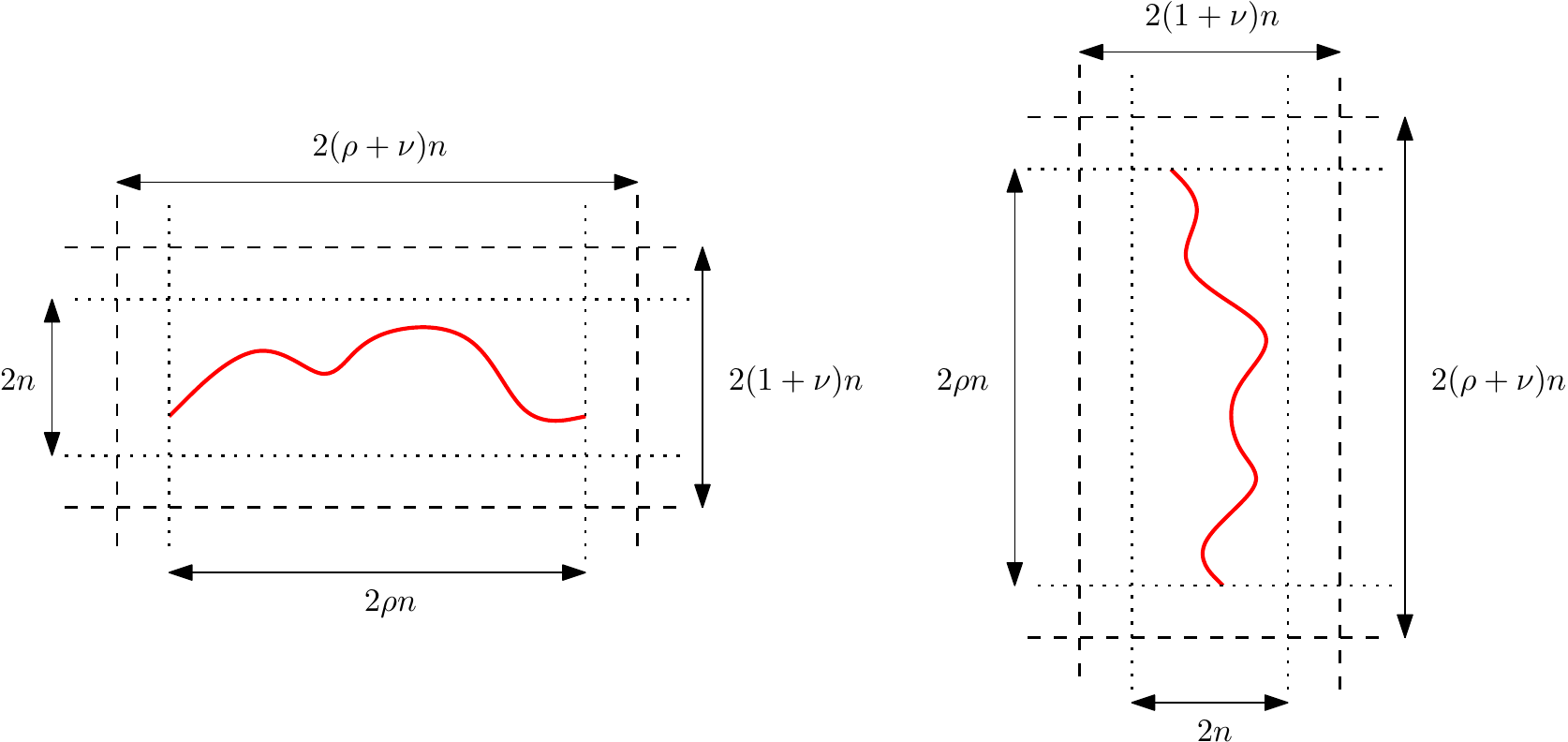}
  \caption{Crossing events in the condition~\eqref{eq:BXP3}. 
    The dotted lines represent the tracks enclosing the domain in which the event takes place, 
    the dashed lines represent the domain in which the \rcm is defined.
  }
  \label{fig:BXP2}
\end{figure}

It is straightforward (as will be seen in Section~\ref{sec:conclusion_q<=4}) 
that the RSW property implies the rest of the points of Theorem~\ref{thm:main} for $1 \leq q \leq 4$. 
The following two sections will thus focus on proving the RSW property for isoradial square lattices (Section~\ref{sec:iso_square_RSW}), then on general doubly-periodic isoradial graphs (Section~\ref{sec:biperiodic_iso}), when $1 \leq q \leq 4$.

\subsection{Isoradial square lattices} \label{sec:iso_square_RSW}

The relevant result for the first stage of the proof (the transfer from regular to arbitrary square lattices) is the following. 

\begin{prop} \label{prop:transfer_RSW}
  Let $\graph 1 = \bbG_{\ba, \bb^{(1)}}$ and $\graph 2 = \bbG_{\ba,\bb^{(2)}}$ be two isoradial square lattices in $\calG(\eps)$. 
  If $\graph 1$ satisfies the RSW property, then so does $\graph 2$.
\end{prop}


The proposition is proved in the latter subsections of this section.
For now, let us see how it implies the following corollary.

\begin{cor} \label{cor:sq_RSW}
  For any $1 \leq q \leq 4$ and any isoradial square lattice $\bbG \in \calG(\eps)$, 
  $\bbG$ satisfies the RSW property. 
\end{cor}

\begin{proof}[Proof of Corollary~\ref{cor:sq_RSW}]
  For the regular square lattice $\bbG_{0, \frac\pi2}$, 
  the random-cluster measure associated by isoradiality (see~\eqref{eq:parameters}) 
  is that with edge-parameter $p_e = \frac{\sqrt q }{1 + \sqrt q}$ .
  It is then known by~\cite{DumSidTas13} that $\bbG_{0, \frac\pi2}$ satisfies the RSW property.
  
  It follows from the application of Proposition~\ref{prop:transfer_RSW} that for any sequence $\bb \in [\eps, \pi - \eps]^\bbZ$, the graph $\bbG_{0, \bb}$ also satisfies the RSW property. 

  Let $\bbG_{\ba, \bb} \in \calG(\eps)$ be an isoradial square lattice.
  Below $\beta_0$ stands for the constant sequence equal to $\beta_0$. 
  Then, $\bbG_{\ba, \beta_0}$ is the rotation by $\beta_0$ of the graph $\bbG_{0, \tilde \ba - \beta_0 + \pi}$, where $\tilde \ba$ is the sequence $\ba$ with reversed order.
  By the previous point, $\bbG_{0, \tilde \ba - \beta_0 + \pi}$ satisfies the RSW property, and hence so does $\bbG_{\ba, \beta_0}$.
  Finally, apply again Proposition~\ref{prop:transfer_RSW} to conclude that $\bbG_{\ba, \bb}$ also satisfies the RSW property.
\end{proof}

%

The rest of the section is dedicated to proving Proposition~\ref{prop:transfer_RSW}. 

\subsubsection{RSW: an alternative definition} \label{sec:alternative_RSW}

Fix an isoradial square lattice $\bbG = \bbG_{\ba, \bb} \in \calG(\eps)$ for some $\eps > 0$. 
Recall that $q \in [1,4]$ is fixed; the estimates below depend only on $q$ and $\eps$. 
Let $x_{i,j}$ be the vertex of $\bbG^\diamond$ between tracks $s_{i-1}, s_{i}$ and $t_{j-1},t_{j}$.
Suppose that $\bbG$ is such that its vertices are those $x_{i, j}$ with $i + j$ even. 
The base of $\bbG$ is then the set $\{(x_{i, 0}: i \in \bbZ \}$.
Moreover, $\bbG$ is translated so that $x_{0,0}$ is the origin $0$ of the plane. 

Define $\calC(m_1, m_2; n)$ to be the event that there exists an open (primal) circuit contained in $\rect(m_2; n)$ that surrounds the segment of the base between vertices $x_{-m_1,0}$ and $x_{m_1,0}$
\footnote{Formally, we allow the circuit to visit vertices of the base, 
  but it is not allowed to cross the base between $x_{-m_1,0}$ and $x_{m_1,0}$.}.
Write $\calC^*(.,.;.)$ for the same event for the dual model.
See figure~\ref{fig:circuits} for an illustration.  

\begin{figure}[htb] \centering
  \begin{minipage}{.48\textwidth}
    \centering
    \includegraphics[scale=1, page=1]{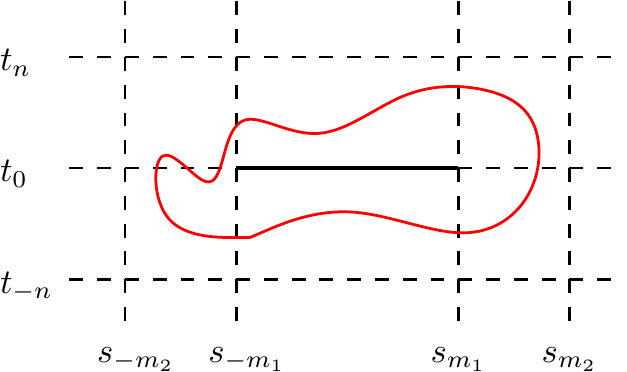}
  \end{minipage}
  \caption{The event $\calC(m_1, m_2; n)$. Such a circuit should not cross the bold segment.}
  \label{fig:circuits}
\end{figure}

The following two results offer a convenient criterion for the RSW property. 
The advantage of the conditions of \eqref{eq:RSW_sq_cond} 
is that they are easily transported between different isoradial square lattices, unlike those of \eqref{eq:BXP3}. 
The main reason is that, due to the last case of Figure~\ref{fig:path_transformation1}, paths may shrink at their endpoints during \stts. 
Circuits avoid this problem.

\begin{lem} \label{lem:RSW_sq_cond}
  Suppose $\bbG$ is as above and suppose that the following conditions hold. 
  There exists $\delta_v > 0$ such that for any $\delta_h > 0$, there exist constants $a \geq 3$ and $b > 3a$ such that for all $n$ large enough, there exist boundary conditions $\xi$ on ${\sq(bn)}$ such that
  \begin{align}\nonumber
    \phi_{\sq(bn)}^\xi \big[ \calC(3an, bn; bn) \big] \geq 1- \delta_h \quad 
    &\text{and} \quad \phi_{\sq(bn)}^\xi \big[ \calC^*( 3an , bn ; bn) \big] \geq 1- \delta_h, \nonumber \\
    \phi_{\sq(bn)}^\xi \big[ \calC_v(an; 2n) \big]  \geq \delta_v \quad
    &\text{and} \quad \phi_{\sq(bn)}^\xi \big[ \calC_v^*(an; 2n) \big]  \geq \delta_v, \nonumber\\
    \phi_{\sq(bn)}^\xi \big[ \calC(an, 3an; n) \big]  \geq \delta_v \quad 
    &\text{and} \quad \phi_{\sq(bn)}^\xi \big[ \calC^*(an, 3an; n) \big]  \geq \delta_v.
    \label{eq:RSW_sq_cond}
  \end{align}
  Then $\bbG$ has the RSW property. 
\end{lem}

Let us mention that the boundary conditions $\xi$ above may be random, 
in which case $\phi_{\sq(bn)}^\xi$ is simply an average of \rcms with different fixed boundary conditions. 
The only important requirement is that they are the same for all the bounds. 

Again, if we were to consider also non-periodic graphs $\bbG$, 
we would require~\eqref{eq:RSW_sq_cond} also for all translates of the events  above. 

The conditions of the lemma above should be understood as follows. 
The last two lines effectively offer lower bounds for the probabilities of vertical and horizontal crossings of certain rectangles.
For Bernoulli percolation, these estimates alone would suffice to prove the RSW property;
for the random-cluster model however, boundary conditions come into play.
The first line is then used to shield the crossing events from any potentially favorable boundary conditions. 
Notice that the fact that $\delta_v > 0$ is fixed and $\delta_h$ may be taken arbitrarily small ensures that events such as those estimated in the first and second (or third) lines must occur simultaneously with positive probability. 
This is the key to the proof. 

Even though the proof is standard (and may be skipped by those familiar with the RSW techniques for the \rcm), 
we present it below. 

\begin{proof}
  Suppose to start that the condition~\eqref{eq:RSW_sq_cond} is satisfied.
  Let $\delta_v > 0$ be fixed.
  Choose $\delta_h \leq \delta_v/2$. 
  Fix $a, b$ as given by the condition.
  Then, for $n$ large enough, by assumption and the inclusion-exclusion formula, there exists $\xi$ such that  
  \begin{align*}
    \phi_{\sq(bn)}^\xi \big[ \calC^*(3a n,b n;bn) \cap \calC_v(an; 2n) \big] \geq \delta_v - \delta_h \geq \delta_h.
  \end{align*}
  Notice that the vertical path defining $\calC_v(an; 2n)$ is necessarily inside the dual circuit defining $\calC^*(3an, bn; bn)$, 
  since the two may not intersect.
  Also, notice that $\calC_v(an; 2n)$ induces a vertical crossing of $\rect(an; 2n)$.
  Thus, we can use the following exploration argument to compare boundary conditions.

  For a configuration $\omega$, define $\Gamma^*(\omega)$ to be the outmost dually-open circuit as in the definition of $C^*(3an, bn; bn)$ 
  if such a circuit exists.
  Let $\Int(\Gamma^*)$ be the region surrounded by $\Gamma^*$, seen as a subgraph of $\bbG$.
  We note that $\Gamma^*$ can be explored from the outside and as a consequence, the \rcm in $\Int(\Gamma^*)$, conditionally on $\Gamma^*$, is given by $\phi_{\Int(\Gamma^*)}^0$.
  Thus,
  \begin{align*}
    \phi_{\sq(bn)}^\xi \big[ \calC^*(3an, bn; bn) \cap \calC_v(an; 2n) \big] 
    & = \sum_{\gamma^*} \phi_{\sq(bn)}^\xi \big[ \calC_v(an; 2n) \,\big|\, \Gamma^* = \gamma^* \big]  \phi_{\sq(bn)}^\xi \big[ \Gamma^* = \gamma^* \big] \\
    & = \sum_{\gamma^*} \phi_{\Int(\Gamma^*)}^0 \big[ \calC_v(an; 2n) \big]  \phi_{\sq(bn)}^\xi \big[ \Gamma^* = \gamma^* \big] \\
    & \leq \sum_{\gamma^*} \phi_{\sq(bn)}^0 \big[ \calC_v(an; 2n) \big]  \phi_{\sq(bn)}^\xi \big[ \Gamma^* = \gamma^* \big] \\
    & \leq \phi_{\sq(bn)}^0 \big[ \calC_v(an; 2n) \big],
  \end{align*}
  where the summations are over all possible realisations $\gamma^*$ of $\Gamma^*$.
  The first inequality is based on the comparison between boundary conditions and on the fact that $\Int(\gamma^*) \subset \sq(bn)$ for all $\gamma^*$.
  Hence, we deduce that,
  \begin{align*}
    \phi_{\sq(bn)}^0 \big[ \calC_v(an; 2n) \big] \geq \delta_h.
  \end{align*}
  Similarly, observe that 
  \begin{align*}
    \phi_{\sq(bn)}^\xi \big[ \calC^*(3an, bn; bn) \cap \calC(an, 3an; n) \big] 
    \geq \delta_h.
  \end{align*}
  Again, the circuit defining $\calC(an, 3an; n)$ is necessarily inside the dual circuit defining $\calC^*(3an, bn; bn)$ 
  and it therefore induces a horizontal crossing of $\rect(an; n)$.
  Using the same exploration argument as above, we deduce that
  \begin{align} \label{eq:RSW_Ch}
    \phi_{\sq(bn)}^0 \big[ \calC_h(an; n) \big] \geq \delta_h.
  \end{align}
  The same may be performed for the dual model.
  Since these computations hold for arbitrary $n$ large enough, we obtain for all $n \geq 1$
  \begin{align*}
    \phi_{\sq(bn)}^0 \big[ \calC_v (an; 2n) \big] \geq \delta_h,
    & \qquad \phi_{\sq(bn)}^0 \big[ \calC_h (an; n) \big] \geq \delta_h
    \qquad \text{and} \\
    \phi_{\sq(bn)}^1 \big[ \calC_v^* (an; 2n) \big] \geq \delta_h,
    & \qquad \phi_{\sq(bn)}^1 \big[ \calC_h^* (an; n) \big] \geq \delta_h.	
  \end{align*}

  We claim that~\eqref{eq:BXP3} follows from the above.
  Indeed, the inequalities above for horizontal crossing are of the desired form. 
  However, vertical crossings are only bounded for short and potentially wide rectangles. 
  Notice however that, by combining crossings as in Figure~\ref{fig:alternative_RSW_crossing} and using the FKG inequality, 
  \begin{align} \label{eq:RSW_Cv}
    \phi^0_{\sq (abn)} \big[ \calC_v (an; a^2n) \big]
    \geq \phi_{\sq (bn)}^0 \big[ \calC_v (an; 2n) \big]^{a^2-1} \phi_{\sq(bn)}^0 \big[ \calC_h (an; n) \big]^{a^2-1}  
    \geq \delta_h^{2a^2 - 2}
  \end{align}

  \begin{figure}[htb] \centering
    \includegraphics[scale=0.8]{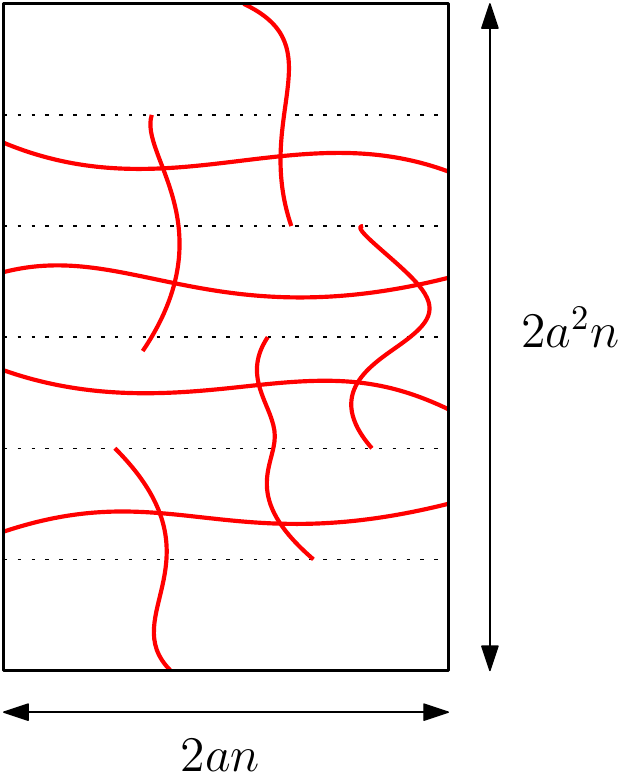}
    \caption{A vertical crossing in $\rect(an; a^2n)$ created by superimposing shorter vertical and horizontal crossings. The distance between two consecutive horizontal dotted lines is $2n$.}
    \label{fig:alternative_RSW_crossing}
  \end{figure}

  Equations~\eqref{eq:RSW_Ch} and~\eqref{eq:RSW_Cv} imply~\eqref{eq:BXP3} with $\rho = a$ and $\nu = a(b-a)$, and Lemma~\ref{lem:equiv_RSW} may be used to conclude.
\end{proof}

As the next lemma suggests, condition~\eqref{eq:RSW_sq_cond} is actually equivalent to the RSW property.
The following statement may be viewed as a converse to Lemma~\ref{lem:RSW_sq_cond}.

\begin{lem} \label{lem:RSW_to_cond}
  Assume that $\bbG$ has the RSW property.
  Fix $a > 1$. 
  Then, there exists $\delta_v > 0$ such that for any $\delta_h > 0$, there exist $b > 3a$ 
  such that for all $n$ large enough, the following condition holds,
  \begin{align}\nonumber
    \phi_{\sq(bn)}^0 \big[ \calC(3an, bn; bn) \big] \geq 1- \delta_h/2 \quad 
    & \text{and} \quad \phi_{\sq(bn)}^1 \big[ \calC^*( 3an , bn ; bn) \big] \geq 1- \delta_h/2, \nonumber \\
    \phi_{\sq(bn)}^0 \big[ \calC_v( \tfrac{an}{2}; \tfrac{an}{2}) \big]  \geq 2\delta_v \quad
    & \text{and} \quad \phi_{\sq(bn)}^1 \big[ \calC_v^*( \tfrac{an}{2}; \tfrac{an}{2}) \big]  \geq 2\delta_v, \nonumber\\
    \phi_{\sq(bn)}^0 \big[ \calC(an, 2an; \tfrac{n}{a}) \big] \geq 2 \delta_v \quad 
    & \text{and} \quad \phi_{\sq(bn)}^1 \big[ \calC^*(an, 2an; \tfrac{n}{a}) \big] \geq 2\delta_v.
    \label{eq:RSW_sq_cond2}
  \end{align}
\end{lem}

The proof is a standard application of the RSW theory and readers are referred to~\cite[App.~C]{Li-thesis}.
Let us only mention that it uses the fact that 
$$
\phi_{\sq(bn)}^0 \big[ \calC(3an, bn; bn) \big]  \xrightarrow[b \to \infty]{}1, \qquad \text{uniformly in } n.
$$
This is a typical consequence of the \emph{strong} RSW property; it appears in other forms in various applications. 

\subsubsection{Transporting RSW: proof of Proposition~\ref{prop:transfer_RSW}} \label{sec:path_transport}

Fix $\graph{1} = \bbG_{\ba, \bb^{(1)}}$ and $\graph{2} = \bbG_{\ba,\bb^{(2)}}$ two isoradial square lattices in $\calG(\eps)$. 
Suppose $\graph{1}$ satisfies the RSW property.

Let $G_\mix$ be the symmetric mixed graph of $\graph{1}$ and $\graph{2}$ constructed in Section~\ref{sec:mix}, 
where the width of each strip is $2M+1$ and the height is $N = N_1 = N_2$ (for $M$ and $N$ to be mentioned below). 
We use here the construction both above and below the base, where each side is convexified separately.  
Let $\tilde G_\mix = \Sigma^\uparrow (G_\mix) = \Sigma^\downarrow (G_\mix)$ be the graph obtained after exchanging the tracks $t_0, \dots, t_N$ of $G_\mix$ with $t_{N+1}, \dots, t_{2N+1}$ and $t_{-1}, \dots, t_{-N}$ with $t_{-(N+1)}, \dots,t_{-2N}$.
Write $\phi_{G_\mix}$ and $\phi_{\tilde G_\mix}$ for the \rcms on $G_\mix$ and $\tilde G_\mix$, respectively, 
with parameters $q \in [1, 4]$, $\beta =1$ and free boundary conditions. 

The estimates below are the key to the proof of Proposition~\ref{prop:transfer_RSW}.
They correspond to similar statements in~\cite{GriMan14} for Bernoulli percolation.

\begin{prop}[Prop.~6.4 of~\cite{GriMan14}] \label{prop:horizontal_transport}
  There exist  $\lambda,n_0 > 1$, depending on $\eps$ only, such that, for all $\rho_{\rmout} > \rho_{\rmin}> 0$, $n \geq n_0$ and sizes $M \geq (\rho_{\rmout} + \lambda) n $ and $N \geq \lambda n$, 
  \begin{align*}
    \phi_{\tilde G_\mix} \big[ \calC(\rho_{\rmin} n, (\rho_{\rmout} + \lambda) n; \lambda n) \big]
    \geq (1- \rho_{\rmout} e^{-n}) \phi_{G_\mix} \big[ \calC (\rho_{\rmin} n, \rho_{\rmout} n; n) \big].
  \end{align*}
\end{prop}

\begin{prop}[Prop.~6.8 of~\cite{GriMan14}] \label{prop:vertical_transport}
  There exist $\delta \in (0, \frac12)$ and $c_n >0$ satisfying $c_n\to 1$ as $n \to \infty$ such that, for all $n$ and sizes $M \geq 4n$ and $N \geq n$, 
  \begin{align*}
    \phi_{\tilde G_\mix} \big[ \calC_v(4n; \delta n) \big] \geq c_n \phi_{G_\mix} \big[ \calC_v(n; n) \big].
  \end{align*}
\end{prop}

The proofs of the two statements are similar to those of~\cite{GriMan14}. 
They do not rely on the independence of the percolation measure, they do however use crucially the independence of the randomness appearing in the \stts. 
More details about this step are given in Section~\ref{sec:details_hv_transport_quantum} when we will treat the quantum case, since more explicit estimates will be needed.
However, we will not provide full proofs since they are very similar to the corresponding statements in~\cite{GriMan14}.

Let us admit the two propositions above for now and finish the proof of Proposition~\ref{prop:transfer_RSW}.

\begin{proof}[Proof of Proposition~\ref{prop:transfer_RSW}]
  Fix parameters $n_0, \lambda > 1$ and $\delta > 0$ as in Propositions~\ref{prop:horizontal_transport} and~\ref{prop:vertical_transport}.
  Since $\graph 1$ satisfies the RSW property, Lemma~\ref{lem:RSW_to_cond} applies to it.
  Choose $a = \max \{ \lambda, \frac2\delta, 1 \}$ and an arbitrary $\delta_h > 0$. 
  By Lemma~\ref{lem:RSW_to_cond}, there exist $b > 3a$ and $\delta_v > 0$ such that, for all $n$ large enough, 
  \begin{align}\nonumber
    & \phi_{\sq({bn})}^0 \big[ \calC(3an, bn ; bn) \big] \geq 1- \delta_h/2, \\
    & \phi_{\sq({bn})}^0 \big[ \calC_v(\tfrac{an}{2}; \tfrac{an}{2}) \big] \geq 2\delta_v,  \nonumber\\
    & \phi_{\sq({bn})}^0 \big[ \calC(an, 2an; \tfrac{n}{a}) \big] \geq 2\delta_v.
    \label{eq:HVH1}
  \end{align}
  We will prove that $\graph 2$ satisfies~\eqref{eq:RSW_sq_cond} 
  for these values of $a$, $\delta_v$ and $\delta_h$, with $b$ replaced by $\tilde b = (1+\lambda)b$. 
  The boundary conditions $\xi$ will be fixed below.
  We start by proving~\eqref{eq:RSW_sq_cond} for the primal events. 
  
  Take $M = N \geq (\lambda+1) b n$ for constructing $G_\mix$. 
  Then, since the balls of radius $bn$ in $G_\mix$ and in $\graph 1$ are identical, we deduce from the above that 
  \begin{align*}
    & \phi_{G_\mix} \big[ \calC(3an , bn; bn) \big] \geq 1- \delta_h/2, \\
    & \phi_{G_\mix} \big[ \calC_v(\tfrac{an}{2}; \tfrac{an}{2}) \big] \geq 2\delta_v, \\
    & \phi_{G_\mix} \big[ \calC(an, 2an; \tfrac{n}{a}) \big] \geq 2\delta_v. 
  \end{align*}
  We used here that the boundary conditions on $\sq(bn)$ in~\eqref{eq:HVH1} are the least favorable for the existence of open paths. 
  
  For $n \geq a n_0$, Propositions~\ref{prop:horizontal_transport} and~\ref{prop:vertical_transport} then imply
  \begin{align*}
    & \phi_{\tilde G_\mix} \big[ \calC(3an ,(\lambda +1) bn ;\lambda bn) \big] \geq (1- e^{-bn}) (1-\delta_h/2),\\
    & \phi_{\tilde G_\mix} \big[ \calC_v(an; \tfrac{\delta}2 an) \big]  \geq 2c_n  \delta_v, \\
    & \phi_{\tilde G_\mix} \big[ \calC(an, (2a +\tfrac\lambda a)n; \tfrac{\lambda}a n ) \big] \geq 2(1- 2a^2 e^{-n/a}) \delta_v.
  \end{align*}
  Now, take $n$ large enough so that $2e^{-bn} < \delta_h$, $2c_n > 1$ and $a^2 e^{-n/a} < 1/4$.
  These bounds ultimately depend on $\eps$ only.
  Observe that this implies~\eqref{eq:RSW_sq_cond} for the primal model.
  Indeed, set $\tilde b = (\lambda +1)b$, then, due to the choice of $a$,
  \begin{align*}
    &\phi_{\tilde G_\mix} \big[ \calC(3an, \tilde b n; \tilde b n) \big] \geq (1-\delta_h/2)^2 \geq 1-\delta_h, \\
    &\phi_{\tilde G_\mix} \big[ \calC_v(an; n) \big] \geq \delta_v, \\
    &\phi_{\tilde G_\mix} \big[ \calC(an, 3a n; n ) \big] \geq \delta_v.
  \end{align*}
  
  The same procedure may be applied for the dual model to obtain the identical bounds for $\calC^*(.,.;,)$ and $\calC_v^*(.;.)$. 
  
  By choice of $M$ and $N$, the region $\sq(\tilde bn)$ of $\tilde G_\mix$ is also a subgraph of $\bbG^{(2)}$.
  This implies~\eqref{eq:RSW_sq_cond} for $\bbG^{(2)}$.
  The boundary conditions $\xi$ appearing in~\eqref{eq:RSW_sq_cond} are those induced on $\sq(\tilde b n)$ by the free boundary conditions on $\tilde G_\mix$.
  These are random boundary conditions, but do not depend on the events under study.
  In particular, they are the same for all the six bounds of~\eqref{eq:RSW_sq_cond}.
\end{proof}

\subsubsection{Sketch of proof for Propositions~\ref{prop:horizontal_transport} and~\ref{prop:vertical_transport} } \label{sec:details_hv_transport}

The proofs of Propositions~\ref{prop:horizontal_transport} and~\ref{prop:vertical_transport} 
are very similar to those of Propositions 6.4 and 6.8 in~\cite{GriMan14}, with only minor differences.
Nevertheless, we sketch them below for completeness.
The estimates in the proofs are specific to the random-cluster model and will be important in Section~\ref{sec:quantum}. 

We keep the notations $G_\mix$ and $\tilde G_\mix$ introduced in the previous section.

\begin{proof}[Proof of Proposition~\ref{prop:horizontal_transport}]
  We adapt the proof from Proposition 6.4 (more precisely, Lemma 6.7) of~\cite{GriMan14} to our case.

  Recall the definition of $\Sigma^\downarrow$, the sequence of \stts to consider here: above the base level, we push down tracks of $\graph{2}$ below those of $\graph{1}$ one by one, from the bottom-most to the top-most; below the base level, we proceed symmetrically.
  Let $\bbP$ be a probability measure defined as follows.
  Pick a configuration $\omega$ on $G_\mix$ according to $\phi_{G_\mix}$; apply the sequence of \stts $\Sigma^\downarrow$ to it using the coupling described in Figure~\ref{fig:simple_transformation_coupling}, where the randomness potentially appearing in each transformation is independent of $\omega$ and of all other transformations.
  Thus, under $\bbP$ we dispose of configurations on all intermediate graphs in the transformation from $G_\mix$ to $\tilde G_\mix$.
  Moreover, in light of Proposition~\ref{prop:coupling}, $\Sigma^\downarrow(\omega)$ has law $\phi_{\tilde G_\mix}$.

  We will prove the following statement
  \begin{align}
    \bbP \big[ \Sigma^\downarrow(\omega) \in \calC(\rho_{\rmin} n, (\rho_{\rmout} + \lambda) n; \lambda n) \,\big|\, 
    \omega \in \calC(\rho_{\rmin} n, \rho_{\rmout} n; n) \big]
    \geq 1- \rho_{\rmout} e^{-n},
    \label{eq:up_drift1}
  \end{align}
  for any values $\rho_{\rmout} > \rho_{\rmin} > 0$, $n \geq n_0$, $M \geq (\rho_{\rmout} + \lambda) n$ and $N \geq \lambda n$, where $\lambda, n_0 > 1$ will be chosen below.
  This readily implies Proposition~\ref{prop:horizontal_transport}. 

  \medbreak

  Fix $\rho_{\rmout}, \rho_{\rmin}, n, M$ and $N$ as above. 
  Choose $\omega_0 \in  \calC(\rho_{\rmin} n, \rho_{\rmout} n; n)$ and let $\gamma$ be an $\omega_0$-open circuit as in the definition of $\calC(\rho_{\rmin} n, \rho_{\rmout} n; n)$.
  As the transformations of $\Sigma^\downarrow = \sigma_K\circ \dots \circ \sigma_1$ are applied to $\omega_0$, the circuit $\gamma$ is transformed along with $\omega_0$.
  Thus, for each $0 \leq k \leq K$, $(\sigma_k\circ \dots \circ \sigma_1)(\gamma)$ is an open path in $(\sigma_k\circ \dots \circ \sigma_1)(\omega_0)$ on the graph $(\sigma_k\circ \dots \circ \sigma_1)(G_\mix)$. 

  Since no \stt of $\Sigma^\downarrow$ affects the base, $\Sigma^\downarrow(\gamma)$ remains a circuit surrounding the segment of the base between $x_{-\rho_{\rmin} n,0}$ and $x_{\rho_{\rmin} n,0}$.
  Therefore, the only thing that is left to prove is that
  \begin{align} \label{eq:up_drift2}
    \bbP \big[\Sigma^\downarrow(\gamma) \in R \big( (\rho_{\rmout} + \lambda) n; \lambda n \big) \,\big|\,
    \omega = \omega_0 \big]
    \geq 1- \rho_{\rmout} e^{-n}.
  \end{align}
  Set
  \begin{align*}
    \gamma^{(k)} = 
    (\Sigma^\downarrow_{N_1+k} \circ \Sigma^\downarrow_{-(N_1+k)}) \circ \cdots \circ
    (\Sigma^\downarrow_{N_1+1} \circ \Sigma^\downarrow_{-(N_1+1)}),
  \end{align*}
  where $\Sigma^\downarrow_i = \Sigma_{t_0, t_i} \circ \cdots \circ \Sigma_{t_{N_1}, t_i}$ for $i \geq 0$ and $\Sigma^\downarrow_i = \Sigma_{t_{-1}, t_{i}} \circ \cdots \circ \Sigma_{t_{-N_1}, t_i}$ for $i < 0$.
  The path $\gamma^{(k)}$ thus defined is the transformation of $\gamma$ after the first $k$ tracks of $\graph 2$ above the base were sent down, and the symmetric procedure was applied below the base.

  \medbreak

  In~\cite{GriMan14}, the vertices of $G_\mix$ visited by $\gamma^{(k)}$ were shown to be contained in a region whose evolution with $k = 0, \dots, N_2 + 1$ is explicit. 
  This is done separately above and below the base level, and we focus next on the upper half-space. 

  Let $H^0 = \{(i,j) \in \bbZ \times \bbN : -(\rho_{\rmout}+1 )n \leq i - j \text{ and } i+j \leq (\rho_{\rmout}+1)n \text{ and } j \leq n\}$.
  Then, $H^{k+1}$ is defined from $H^k$ as follows. If $(i,j) \in \bbZ \times \bbN$ is such that $(i,j), (i-1,j)$ or $(i+1,j)$ are in $H^k$, then $(i,j) \in H^{k+1}$. 
  Otherwise, if $(i,j-1) \in H^k$, then $(i,j)$ is included in $H^k$ with probability $\eta \in (0, 1)$, independently of all previous choices.
  We will see later how the value of $\eta$ is chosen using the bounded-angles property.

  In consequence, the sets $(H^k)_{0 \leq k \leq N}$ are interpreted as a growing pile of sand, with a number of particles above every $i \in \bbZ$. 
  At each stage of the evolution, the pile grows laterally by one unit in each direction; 
  additionally, each column of the pile may increase vertically by one unit with probability $\eta$ (see Figure~\ref{fig:mountains}).

  \begin{figure}[htb]
    \centering
    \includegraphics[scale=1]{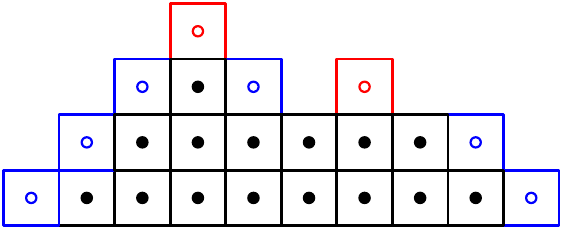}
    \caption{One step of the evolution of $H$: $H^k$ is drawn in black, $H^{k+1}$ contains the additional blue points (since they are to the left or right of vertices in $H^k$) and the red points (these are added due to the random increases in height).  }
    \label{fig:mountains}
  \end{figure}

  Loosely speaking,~\cite[Lem.~6.6]{GriMan14} shows that, if $\eta$ is chosen well, then 
  all vertices $x_{i,j}$ visited by $\gamma^{(k)}$ have $(i,j) \in H^k$ \footnote{This is not actually true, since there is a horizontal shift to be taken into account; let us ignore this technical detail here.}. More precisely, the process $(H^k)_{0 \leq k \leq N}$ may be coupled with the evolution of $(\gamma^{(k)})_{0 \leq k \leq N}$ so that the above is true. 
  This step is proved by induction on $k$, and relies solely on the independence of the \stts and on the estimates of Figure~\ref{fig:path_transformations}.
  Then, \eqref{eq:up_drift2} is implied by the following bound on the maximal height of $H^{N}$:
  \begin{align} \label{eq:Hbound}
    \bbP \big[ \max \{j : (i,j) \in H^{\lambda n} \} \geq \lambda n \big] <  \rho_{{\rmout}}e^{-n}
  \end{align} 
  for some $\lambda > 0$ and all $n$ large enough. 
  The existence of such a (finite) constant $\lambda$ is guaranteed by~\cite[Lem.~3.11]{GriMan13}. 
  It depends on $\eta$, and precisely on the fact that $\eta < 1$~\cite[Lem.~6.7]{GriMan14}. 
  The choice of $\eta <1$ that allows the domination of $(\gamma^{(k)})_{0 \leq k \leq N}$ by $(H^k)_{0 \leq k \leq N}$ is done as follows.


  \medbreak

  We proceed in the same way as in the proofs of Lemmas~6.6 and~6.7 of~\cite{GriMan14}.
  We shall analyze the increase in height of portions of $\gamma^{(k)}$ as given by Figure~\ref{fig:path_transformations}.
  Essentially, the only cases in which $\gamma^{(k)}$ increases significantly in height are depicted in the third and the last line of Figure~\ref{fig:path_transformations}.

  \begin{figure}[htb]
    \centering
    \includegraphics[scale=1]{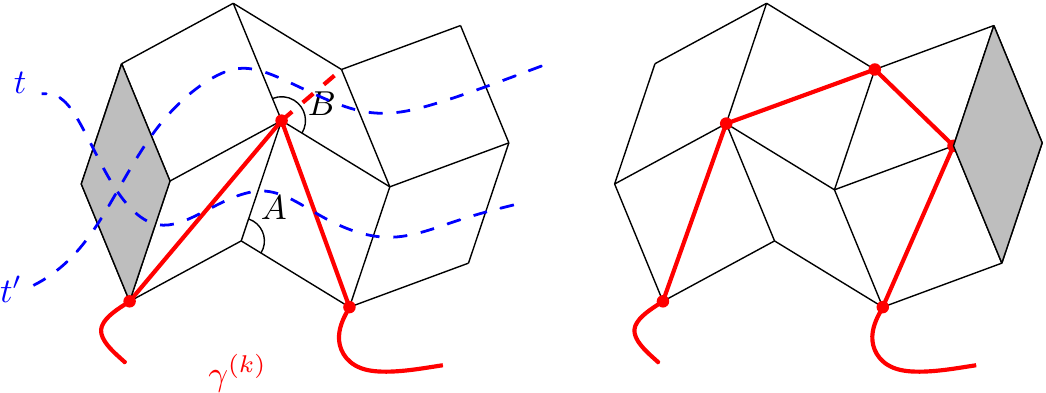}
    \caption{\Stts between tracks $t$ and $t'$ corresponding to the third line of Figure~\ref{fig:path_transformations}.
      The tracks $t$ and $t'$ have transverse angles $A$ and $B$ respectively.
      We assume that a portion of the path $\gamma^{(k)}$ reaches between the tracks $t$ and $t'$ as shown in the figure.
      Moreover, if the dashed edge is open on the left, with probability $\eta_{A, B} = y_{\pi-A} y_{\pi - (B-A)} / q$, the path $\gamma^{(k)}$ drifts upwards by 1 after the track exchange.
    }
    \label{fig:sandpile}
  \end{figure}

  Let us examine the situation which appears in the third line of Figure~\ref{fig:path_transformations} and consider the notations as in Figure~\ref{fig:sandpile}.
  Using the notation of Figure~\ref{fig:sandpile} for the angles $A$ and $B$, the probability that the height of such a $\gamma^{(k)}$ increases by 1 is given by
  \begin{align*}
    \eta_{A, B} & = \frac{y_{\pi - A} y_{\pi - (B-A)}}{q} = \frac{\sin(rA) \sin(r(B-A))}{\sin(r(\pi-A)) \sin(r (\pi- (B-A)))} \\
    & = \frac{ \cos( r(2A-B) ) - \cos( rB ) }{ \cos(r(2A-B)) - \cos(r(2\pi - B)) },
  \end{align*}
  where we recall that $r = \cos^{-1} (\tfrac{\sqrt{q}}{2}) \leq \tfrac13$ and that, due to the $BAP(\eps)$, 
  $A, B \in [\eps, \pi-\eps]$.

  The same computation also applies to the last line of Figure~\ref{fig:path_transformations}.
  Then, $\eta$ may be chosen as
  \begin{align} \label{eq:eta_sandpile}
    \eta := \sup_{A, B \in [\eps, \pi-\eps]} \eta_{A, B} < 1.
  \end{align}
  The domination of the set of vertices of $\gamma^{(k)}$ by $H^k$ is therefore valid for this value of $\eta$, 
  and~\eqref{eq:up_drift2} is proved for the resulting constant $\lambda$. 

\end{proof}

\begin{rmk}
  When we deal with the quantum model in Section~\ref{sec:quantum}, it will be important to have a more precise estimate on $\eta(\eps)$.
  In particular, we will show that, in this special case, $1 - \eta(\eps) \sim \tau(q) \eps$ as $\eps \to 0$ for some constant $\tau := \tau(q)$ depending only on $q \in [1, 4]$. 
\end{rmk}

\begin{proof}[Proof of Proposition~\ref{prop:vertical_transport}]
  We adapt Proposition 6.8 of~\cite{GriMan14} to our case.
  Fix $n$ and $N, M \geq 2n$, and consider the graph $G_\mix$ as described in the previous section. 
  We recall the definition of $\Sigma^\uparrow$, the sequence of \stts we consider here: above the base level, we pull up tracks of $\graph{1}$ above those of $\graph{2}$ one by one, from the top-most to the bottom-most; below the base level, we proceed symmetrically.

  As in the previous proof, write $\bbP$ for the measure taking into account the choice of a configuration $\omega_0$ according to the \rcm $\phi_{G_\mix}$ as well as the results of the \stts in $\Sigma^\uparrow$ applied to the configuration $\omega_0$.

  The events we are interested in only depend on the graph above the base level, 
  hence we are not concerned with what happens below.
  For $0 \leq i \leq N$, recall from Section~\ref{sec:mix} the notation 
  \begin{align*}
    \Sigma^\uparrow_{i} = \Sigma_{t_{i}, t_{2N+1}} \circ \cdots \circ \Sigma_{t_{i}, t_{N+1}}, 
  \end{align*}
  for the sequence of \stts moving the track $t_{i}$ of $\graph{1}$ above $\graph{2}$. 
  Then  $\Sigma^\uparrow  = \Sigma^\uparrow_0 \circ \cdots \circ \Sigma^\uparrow_{N}$.
  

  First, note that if $\omega \in \calC_v(n; n)$, we also have $\Sigma^\uparrow_{n+1} \circ \cdots \circ \Sigma^\uparrow_{N} (\omega) \in \calC_v(n; n)$, since the two configurations are identical between the base and $t_n$.
  We will now write, for $0 \leq k \leq n+1$,
  \begin{align*}
    G^k & = \Sigma^\uparrow_{n-k+1} \circ \cdots \circ \Sigma^\uparrow_N (G_\mix), \\
    \omega^k & = \Sigma^\uparrow_{n-k+1} \circ \cdots \circ \Sigma^\uparrow_N (\omega), \\
    D^k & = \{ x_{u, v} \in G^k : |u| \leq n+2k+v, 0 \leq v \leq N+n \}, \\
    h^k & = \sup \{ h \leq N : \exists u, v \in \bbZ \text{ with } x_{u, 0} \xleftrightarrow{D^k, \omega^k} x_{v, h}  \}.
  \end{align*}
  That is, $h^k$ is the highest level that may be reached by an $\omega^{k}$-open path lying in the trapezoid~$D^k$.
  We note that $G^{n+1} = \tilde G_\mix$ and $\omega^{n+1}$ follows the law of $\phi_{\tilde G_\mix}$.

  With these notations, in order to prove Proposition~\ref{prop:vertical_transport}, it suffices to show the equivalent of~\cite[(6.23)]{GriMan14}, that is
  \begin{align} \label{eq:vertical_h_control}
    \bbP \big[ h^{n+1} \geq \delta n \big]
    \geq c_n \bbP \big[ h^0 \geq n \big],
  \end{align}
  for some $\delta \in (0, \frac12)$ to be specified below and explicit constants $c_n$ with $c_n \to 1$ as $n \to \infty$. 
  Indeed, 
  \begin{align*}
    \bbP [ h^0 \geq n ] \geq \bbP [\omega^0 \in \calC_v(n; n)] =  \phi_{G_\mix}[\calC_v(n; n)].
  \end{align*}
  Moreover, if $h^{n+1} \geq \delta n$, then $\omega^{n+1} \in \calC_v(4n; \delta n)$, 
  and therefore we have $ \phi_{\tilde G_\mix}[ \calC_v(4n; \delta n) ] \geq \bbP(h^{n+1} \geq \delta n)$.

  We may now focus on proving~\eqref{eq:vertical_h_control}.
  To do that, we adapt the corresponding step of~\cite{GriMan14} (namely Lemma 6.9). 
  It shows that $(h^k)_{0 \leq k\leq n+1}$ can be bounded stochastically from below by the Markov process $(H^k)_{0 \leq k\leq n+1}$~\footnote{To be precise, it is shown that for any $k$, the law of $h^k$ dominates that of $H^k$. It is not true that the law of the whole process $(h^k)_{0 \leq k\leq n+1}$ dominates that of $(H^k)_{0 \leq k\leq n+1}$. This step uses \cite[Lem.~3.7]{GriMan13}.}
  given by
  \begin{align}\label{eq:H_Delta}
    H^k = H^0 + \sum_{i=1}^k \Delta_i,
  \end{align}
  where $H^0 = n$ and the $\Delta_i$ are independent random variables with common distribution
  \begin{align} \label{eq:distribution}
    \bbP(\Delta = 0) = 2 \delta, \quad 
    \bbP(\Delta = -1) = 1 - 2 \delta,
  \end{align}
  for some parameter $\delta$ to be specified later.
  Once the above domination is proved, the inequality~\eqref{eq:vertical_h_control} follows by the law of large numbers. 

  The proof of~\eqref{eq:distribution} in \cite{GriMan14} (see equation (6.24) there)
  uses only the independence between different \stts and the finite-energy property of the model.
  Both are valid in our setting. 
  We sketch this below.

  Fix $0 \leq k \leq n$ and let us analyse the $(N-(n-k)+1)^{th}$ step of $\Sigma^\uparrow$, that is $\Sigma^\uparrow_{n-k}$. 
  Write $\Psi_j := \Sigma_{t_{n-k}, t_{N+j}} \circ \cdots \circ \Sigma_{t_{n-k}, t_{N+1}}$ for $0 \leq j \leq N$.
  In other words, $\Psi_j$ is the sequence of \stts that applies to $G^k$ and moves the track $t_{n-k}$ 
  above $j$ tracks of $\graph{2}$, namely $t_{N+1}, \dots, t_{N+j}$.
  Moreover, $\Psi_N = \Sigma^\uparrow_{n-k}$; hence, $\Psi_{N} (G^{k}) = G^{k+1}$.

  \begin{figure}[htb]
    \centering
    \includegraphics[scale=1]{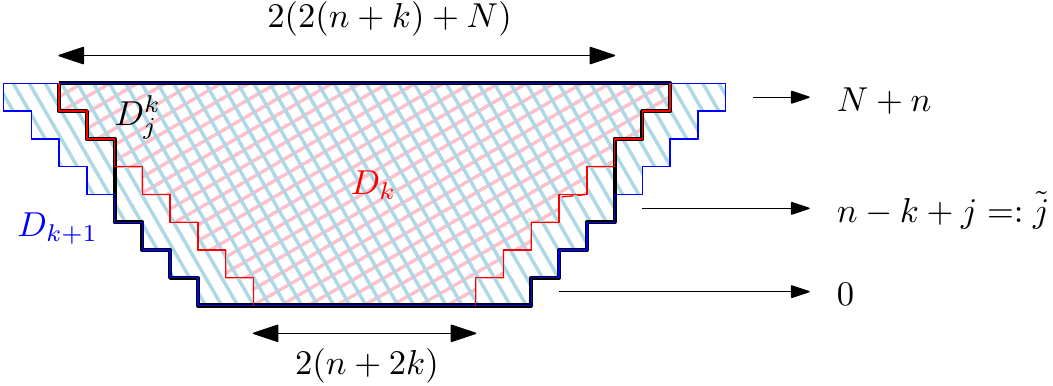}
    \caption{The evolution of $D^{k}$ (red) to $D^{k+1}$ (blue) via intermediate steps $D^k_j$ (black).}
    \label{fig:Dk_classical}
  \end{figure}

  Let $D_j^k$ be the subgraph of $\Psi_j(G^k)$ induced by vertices $x_{u, v}$ with $0 \leq v \leq N+n$ and
  \begin{align*}
    |u| \leq \left\{
      \begin{array}{ll}
        n+2k+v+2 & \text{ if } v \leq \sfj, \\
        n+2k+v+1 & \text{ if } v = \sfj +1, \\
        n+2k+v & \text{ if } v \geq \sfj +2.
      \end{array}
    \right.,
  \end{align*}
  where we let $\sfj = n-k+j$.
  Note that $D^{k} \subseteq D_0^k \subseteq \cdots \subseteq D_{N}^k \subseteq D^{k+1}$, see Figure~\ref{fig:Dk_classical} for an illustration.
  Let $\omega_j^k = \Psi_j(\omega^k)$.
  If $\gamma$ is a $\omega^k_j$-open path living in $D^k_j$, then $\Sigma_{t_{n-k}, t_{N+j+1}} (\gamma)$ is a $\omega^k_{j+1}$-open path living in $D^k_{j+1}$.
  This is a consequence after a careful inspection of Figure~\ref{fig:path_transformations}, where blue points indicate possible horizontal drifts.
  Define also
  \begin{align*}
    h_j^k = \sup \{ h \leq N: \exists u, v \in \bbZ \text{ with } x_{u, 0} \xleftrightarrow{D_j^k, \omega_j^k} x_{v, h} \}.
  \end{align*}
  Then, $h^k \leq h_0^{k}$ and $h_n^{k} \leq h^{k+1}$.
  As in~\cite{GriMan14}, we need to prove that for $0 \leq j \leq N-1$,
  \begin{align}
    & h_{j+1}^k = h_j^k & \text{ if } h_j^k \neq \sfj, \sfj +1, \label{eq:hjk1} \\
    & h_{j+1}^k - h_j^k = 0 \text{ or } 1 & \text{ if } h_j^k = \sfj, \label{eq:hjk2} \\
    & h_{j+1}^k - h_j^k = -1 \text{ or } 0 & \text{ if } h_j^k = \sfj +1, \label{eq:hjk3} \\
    & \bbP(h_{j+1}^k \geq h \,|\, h_j^k = h) \geq 2 \delta & \text{ if } h = \sfj +1. \label{eq:hjk4}
  \end{align}
  The four equations above imply the existence of a process $H^k$ as in~\eqref{eq:distribution}.

  As explained in~\cite{GriMan14}, \eqref{eq:hjk1}, \eqref{eq:hjk2} and \eqref{eq:hjk3} are clear 
  because the upper endpoint of a path is affected by ${\mathsf \Sigma} := \Sigma_{t_{n-k}, t_{N+j+1}}$ only if it is at height $\sfj$ or $\sfj+1$.
  The behavior of the upper endpoint can be analyzed using Figure~\ref{fig:path_transfo_ep}.
  More precisely,
  \begin{itemize}
  \item when it is at height $\sfj +1$, the upper endpoint either stays at the same level or drifts downwards by 1;
  \item when it is at height $\sfj$, it either stays at the same level or drifts upwards by 1.
  \end{itemize}
  Hence, the rest of the proof is dedicated to showing~\eqref{eq:hjk4}.

  We start with a preliminary computation.
  Fix $j$ and let $\calP_j$ be the set of paths $\gamma$ of $\Psi_j(G^{k})$, contained in $D_j^k$, 
  with one endpoint at height $0$, the other at height $h(\gamma)$, and all other vertices with heights between $1$ and $h(\gamma)-1$.

  Assume that in ${\mathsf \Sigma}$, the additional rhombus is slid from left to right and define $\Gamma$ to be the left-most path of $\calP_j$ reaching height $h_j^k$~\footnote{Otherwise $\Gamma$ should be taken right-most.}.
  (Such a path exists due to the definition of $h_j^k$.)
  This choice is relevant since later on, we will need to use negative information in the region on the left of the path $\gamma$.
  Moreover, for $\gamma, \gamma' \in \calP_j$, we write $\gamma' < \gamma$ if $\gamma' \neq \gamma$, $h(\gamma') = h(\gamma)$ and $\gamma'$ does not contain any edge strictly to the right of $\gamma$.

  Denote by $\Gamma = \Gamma(\omega^k_j)$ the $\omega^k_j$-open path of $\calP_j$ that is the minimal element of $\{ \gamma \in \calP_j : h(\gamma) = h^k_j, \gamma \text{ is } \omega^k_j \text{-open} \}$.
  Given a path $\gamma \in \calP_j$, we can write $\{ \Gamma = \gamma \} = \{ \gamma \text{ is } \omega_j^k \text{-open} \} \cap N_\gamma$ where $N_\gamma$ is the decreasing event that
  \begin{enumerate}
  \item[(a)] there is no $\gamma' \in \calP_j$ with $h(\gamma') > h(\gamma)$, all of whose edges not belonging to $\gamma$ are $\omega_j^k$-open;
  \item[(b)] there is no $\gamma' < \gamma$ with $h(\gamma') = h(\gamma)$, all of whose edges not belonging to $\gamma$ are $\omega_j^k$-open.
  \end{enumerate}

  Let $F$ be a set of edges disjoint from $\gamma$, write $C_F$ for the event that all the edges in $F$ are closed.
  We find,
  \begin{align}
    \bbP[ C_F \,|\, \Gamma = \gamma ]
    & = \frac{\bbP[ N_\gamma \cap C_F \,|\, \gamma \text{ is open} ]}{\bbP[ N_\gamma \,|\, \gamma \text{ is open}]} \nonumber \\
    & \geq \bbP[C_F \,|\, \gamma \text{ is open} ] \nonumber \\
    & \geq \phi_K^1[C_F], \label{eq:CF_proba}
  \end{align}
  where the second line is given by the FKG inequality due to the fact that $\bbP[ \cdot \,|\, \gamma \text{ is open}]$ is still a random-cluster measure and both $N_\gamma$ and $C_F$ are decreasing events.
  In the last line, we compare the boundary conditions, where $K$ is the subgraph consisting of rhombi containing the edges of $F$.

  \begin{figure}[htb]
    \centering
    \includegraphics[scale=1, page=2]{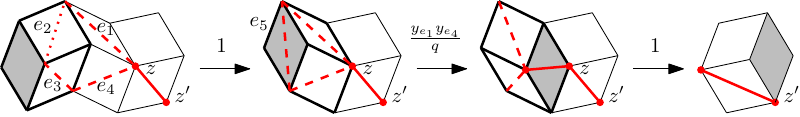}
    \caption{Three \stts contributing to ${\mathsf \Sigma}$ slid the gray rhombus from left to right.
      The dashed edges are closed, the bold edges are open and the state of dotted edge $e_2$ does not really matter.
      The first and last passages occur with probability $1$, and the second with probability $y_{e_1} y_{e_4} / q$.}
    \label{fig:non_decrease}
  \end{figure}

  Now, we are ready to show~\eqref{eq:hjk4}.
  Let $\gamma \in \calP_j$ with $h(\gamma) = \sfj+1$ and assume $\Gamma(\omega_j^k) = \gamma$.
  Now, it is enough to show that
  \begin{align} \label{eq:proba_growth_h}
    \bbP \big[ h ({\mathsf \Sigma} (\gamma) ) \geq \sfj+1 \,\big|\, \Gamma = \gamma \big] \geq 2 \delta.
  \end{align}

  Let $z = x_{u, \sfj+1}$ denote the upper endpoint of $\gamma$ and let $z'$ denote the other endpoint of the unique edge of $\gamma$ leading to $z$.
  Either $z' = x_{u+1, \sfj}$ or $z' = x_{u-1, \sfj}$.
  In the second case, it is always the case that $h({\mathsf \Sigma}(\gamma) ) \geq \sfj+1$.

  Assume that $z' = x_{u+1, \sfj}$ as in Figure~\ref{fig:non_decrease} and consider edges $e_i$ for $i = 1, \dots, 4$ as follow,
  \begin{align*}
    e_1 = \langle x_{u, \sfj+1}, x_{u-1, \sfj+2} \rangle, & \quad
    e_2 = \langle x_{u-1, \sfj+2}, x_{u-2, \sfj+1} \rangle, \\
    e_3 = \langle x_{u-2, \sfj+1}, x_{u-1, \sfj} \rangle, & \quad
    e_4 = \langle x_{u-1, \sfj}, x_{u, \sfj+1} \rangle.
  \end{align*}
  Let us now analyse the \stts that affect $e_1, \dots, e_4$; these are depicted in Figure~\ref{fig:non_decrease}.
  We note that conditioning on the event $C_F \cap \{ \Gamma = \gamma \}$, where $F = \{ e_3, e_4 \}$, we have:
  \begin{enumerate}
  \item[(a)] The edge $e_1$ must be closed due to the conditioning $\{ \Gamma = \gamma \}$.
  \item[(b)] Whichever the state of $e_2$ is, the edge $e_5$ is always closed.
  \item[(c)] The second passage occurs with probability $y_{e_1} y_{e_4} / q$.
  \item[(d)] The third passage is deterministic.
  \end{enumerate}
  Thus,
  \begin{align*}
    \bbP \big[ h ({\mathsf \Sigma}) \geq \sfj+1 \,\big|\, \Gamma = \gamma \big] &
    \geq \frac{y_{e_1} y_{e_4}}{q} \cdot \bbP [ C_F \,|\, \Gamma = \gamma ].
  \end{align*}
  Moreover, the preliminary computation~\eqref{eq:CF_proba} gives that
  \begin{align*}
    \bbP [ C_F \,|\, \Gamma = \gamma ] 
    \geq \rcisolaw{}{K}{1} [ C_F ]
    = (1-p_{e_3}) (1-p_{e_4}),
  \end{align*}
  where $K$ consists only of the two rhombi containing $e_3$ and $e_4$ and we use the fact that in the \rcm $\rcisolaw{}{K}{1}$, these edges are independent (the number of clusters is always equal to 1).
  Finally,
  \begin{align} \label{eq:path_CF_proba}
    \bbP \big[ h ({\mathsf \Sigma}) \geq \sfj+1 \,\big|\, \Gamma = \gamma \big] &
    \geq \frac{y_{e_1} p_{e_4} (1-p_{e_3})}{q}
    \geq \frac{y_{\pi-\eps} p_{\pi-\eps} (1-p_{\eps})}{q} \geq 2 \delta,
  \end{align}
  where 
  \begin{align} \label{eq:est_delta}
    \delta = \frac12 \min \left\{ \frac{y_{\pi-\eps} p_{\pi-\eps} (1-p_{\eps})}{q}, 1 \right\} > 0.
  \end{align}
  To conclude, we have
  \begin{align*}
    \frac{ \bbP [ h^{n+1} \geq \delta n ] }{ \bbP [ h^0 \geq n ] }
    \geq \frac{ \bbP [ H^{n+1} \geq \delta n ] }{ \bbP [ H^0 \geq n ] }
    \geq \bbP [ H^{n+1} \geq \delta n \,|\, H^0 \geq n ] =: c_n(\delta),
  \end{align*}
  and since $H^n / n \rightarrow 2 \delta$ as $n \rightarrow \infty$ due to the law of large numbers, we know that $c_n \rightarrow 1$ as $n \rightarrow \infty$.
\end{proof}

\subsection{Doubly-periodic isoradial graphs} \label{sec:biperiodic_iso}

Now that the RSW property  is proved for isoradial square lattices,
we transfer it to arbitrary doubly-periodic isoradial graphs $\bbG$.
We do this by transforming a finite part of $\bbG$ (as large as we want) into a local isoradial square lattice using \stts.
The approach is based on the combinatorial result Proposition~\ref{prop:bi_to_square}.

\begin{prop} \label{prop:biperiodic_RSW}
  Any doubly-periodic isoradial graph $\bbG$ satisfies the RSW property.
\end{prop}

\begin{proof}
  Let $\bbG$ be a doubly-periodic isoradial graph with grid $(s_n)_{n \in \bbZ}$ and $(t_n)_{n \in \bbZ}$.
  Fix a constant $d > 1$ as given by Proposition~\ref{prop:bi_to_square} applied to $\bbG$. 
  In the below formula, $\calC_h^{\hp}(n; n)$ denotes the horizontal crossing event in the half-plane rectangular domain $\HR(n; n) := \rect(-n, n; 0, n)$.
  We will show that
  \begin{align}
    \rcisolaw{}{\sq (6dn)}{0} \big[ \calC_h^{\hp} (n; n) \big] =
    \rcisolaw{}{\sq (6dn)}{0} \big[ \calC_h (-n, n; 0, n) \big] \geq \delta, \label{eq:biperiodic_RSW2}
  \end{align}
  for some constant $\delta > 0$ which does not depend on $n$.
  Moreover, a careful inspection of the forthcoming proof shows that $\delta$ only depends the bounded angles parameter $\eps >0$ and on the size of the fundamental domain of $\bbG$.
  The same estimate is valid for the dual model, since it is also a random-cluster model on an isoradial graph with $\beta = 1$.

  The two families of tracks $(s_n)_{n \in \bbZ}$ and $(t_n)_{n \in \bbZ}$ play symmetric roles, therefore~\eqref{eq:biperiodic_RSW2} may also be written 
  \begin{align} 
    \rcisolaw{}{\sq (6dn)}{0} \big[ \calC_v (0, n; -n, n) \big] \geq \delta. \label{eq:biperiodic_RSW3}
  \end{align}
  The two inequalities~\eqref{eq:biperiodic_RSW2} and~\eqref{eq:biperiodic_RSW3} together with their dual counterparts imply the RSW property by Lemma~\ref{lem:equiv_RSW}~\footnote{The conditions of Lemma~\ref{lem:equiv_RSW} differ slightly from \eqref{eq:biperiodic_RSW2} and \eqref{eq:biperiodic_RSW3} in the position of the rectangle and the domain where the measure is defined. Getting from one to the other is a standard application of the comparison between boundary conditions.}.

  \medskip

  The rest of the proof is dedicated to~\eqref{eq:biperiodic_RSW2}.
  In proving~\eqref{eq:biperiodic_RSW2}, we will assume $n$ to be larger than some threshold depending on $\bbG$ only; 
  this is not a restrictive hypothesis.

  Let $(\sigma_k)_{1 \leq k \leq K}$ be a sequence of \stts as in Proposition~\ref{prop:bi_to_square} such that in $\tilde \bbG := (\sigma_K \circ \cdots \circ \sigma_1) (\bbG)$, the region enclosed by $s_{-4n}$, $s_{4n}$, $t_{-2n}$ and $t_{2n}$ has a square lattice structure. 
  Recall that all the transformations $\sigma_k$ act horizontally between $s_{-6dn}$ and $s_{6dn}$ and vertically between $t_{n}$ and $t_{-n}$. 


  Consider the following events for $\tilde \bbG$. 
  Let $\tilde \calC$ be the event that there exists an open circuit contained in the region 
  between $s_{-2n}$ and $s_{2n}$ and between $t_{-n/2}$ and $t_{n/2}$ 
  surrounding the segment of the base between $s_{-n}$ and $s_n$. 
  Let $\tilde \calC^*$ be the event that there exists an open circuit contained in the region 
  between $s_{-3n}$ and $s_{3n}$ and between $t_{-3n/2}$ and $t_{3n/2}$ 
  surrounding the segment of the base between $s_{-2n}$ and $s_{2n}$. 

  %
  Let $\tilde G$ be the subgraph of $\tilde \bbG$ contained between $s_{-4n}$, $s_{4n}$, $t_{-2n}$ and $t_{2n}$. 
  Then $\tilde G$ is a finite section of a square lattice with $4n+1$ horizontal tracks, but potentially more than $8n+1$ vertical ones. 
  Indeed, any track of $\bbG$ that intersects the base between $s_{-4n}$, $s_{4n}$ is transformed into a vertical track of $\tilde G$. 

  Write $(\tilde s_n)_{n \in \bbZ}$ for the vertical tracks of $\tilde G$, with $\tilde s_0$ coinciding with $s_0$ 
  (this is coherent with the notation in the proof of Proposition~\ref{prop:bi_to_square}). 
  Then, the tracks $(s_n)_{n \in \bbZ}$ are a periodic subset of $(\tilde s_n)_{n \in \bbZ}$, with period bounded by the number of tracks intersecting a fundamental domain of $\bbG$.
  It follows that there exist constants $a, b$ depending only on $\bbG$, not on $n$, such that the number of tracks $(\tilde s_n)_{n \in \bbZ}$ between any two tracks $s_i$ and $s_j$ (with $i \leq j$) is between $(j-i)a -b$ and $(j-i)a + b$. 

  By the above discussion, for some constant $c > 1$ and $n$ large enough 
  (larger than some $n_0$ depending only on $a$ and $b$, therefore only on the size of the fundamental domain of $\bbG$),
  the events $\tilde \calC$ and $\tilde \calC^*$ may be created using crossing events as follow:
  \begin{align*}
    \tilde \calH_1 \cap \tilde \calH_2 \cap \tilde \calV_1 \cap \tilde \calV_2 & \subseteq \tilde \calC & \text{and} &
    & \tilde \calH_1^* \cap \tilde \calH_2^* \cap \tilde \calV_1^* \cap \tilde \calV_2^* & \subseteq \tilde \calC^*,
  \end{align*}
  where 
  \begin{align*}
    \tilde \calH_1 & = \calC_h(-(c+1)n, (c+1)n; 0, \tfrac{n}{2} ), & \tilde \calH_1^* & = \calC_h^*(-(2c+1)n, (2c+1)n; n, \tfrac{3n}{2}), \\
    \tilde \calH_2 & = \calC_h(-(c+1)n, (c+1)n; -\tfrac{n}{2}, 0 ), & \tilde \calH_2^* & = \calC_h^*(-(2c+1)n, (2c+1)n; -\tfrac{3n}{2}, -n), \\
    \tilde \calV_1 & = \calC_v(-(c+1)n, -cn; -\tfrac{n}{2}, \tfrac{n}{2} ), & \tilde \calV_1^* & = \calC_v^*(-(2c+1)n, -2cn; -\tfrac{3n}{2}, \tfrac{3n}{2}), \\
    \tilde \calV_2 & = \calC_v(cn, (c+1)n; -\tfrac{n}{2}, \tfrac{n}{2} ), & \tilde \calV_2^* & = \calC_v^*(2cn, (2c+1)n; -\tfrac{3n}{2}, \tfrac{3n}{2}),
  \end{align*}
  are defined in terms of the tracks $(\tilde s_n)_{n \in \bbZ}$ and $(t_n)_{n \in \bbZ}$.
  These horizontal and vertical crossing events are shown in Figure~\ref{fig:biperiodic_RSW}.

  \begin{figure}[htb]
    \centering
    \includegraphics[width=0.6\textwidth]{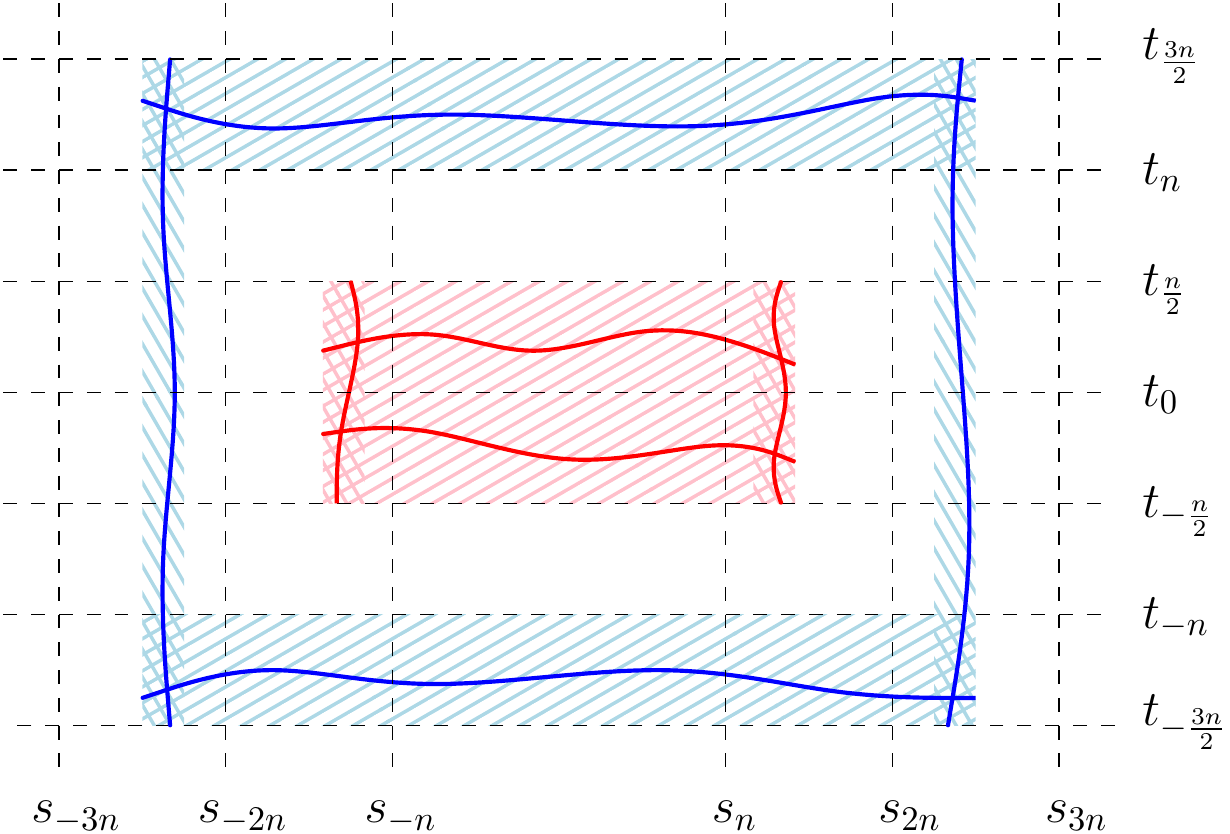}
    \caption{The crossing events $\tilde \calH_1$, $\tilde \calH_2$, $\tilde \calV_1$ and $\tilde \calV_2$ are depicted in red; 
      they induce a circuit around the segment of the base between $s_{-n}$ and $s_{n}$. 
      The events $\tilde \calH_1^*$, $\tilde \calH_2^*$, $\tilde \calV_1^*$ and $\tilde \calV_2^*$ are represented in blue.}
    \label{fig:biperiodic_RSW}
  \end{figure}

  Notice that all events above depend only on the configuration in $\tilde G$.
  Let $\phi_{\tilde \bbG}$ denote some infinite-volume measure on $\tilde \bbG$.
  By the RSW property for the square lattice $\tilde G$ (that is, by Corollary~\ref{cor:sq_RSW}), the comparison between boundary conditions and the FKG inequality, 
  \begin{align*}
    \phi_{\tilde \bbG} \big[ \tilde \calC^* \big]
    \geq \phi_{\tilde G}^1 \big[ \tilde \calH_1^* \big] 
    \phi_{\tilde G}^1 \big[ \tilde \calH_2^* \big] 
    \phi_{\tilde G}^1 \big[ \tilde \calV_1^* \big] 
    \phi_{\tilde G}^1 \big[ \tilde \calV_2^* \big] \geq \delta_1,
  \end{align*}
  for some $\delta_1 > 0$ independent of $n$. 
  Moreover, by the same reasoning, 
  \begin{align*}
    \phi_{\tilde \bbG} \big[ \tilde \calC \,\big|\, \tilde \calC^* \big] \geq
    \phi_{\tilde \rect}^0 \big[ \tilde \calH_1 \big]
    \phi_{\tilde \rect}^0 \big[ \tilde \calH_2 \big]
    \phi_{\tilde \rect}^0 \big[ \tilde \calV_1 \big]
    \phi_{\tilde \rect}^0 \big[ \tilde \calV_2 \big] \geq \delta_2,
  \end{align*}
  for some $\delta_2 > 0$ independent of $n$, where $\tilde \rect = \tilde \rect(2cn; n)$ is defined with respect to the tracks $(\tilde s_n)_{n \in \bbZ}$ and $(t_n)_{n \in \bbZ}$.
  We conclude that 
  \begin{align*}
    \phi_{\tilde \bbG} \big[ \tilde \calC \cap \tilde \calC^* \big] \geq \delta_1 \delta_2 > 0.
  \end{align*}

  Let $\bbP$ be the probability that consists of choosing a configuration $\tilde \omega$ on $\tilde \bbG$ according to $\phi_{\tilde \bbG}$, then applying the inverse transformations $\sigma_K^{-1},\dots, \sigma_1^{-1}$ to it.
  Thus, $\omega := (\sigma_1^{-1}\circ\dots\circ \sigma_K^{-1})(\tilde\omega)$ is a configuration on $\bbG$ chosen according to some infinite-volume measure~$\phi_\bbG$. 

  Let $\tilde \omega \in \tilde \calC \cap \tilde \calC^*$, and write $\tilde \gamma$ and $\tilde \gamma^*$ for two circuits as in the definitions of $\tilde \calC$ and $\tilde \calC^*$ respectively.
  The two circuits $\tilde \gamma$ and $\tilde \gamma^*$ are transformed by $(\sigma_1^{-1} \circ \dots \circ \sigma_K^{-1})$ into circuits on $\bbG$; 
  call $\gamma = (\sigma_1^{-1}\circ\dots\circ \sigma_K^{-1}) (\tilde \gamma)$ and  $\gamma^* = (\sigma_1^{-1} \circ \dots \circ \sigma_K^{-1}) (\tilde \gamma^*)$ their respective images. 
  Then, $\gamma$ is $\omega$-open and $\gamma^*$ is $\omega^*$-open. 

  Since the transformations $\sigma_1^{-1}, \dots, \sigma_K^{-1}$ only affect the region between $s_{-6dn}$, $s_{6dn}$, $t_{-2n}$ and $t_{2n}$, both $\gamma$ and $\gamma^*$ are contained in this region of $\bbG$, that is in $\rect(6dn; 2n)$.
  Additionally, since the transformations do not affect the base, $\gamma^*$ surrounds the segment of the base between $s_{-2n}$ and $s_{2n}$ 
  while $\gamma$ surrounds the segment of the base between $s_{-n}$ and $s_{n}$ but only traverses the base 
  between $s_{-2n}$ and $s_{2n}$.

  Write $\calC$ for the event that a configuration on $\bbG$ has an open circuit contained in $\rect(6dn; 2n)$, 
  surrounding the segment of the base between $s_{-n}$ and $s_n$ and traversing the base only between $s_{-2n}$ and $s_{2n}$. 
  Also, set $\calC^*$ to be the event that a configuration on $\bbG$ has a dually-open circuit contained in $\rect(6dn; 2n)$, 
  surrounding the segment of the base between $s_{-2n}$ and $s_{2n}$.

  Both $\calC$ and $\calC^*$ are reminiscent of the events $\tilde \calC$ and $\tilde \calC^*$, in spite of small differences. 
  Indeed, the discussion above shows that if $\tilde \omega \in \tilde \calC \cap \tilde \calC^*$, then $\omega \in \calC \cap \calC^*$.
  Thus,
  \begin{align*}
    \phi_\bbG \big[ \calC \cap \calC^* \big] 
    = \bbP \big[ \omega \in \calC \cap \calC^* \big]
    \geq \bbP \big[ \tilde \omega \in \tilde \calC \cap \tilde \calC^* \big]
    = \phi_{\tilde \bbG} \big[ \tilde \calC \cap \tilde \calC^* \big]
    \geq \delta_1\delta_2.
  \end{align*}

  For a configuration $\omega$ on $\bbG$, write $\Gamma^*(\omega)$ for the exterior-most dually-open circuit as in the definition of $\calC^*$
  (that is contained in $\rect(6dn; 2n)$ and surrounding the segment of the base between $s_{-2n}$ and $s_{2n}$), if such a circuit exists. 
  Let $\Int(\Gamma^*)$ be the region surrounded by $\Gamma^*$, seen as a subgraph of $\bbG$. 

  It is standard that $\Gamma^*$ may be explored from the outside and therefore that, conditionally on $\Gamma^*$, 
  the \rcm in $\Int(\Gamma^*)$ is $\phi_{\Int(\Gamma^*)}^0$. 

  Observe that for $\omega \in \calC \cap \calC^*$, due to the restrictions over the intersections with the base, 
  any circuit in the definition of $\calC$ is surrounded by any in the definition of $\calC^*$. 
  Thus, if for $\omega \in \calC^*$, the occurence of $\calC$ only depends on the configuration inside $\Int(\Gamma^*)$.
  Therefore,
  \begin{align*}
    \phi_\bbG \big[ \calC \cap \calC^* \big]
    & = \phi_\bbG \big[ \calC \,\big|\, \calC^* \big] \phi_\bbG \big[ \calC^* \big] \\
    & = \sum_{\gamma^*} \phi_\bbG \big[ \calC \,\big|\, \Gamma^* = \gamma^* \big] \phi_\bbG \big[ \Gamma^* = \gamma^* \big] \\
    & = \sum_{\gamma^*} \phi_{\Int(\gamma^*)}^0 \big[ \calC \big] \phi_\bbG \big[ \Gamma^* = \gamma^* \big] \\
    & \leq \sum_{\gamma^*}  \phi_{\rect(6dn; 2n)}^0 \big[ \calC \big] \phi_\bbG \big[ \Gamma^* = \gamma^* \big] \\
    & = \phi_{\rect(6dn; 2n)}^0 \big[ \calC \big] \phi_\bbG \big[ \calC^* \big],
  \end{align*}
  where the sum above is over all deterministic circuits $\gamma^*$ on $\bbG^*$, as in the definition of $\calC^*$.
  In the before last line, we used the fact that $\Int(\gamma^*) \subseteq \rect(6dn; 2n)$,
  where $\rect(6dn; 2n)$ is defined using tracks in $\bbG$,
  and the comparison between boundary conditions to say that the free boundary conditions on $\pd \Int (\gamma^*)$ are less favorable to the increasing event $\calC$ than those on the more distant boundary $\pd \rect(6dn; 2n)$.

  Due to the previous bound on $\phi_\bbG \big[ \calC \cap \calC^* \big]$, we deduce that 
  $$
  \phi_{\rect(6dn; 2n)}^0 \big[ \calC \big]  \geq \delta_1 \delta_2.
  $$
  Finally, notice that any circuit as in the definition of $\calC$ contains a horizontal crossing of $\HR(n; n)$. 
  We conclude from the above that 
  \begin{align*}
    \phi_{\rect(6dn; 2n)}^0 \big[ \HR(n; n) \big]  \geq \delta_1 \delta_2.
  \end{align*}
  This implies~\eqref{eq:biperiodic_RSW2} by further pushing away the unfavorable boundary conditions. 
\end{proof}

\subsection{Tying up loose ends} \label{sec:conclusion_q<=4}

As mentioned already, Theorem~\ref{thm:main} and Corollary~\ref{cor:phase_transition} for $1 \leq q \leq 4$
follow directly from the RSW property (i.e., from Proposition~\ref{prop:biperiodic_RSW}). We mention here the necessary steps. 
They are all standard for those familiar with the random-cluster model; detail are provided in~\cite[App.~C]{Li-thesis}.

Fix $\bbG$ a doubly-periodic isoradial graph and $q \in [1,4]$.  
We start with the following lemma which is the key to all the proofs.

\begin{lem} \label{lem:annulus_proba}
  For $j \geq 1$, define the annuli $A_j = [-2^{j+1},2^{j+1}]^2 \setminus [-2^{j},2^{j}]^2$.
  Then, there exists $c > 0$ such that for all $j \geq 1$ and $\xi = 0, 1$, we have
  \begin{align}\label{eq:annulus_proba}
    \rcisolaw{}{A_j}{\xi} \big[ \text{there exists an open circuit surrounding 0 in } A_j \big] \geq c.
  \end{align}
  By duality, the same also holds for a dually-open circuit.
\end{lem}

\begin{proof}
  This is proved by combining crossings of rectangles via the FKG inequality, 
  as in Figure~\ref{fig:biperiodic_RSW}.
\end{proof}

The estimates of the Lemma~\ref{lem:annulus_proba} for the dual model 
imply an upper bound on the one-arm probability under $\phi^1_{\bbG}$, 
as that in the second point of Theorem~\ref{thm:main}.
Indeed, if a dually-open circuit occurs in $A_j$ for some $j \leq \log_2 n -2$, then the event $\{ 0 \leftrightarrow \partial B_n \}$ fails. 
The fact \eqref{eq:annulus_proba} is uniform in the boundary conditions on $A_j$ allows us to ``decouple'' the events of \eqref{eq:annulus_proba}, and proves that the probability of no circuit occurring in any of $A_1, \dots, A_{\log_2 n -2}$ is bounded above by $(1-c)^{\log_2 n - 2}$.

The converse bound is obtained by a straightforward construction of a large cluster using crossings of rectangles of the form $[0,2^j] \times [0,2^{j+1}]$ and their rotation by $\frac\pi2$, combined using the FKG inequality.

From the above, we deduce that $\phi^1_{\bbG}(0\leftrightarrow \infty) = 0$. 
The uniqueness of the critical infinite volume measure (the first point of Theorem~\ref{thm:main}) follows using a standard coupling argument. 

Finally, to prove Corollary~\ref{cor:phase_transition}, we use the differential inequality of \cite{GraGri11}, as done in~\cite{DumMan16}.

\subsection{Universality of arm exponents: Theorem~\ref{thm:universality}}

%

The proof of universality of arm exponents (Theorem~\ref{thm:universality}) follows exactly the steps of~\cite[Sec.~8]{GriMan14}.
Arm events will be transferred between isoradial graphs using the same transformations as in the previous sections.
As already discussed in Section~\ref{sec:star-triangle}, these transformations alter primal and dual paths, especially at their endpoints. 
When applied to arm events, this could considerably reduce the length of the arms. 
To circumvent such problems and shield the endpoints of the arms from the effect of the \stts, 
we define a variation of the arm events.
It roughly consists in ``attaching'' the endpoints of the arms to a track which is not affected by the transformations. 
Some notation is necessary.

Fix $\eps > 0$ and a doubly-periodic isoradial graph $\bbG \in \calG(\eps)$ with grid $(s_n)_{n \in \bbZ}$ and $(t_n)_{n \in \bbZ}$.
Recall that the vertices of $\bbG^\diamond$ that are below and adjacent to $t_0$ form the base of $\bbG$. 
Also, recall the notation $x \leftrightarrow y$ and write $x \xleftrightarrow{*} y$ for connections in the dual configuration.

For $n < N$ and $k \in \{1\} \cup 2\bbN$, define the event $\tilde A_k(n, N)$ as
\begin{enumerate}
\item for $k=1$: there exist primal vertices $x_1 \in \sq(n)$ and $y_1 \notin \sq(N)$, 
  both on the base, such that $x_1 \leftrightarrow y_1$;
\item for $k=2$: there exist $x_1, x_1^* \in \sq(n)$ and $y_1, y_1^* \notin \sq(N)$, 
  all on the base, such that $x_1 \leftrightarrow y_1$ and $x_1^* \xleftrightarrow{*} y_1^*$;
\item for $k = 2j \geq 4$: $\tilde A_k(n, N)$ is the event that there exist 
  $x_1, \dots, x_j \in \sq(n)$ and $y_1, \dots, y_j \notin \sq(N)$, all on the base, such that
  $x_i \leftrightarrow y_i$ for all $i$ and $x_i \nlra x_{j}$ for all $i \neq j$.  
\end{enumerate}

Notice the resemblance between $\tilde A_k(n, N)$ and $A_k(n, N)$, where the latter is defined just before the statement of Theorem~\ref{thm:universality}. 
In particular, observe that in the third point, the existence of $j$ disjoint clusters uniting $\pd \sq(n)$ to $\pd \sq(N)$ 
indeed induces $2j$ arms of alternating colours in counterclockwise order. 
Two differences between $\tilde A_k(n, N)$  and $A_k(n, N)$ should be noted: the fact that in the former arms are forced to have extremities on the base 
and that the former is defined in terms or graph distance while the latter in terms of Euclidean distance. 
As readers probably expect, this has only a limited impact on the probability of such events. 

For the rest of the section, fix $q \in [1,4]$ and write $\rcisolaw{}{\bbG}{}$ for the unique infinite-volume random-cluster  measure on $\bbG$ with parameters $\beta=1$ and $q$.

\begin{lem}\label{lem:equiv_arm_events}
  Fix  $k \in \{1\} \cup 2\bbN$. 
  There exists $c > 0$ depending only on $\eps$, $q$, $k$ and the fundamental domain of $\bbG$ such that
  \begin{align} \label{eq:equiv_arm_events}
    c \, \rcisolaw{}{\bbG}{} [ A_k(n, N) ]
    \leq \rcisolaw{}{\bbG}{} [ \tilde A_k(n, N) ]
    \leq c^{-1} \rcisolaw{}{\bbG}{} [ A_k(n, N) ]
  \end{align}
  for all $N > n$ large enough.
\end{lem}

The above is a standard consequence of what is known in the field as the arm separation lemma~\cite[Lem.~D.1]{Li-thesis}. 
The proofs of the separation lemma and Lemma~\ref{lem:equiv_arm_events} are both fairly standard but lengthy applications of the RSW theory of Theorem~\ref{thm:main}; they are discussed in~\cite[App.~D]{Li-thesis} (see also~\cite[Prop.~5.4.2]{Manolescu-thesis} for a version of these for Bernoulli percolation).

We obtain Theorem~\ref{thm:universality} in two steps, first for isoradial square lattices, then for doubly-periodic isoradial graphs. 
The key to the first step is the following proposition. 

\begin{prop} \label{prop:transfer_arms}
  Let $\graph 1 = \bbG_{\ba, \bb^{(1)}}$ and $\graph 2 = \bbG_{\ba, \bb^{(2)}}$ be two isoradial square lattices in $\calG(\eps)$.
  Fix $k \in \{1\}\cup2\bbN$. Then
  \begin{align*}
    \rcisolaw{}{\graph 1}{\xi} [ \tilde A_k(n, N) ] =  \rcisolaw{}{\graph 2}{\xi} [\tilde A_k(n, N) ], 
  \end{align*}
  for all $n < N$.
\end{prop}

\begin{proof}
  Fix $k\in \{1\}\cup2\bbN$ and take $N > n > 0$ and $M \geq N$ (one should imagine $M$ much larger than $N$).
  Let $G_\mix$ be the symmetric mixed graph of Section~\ref{sec:mix} 
  formed above the base of a block of $M$ rows and $2M+1$ columns of $\graph 2$ 
  superposed on an equal block of $\graph 1$, then convexified, and symmetrically in below the base. 
  Construct $\tilde G_\mix$ in the same way, with the role of $\graph 1$ and $\graph 2$ inversed. 
  Recall, from Section~\ref{sec:mix}, the series of \stts $\Sigma^\downarrow$ that transforms $G_\mix$ into $\tilde G_\mix$.
  
  Write $\phi_{G_\mix}$ and $\phi_{\tilde G_\mix}$ for the \rcms on $G_\mix$ and $\tilde G_\mix$, respectively, with $\beta = 1$ and free boundary conditions. 
  The events $\tilde A_k(n, N)$ are also defined on $G_\mix$ and $\tilde G_\mix$. 
  
  Let $\omega$ be a configuration on $G_\mix $ such that $\tilde A_k(n, N)$ occurs.
  Then, under the configuration $\Sigma^\downarrow(\omega)$ on $\tilde G_\mix$, $\tilde A_k(n, N)$ also occurs.
  Indeed, the vertices $x_i$ and $y_i$ (and $x_1^*$ and $y_1^*$ when $k=2$) are not affected by the \stts in $\Sigma^{\downarrow}$ 
  and connections between them are not broken nor created by any \stt. 
  Thus, $\phi_{G_\mix}[ \tilde A_k(n, N) ] \leq \phi_{\tilde G_\mix} [ \tilde A_k(n, N) ]$. Since the roles of $G_\mix$ and $\tilde G_\mix$ are symmetric, we find
  \begin{align}\label{eq:tilde_A_k_equal1}
    \phi_{G_\mix}[ \tilde A_k(n, N) ] = \phi_{\tilde G_\mix} [ \tilde A_k(n, N) ]
  \end{align}
  Observe that the quantities in \eqref{eq:tilde_A_k_equal1} depend implicitly on $M$. 
  When taking $M \rightarrow \infty$, due to the uniqueness of the infinite-volume \rcms in $\graph 1$ and $\graph 2$, we obtain
  \begin{align*}
    &\phi_{G_\mix}[ \tilde A_k(n, N) ] \xrightarrow[M\to \infty]{} \phi_{\graph 1}[ \tilde A_k(n, N) ] \qquad \text{and} \\
    &\phi_{\tilde G_\mix}[ \tilde A_k(n, N) ] \xrightarrow[M\to \infty]{} \phi_{\graph 2}[ \tilde A_k(n, N) ].
  \end{align*}
  Thus, \eqref{eq:tilde_A_k_equal1} implies the desired conclusion. 
\end{proof}

\begin{cor} \label{cor:transfer_arms2}
  Let $\bbG = \bbG_{\ba, \bb}$ be an isoradial square lattice in $\calG(\eps)$ and fix $k \in \{1\}\cup2\bbN$. 
  Then, there exists $c > 0$ depending only on $\eps$, $q$ and  $k$ such that, 
  \begin{align*}
    c\, \rcisolaw{}{\bbG}{} [ A_k(n, N) ]
    \leq \rcisolaw{}{\bbZ^2}{} [ A_k(n, N) ]
    \leq c^{-1} \rcisolaw{}{\bbG}{} [ A_k(n, N) ],
  \end{align*}
  for any $n < N$.
\end{cor}

\begin{proof}
  Fix $\bbG_{\ba, \bb}$ and $k$ as in the statement. The constants $c_i$ below depend on $\eps$, $q$ and  $k$ only.
  
  Let $\tilde \beta_k = \alpha_{-k} - \beta_0 + \pi$ and write $\tilde \bb$ for the sequence $(\tilde \beta_k)_{k \in \bbZ}$.
  Due to the choice of $\bbG_{\ba, \bb}$, we have $\tilde \bb \in [\eps, \pi-\eps]^\bbZ$.
  Proposition~\ref{prop:transfer_arms} and Lemma~\ref{lem:equiv_arm_events} applied to 
  $\bbZ^2 = \bbG_{0, \frac\pi2}$ and $\bbG_{0, \tilde\bb}$ yield a constant $c_1 > 0$ such that
  \begin{align} \label{eq:arm_equiv1}
    c_1\, \rcisolaw{}{\bbG_{0,\tilde \bb}}{} [ A_k(n, N) ]
    \leq \rcisolaw{}{\bbZ^2}{} [ A_k(n, N) ]
    \leq c_1^{-1} \rcisolaw{}{\bbG_{0,\tilde \bb}}{} [ A_k(n, N) ].
  \end{align}
  
  As in the proof of Corollary~\ref{cor:sq_RSW}, $\bbG_{\ba, \beta_0}$ is the rotation by $\beta_0$ of the graph $\bbG_{0, \tilde \bb}$.
  This does not imply that the arm events have the same probability in both graphs (since they are defined in terms of square annuli).
  However,~\cite[Prop.~D.2]{Li-thesis} about arms extension provides a constant $c_2 > 0$ such that 
  \begin{align} \label{eq:arm_equiv2}
    c_2\, \rcisolaw{}{\bbG_{\ba, \beta_0}}{} [ A_k(n, N) ]
    \leq \rcisolaw{}{\bbG_{0, \tilde \bb}}{} [ A_k(n, N) ]
    \leq c_2^{-1} \rcisolaw{}{\bbG_{\ba, \beta_0}}{} [ A_k(n, N) ].
  \end{align}
  
  Finally apply Proposition~\ref{prop:transfer_arms} and Lemma~\ref{lem:equiv_arm_events}
  to $\bbG_{\ba, \beta_0}$ and $\bbG_{\ba,\bb}$ to obtain a constant $c_3 > 0$ such that
  \begin{align} \label{eq:arm_equiv3}
    c_3\, \rcisolaw{}{\bbG_{\ba, \beta_0}}{} [ A_k(n, N) ]
    \leq \rcisolaw{}{\bbG_{\ba,\bb}}{} [ A_k(n, N) ]
    \leq c_3^{-1} \rcisolaw{}{\bbG_{\ba, \beta_0}}{} [ A_k(n, N) ].
  \end{align}
  
  Writing \eqref{eq:arm_equiv1}, \eqref{eq:arm_equiv2} and \eqref{eq:arm_equiv3} together yields the conclusion with $c =c_1c_2c_3$.
\end{proof}

Theorem~\ref{thm:universality} is now proved for isoradial square lattices. 
To conclude, we extend the result to all doubly-periodic isoradial graphs.

\begin{proof}[Proof of Theorem~\ref{thm:universality}]
  Consider a doubly-periodic graph $\bbG \in \calG(\eps)$ for some $\eps > 0$, 
  with grid $(s_n)_{n \in \bbZ}$ and $(t_n)_{n \in \bbZ}$.
  Fix $k\in \{1\} \cup 2\bbN$.
  The constants $c_i$ below depend on $\eps$, $q$, $k$ and the size of the period of $\bbG$.
  
  Choose $n < N$ and $M \geq N$ (one should think of $M$ as much larger than $N$).  
  Proposition~\ref{prop:bi_to_square} (the symmetrized version) provides \stts $(\sigma_k)_{1 \leq k \leq K}$ 
  such that, in $\tilde \bbG = (\sigma_K \circ \cdots \circ \sigma_1) (\bbG)$, the region $\sq(M)$ has a square lattice structure.
  Moreover, each $\sigma_k$ acts between $s_{-dM}$ and $s_{dM}$ (for some fixed $d \geq 1$) and between $t_M$ and $t_{-M}$, none of them affecting any rhombus of $t_0$.
  
  In a slight abuse of notation (since $(s_n)_{n \in \bbZ},(t_n)_{n \in \bbZ}$ is not formally a grid in $\tilde \bbG$)
  we define $\tilde A_k(n, N)$ for $\tilde \bbG$ as for $\bbG$.
  
  Let $\omega$ be a configuration on $\bbG$ such that $\tilde A_k(n, N)$ occurs.
  Then, the image configuration $(\sigma_K \circ \cdots \circ \sigma_1) (\omega)$ on $\tilde \bbG$ is such that $\tilde A_k(n, N)$ occurs.
  Indeed, the transformations do not affect the endpoints of any of the paths defining $\tilde A_k(n, N)$.
  Thus, 
  \begin{align*}
    \phi_\bbG [\tilde A_k(n, N)] \leq \phi_{\tilde \bbG} [\tilde A_k(n, N)].
  \end{align*}
  The transformations may be applied in reverse order to obtain the converse inequality. 
  In conclusion,
  \begin{align} \label{eq:arms_iso_to_sq1}
    \phi_\bbG [\tilde A_k(n, N)] = \phi_{\tilde \bbG} [\tilde A_k(n, N)].
  \end{align}
  
  The right-hand side of the above depends implicitly on $M$. 
  Write $\bbG^{\rmsq}$ for the isoradial square lattice such that the region $\sq(M)$ of $\tilde\bbG$ is a centered rectangle of $\bbG^{\rmsq}$.
  (It is easy to see that there exists a lattice that satisfies this condition for all $M$ simultaneously). 
  The vertical tracks $s_k$ of $\tilde\bbG$ correspond to vertical tracks in $\bbG^{\rmsq}$ with an index between $k$ and $dk$, 
  where $d$ is the maximal number of track intersection on $t_0$ between two consecutive tracks $s_j$, $s_{j+1}$ in $\bbG$.   
  
  In conclusion, taking $M \rightarrow \infty$ and using the uniqueness of the infinite-volume measure on $\bbG^{\rmsq}$, we find
  \begin{align} \label{eq:arms_iso_to_sq2}
    \rcisolaw{}{\bbG^{\rmsq}}{} [ \tilde A_k(n, d N) ] 
    \leq \lim_{M\to \infty}\rcisolaw{}{\tilde \bbG}{} [  \tilde  A_k(n, N) ] 
    \leq \rcisolaw{}{\bbG^{\rmsq}}{} [  \tilde  A_k(dn, N) ].
  \end{align}
  Due to Lemma~\ref{lem:equiv_arm_events} and to the extension of arms (\cite[Prop.~D.2]{Li-thesis}), 
  \begin{align*}
    c_1\rcisolaw{}{\bbG^{\rmsq}}{} [A_k(n, N) ]
    \leq \rcisolaw{}{\bbG^{\rmsq}}{} [ \tilde A_k(n, d N) ] 
    \leq \rcisolaw{}{\bbG^{\rmsq}}{} [\tilde  A_k(dn, N) ] 
    \leq c_1^{-1}\rcisolaw{}{\bbG^{\rmsq}}{} [A_k(n, N) ],
  \end{align*}
  for some constant $c_1>0$. 
  Using this,~\eqref{eq:arms_iso_to_sq2} and Lemma~\ref{lem:equiv_arm_events}, we find 
  \begin{align*}
    c_2\rcisolaw{}{\bbG^{\rmsq}}{} [A_k(n, N) ]
    \leq \phi_\bbG [ A_k(n, N)]	
    \leq c_2^{-1}\rcisolaw{}{\bbG^{\rmsq}}{} [A_k(n, N)],
  \end{align*}
  for some $c_2 > 0$. 
  Using Corollary~\eqref{cor:transfer_arms2}, we obtain the desired result.
\end{proof}

\section{Proofs for $q > 4$} \label{sec:q>4}

Fix $q > 4$ and $\bbG$ a doubly-periodic isoradial graph with grid $(s_n)_{n\in\bbZ}$, $(t_n)_{n\in\bbZ}$.
Unless otherwise stated, write $\phi_\bbG^\xi$ for the isoradial \rcm on $\bbG$ with parameters $q$, $\beta = 1$ and free ($\xi = 0$) or wired ($\xi = 1$) boundary conditions.
We will use the same notation as in Sections~\ref{sec:q<=4_notation} and~\ref{sec:alternative_RSW}.

The main goal of this section is to prove that there exist constants $C, c > 0$ such that 
\begin{align} \label{eq:exp_decay2}
  \phi_{\bbG}^{0} \big[ 0 \leftrightarrow \pd \sq (n) \big] \leq C \exp(-cn), \qquad \forall n \geq 1.
\end{align}
As we will see in Section~\ref{sec:conclusions_q>4}, Theorem~\ref{thm:main2} and Corollary~\ref{cor:phase_transition} follow from \eqref{eq:exp_decay2} through standard arguments
\footnote{When the graph is not periodic, a condition similar to \eqref{eq:exp_decay2} should be shown for all vertices of $\bbG$, not just $0$. 
  It will be apparent from the proof that the values of $c$ and $C$ only depend on the parameter in the bounded angles property and on the distance between the tracks of the grid. It is then straightforward to adapt the proof to graphs with the conditions of \cite{GriMan14}.}.

The strategy used to transfer \eqref{eq:exp_decay2} from the regular square lattice to arbitrary isoradial graphs is similar to that used in the previous section. However, note that the hallmark of the regime $q > 4$ is that boundary conditions influence the model at infinite distance. The arguments in the previous section were based on local modifications of graphs; in the present context, the \rcm in the modified regions is influenced by the structure of the graph outside. This generates additional difficulties that require more careful constructions.  


We start with a technical result that will be useful throughout the proofs.
For $N,M \geq 1$, write $\HR(N;M) = \rect(-N,N;0,M)$ for the half-plane rectangle which is the subgraph of $\bbG$ contained between $t_{-N}$, $t_{N}$, $s_0$ and $s_M$.

\begin{prop} \label{prop:exp_decay_half_full}
  Suppose that there exist constants $C_0, c_0 > 0$ such that for all $N > n$,
  \begin{align} \label{eq:exp_decay3}
    \phi_{\HR(N;N)}^{0} \big[ 0 \leftrightarrow \pd \sq(n) \big] \leq C_0 \exp(-c_0n).
  \end{align}
  Then, there exist constants $C, c > 0$ such that~\eqref{eq:exp_decay2} is satisfied.
  The constants $C, c$ depend only on $C_0, c_0$, on the parameter $\eps$ such that $\bbG \in \calG(\eps)$ 
  and on the size of the fundamental domain of $\bbG$.  
\end{prop}

Observe that~\eqref{eq:exp_decay3} may seem weaker than~\eqref{eq:exp_decay2}. 
Indeed, while $\phi_\bbG^0$ is the limit of $\phi_{\sq(N)}^{0}$ as $N \to \infty$, 
the limit of the measures $\phi_{\HR(N;N)}^{0}$ is what would naturally 
be called the half-plane infinite-volume measure with free boundary conditions. 
Connections departing from $0$ in the latter measure are (potentially) considerably less likely than in $\phi_\bbG^0$ due to their proximity to a boundary with the free boundary conditions. 

\begin{figure}[htb]
  \centering
  \includegraphics[width=0.9\textwidth]{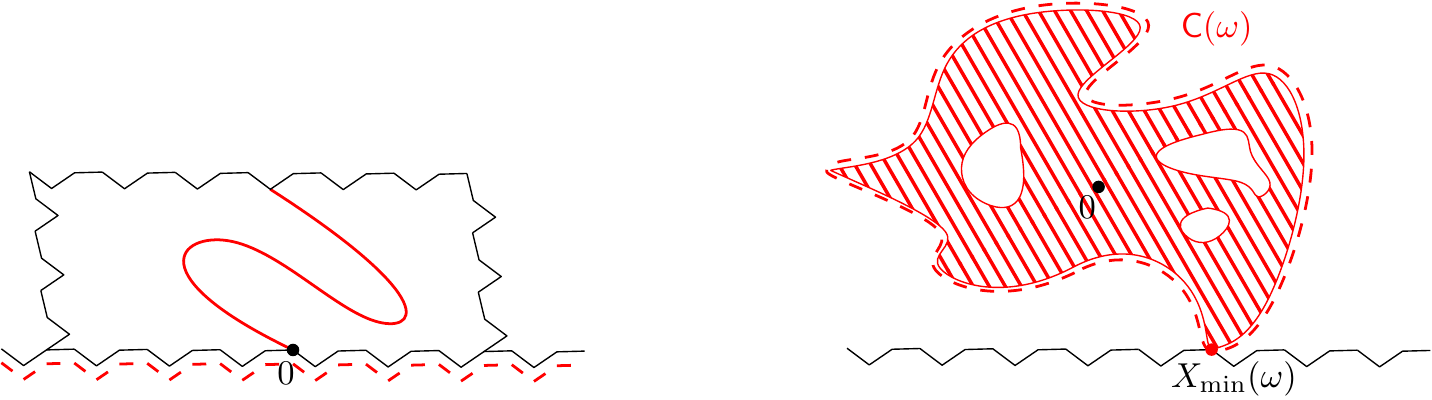}
  \caption{\textbf{Left:} The event of~\eqref{eq:exp_decay3}. 
    \textbf{Right:} If $X_{\min}$ is the lowest point of the cluster of $0$ in $\sq(N)$, 
    then the environment around $X_{\min}$ is less favourable to connections than that of the left image.}
  \label{fig:exp_decay_boundary}
\end{figure}

\begin{proof}
  Fix $N \geq 1$. We will prove \eqref{eq:exp_decay2} for the measure $\phi_{\sq(N)}^0$ instead of $\phi_{\bbG}^0$. 
  It will be apparent from the proof that the constants $c,C$ do not depend on $N$. 
  Thus $N$ may be taken to infinity, and this will provide the desired conclusion.

  For simplicity of notation, let us assume that the grid $(s_n),(t_n)$ of $\bbG$ is such that $\rect(0,1;0,1)$ is a fundamental domain of $\bbG$. 
  Recall that $x_{i,j}$ denotes the vertex of $\bbG$ just to the left of $s_i$ and just below $t_j$. 
  Then, all vertices $x_{i,j}$ are translates of $0$ by vectors that leave $\bbG$ invariant. Write $\|x_{i,j}\| = \max\{|i|,|j|\}$, in accordance with the notation~$\sq(\cdot)$.

  For a random-cluster configuration $\omega$ on $\sq(N)$, let $\sfC(0)$ denote the connected component of the origin.
  Let $X_{\min} = X_{\min}(\omega)$ be a point $x_{i,j}$ of minimal index $j$ such that $\sfC(0)$ intersects $x_{i,j} + \rect(0,1;0,1)$. If several such points exist, choose one according to some rule (e.g., that of minimal $i$). 
  We will estimate the connection probability $\phi_{\sq(N)}^0 \big[ 0 \leftrightarrow \partial \sq(n) \big]$ by studying the possible values of $X_{\min}(\omega)$:
  \begin{align} \label{eq:exp_decay_sum}
    \phi_{\sq(N)}^0 \big[ 0 \leftrightarrow \pd \sq(n) \big]
    = \sum_{\substack{-N \leq i \leq N\\-N \leq j \leq 0}} 
    \phi_{\sq(N)}^0 \big[ 0 \leftrightarrow \pd \sq(n) \text{ and } X_{\min} = x_{i,j} \big].
  \end{align}
  Fix $i,j$ as in the sum and write $\sfC(x_{i,j})$ for the connected component of $x_{i,j}$. 
  By the finite energy property, there exists $\eta$ depending only on the bounded angles property and the size of the fundamental domain of $\bbG$ such that
  \begin{align*}
    \phi_{\sq(N)}^0 \big[ 0 \leftrightarrow \pd \sq(n) \text{ and } X_{\min} = x_{i,j} \big]
    \leq \eta\, \phi_{\sq(N)}^0 \big[ 0 \leftrightarrow x_{i,j} \leftrightarrow \pd \sq(n) \text{ and } X_{\min}= x_{i,j} \big].
  \end{align*}
  Notice that if the event on the right-hand side above occurs, 
  then $x_{i,j}$ is connected to ``distance'' $r := \max\{ \|x_{i,j}\|, \frac{n}{2} \}$; that is $x_{i,j} \leftrightarrow  x_{i,j} + \pd \sq(r)$.
  Moreover, the connected component of $x_{i,j}$ is contained above track $t_j$. 
  By the translation invariance and the comparison between boundary conditions, 
  \begin{align*}
    & \phi_{\sq(N)}^0 \big[ x_{i,j} \leftrightarrow x_{i,j} + \pd \sq(r) \text{ and } \sfC(x_{i,j}) \text{ contained above } t_{j} \big] \\
    & \leq \phi_{\sq(2N)}^0 \big[ 0 \leftrightarrow \pd \sq(r) \text{ and } \sfC(0) \text{ contained above } t_{0} \big]
  \end{align*}
  Let $\Gamma^*$ be the lowest dual left-right crossing of $\sq(2N)$ contained above $t_0$ (actually we allow $\Gamma^*$ to use the faces of $\bbG^\diamond$ below $t_0$ but adjacent to it). If $\sfC(0)$ is contained above $t_{0}$, then $\Gamma^*$ passes under $\sfC(0)$. 
  By conditioning on the values $\gamma^*$ that $\Gamma^*$ may take and using the comparison between boundary conditions we find
  \begin{align*}
    & \phi_{\sq(2N)}^0\big[0 \leftrightarrow \pd \sq(r) \text{ and } \sfC(0) \text{ contained above } t_{0} \big] \\
    & \quad \leq \sum_{\gamma^*}
    \phi_{\sq(2N)}^0 \big[0 \leftrightarrow \pd \sq(r) \text{ and } \sfC(0) \text{ contained above } \gamma^* \,|\, \Gamma^* = \gamma^*\big] \,
    \phi_{\sq(2N)}^0 \big[ \Gamma^* = \gamma^* \big]\\
    & \quad \leq \sum_{\gamma^*} \phi_{\HR(2N;2N)}^0 \big[ 0 \leftrightarrow \pd \sq(r) \big] \,
    \phi_{\sq(2N)}^0 \big[ \Gamma^* = \gamma^* \big]
    \leq C \exp(-cr).
  \end{align*}
  The last inequality is due to~\eqref{eq:exp_decay3}.
  Inserting this into~\eqref{eq:exp_decay_sum} (recall that $r = \max \{ \|x_{i,j}\|, \frac{n}{2} \}$) we find
  \begin{align*}
    \phi_{\sq(N)}^0 \big( 0 \leftrightarrow \partial \sq(n) \big)
    &= \sum_{\substack{-N \leq i \leq N\\-N \leq j \leq 0}} 
    \eta C \exp(- c \max\{ \|x_{i,j}\|; n/2\}) \\
    & \leq \tfrac{n^2}{2} \eta C\exp(- \tfrac{c}2 n) + \sum_{k> n} 2k \eta C \exp(-ck)\\
    & \leq  C'\exp(-c'n),
  \end{align*}
  for some adjusted constants $c',C' > 0$ that do not depend on $n$ or $N$.
  Taking $N \to \infty$, we obtain the desired conclusion.

\end{proof}

The following result will serve as the input to our procedure. 
It concerns only the regularly embedded square lattice and is a consequence of~\cite{DumSidTas13} and~\cite{DumGanHar16}. 
For coherence with the notation above, we consider the square lattice as having edge-length $\sqrt 2$ and rotated by $\frac\pi4$ with respect to its usual embedding. This is such that the diamond graph has vertices $\{ (a,b): a,b  \in \bbZ \}$, with those with $a+b$ even being primal vertices. In a slight abuse of notation, write $\bbZ^2$ for the lattice thus embedded. 

Write $\phi_{\HR(N;N)}^{1/0}$ for the \rcm on the domain $\HR(N; N)$ of $\bbZ^2$ with $\beta = 1$, free boundary conditions on $[-N, N] \times \{ 0 \}$ and wired boundary conditions for the rest of the boundary.
Also define $\bbH = \bbZ \times \bbN$ to be the upper-half plane of $\bbZ^2$.
Write $\phi_{\bbH}^{1/0}$ for the half-plane \rcm which is the weak (decreasing) limit of $\phi_{\HR(N;N)}^{1/0}$ for $N \rightarrow \infty$.

\begin{prop} \label{prop:input_exp}
  For the regular square lattice and $q > 4$, there exist constants $C_0, c_0 >0$ such that, for all $n \geq 1$, 
  \begin{align}
    \phi_{\bbH}^{1/0} \big[ 0 \leftrightarrow \pd \sq(n) \big] \leq C_0 \exp(-c_0 n).
    \label{eq:input_exp}
  \end{align}
\end{prop}

\begin{proof}
  Fix $q > 4$. It is shown in~\cite{DumGanHar16} that the phase transition of the \rcm on $\bbZ^2$ is discontinuous
  and that the critical measure with free boundary conditions exhibits exponential decay. That is, $\phi_{\bbZ^2}^{0}$ satisfies \eqref{eq:exp_decay2} for some constants $C, c > 0$.
  To prove \eqref{eq:input_exp}, it suffices to show that the weak limit $\phi^{1/0}_\bbH$ of the measures $\phi_{\HR(N;N)}^{1/0}$ has no infinite cluster almost surely. Indeed, then $\phi^{1/0}_\bbH$ is stochastically dominated by $\phi^0_{\bbZ^2}$.

  The rest of the proof is dedicated to showing that $\phi^{1/0}_\bbH \big[ 0 \leftrightarrow \infty \big] = 0$, and we do so by contradiction.
  Assume the opposite.
  By ergodicity of $\phi^{1/0}_\bbH$, for any $\varepsilon > 0$, there exists $N > 0$ such that 
  \begin{align} \label{eq:contradiction1}
    \phi^{1/0}_\bbH \big[ \sq(N) \leftrightarrow \infty \big] \geq 1-\eps.
  \end{align}
  Furthermore, $\phi^{1/0}_\bbH$ is also the decreasing limit of the measures $\phi^{1/0}_{\bbS_\ell}$, where $\bbS_\ell = \bbZ \times [0, \ell]$ and $1/0$ refers to the boundary conditions which are wired on the top and free on the bottom of the strip $\bbS_\ell$ 
  (boundary conditions at infinity on the left and right are irrelevant since the strip is essentially one dimensional).
  Therefore, 
  \begin{align}\label{eq:ergo}
    \phi^{1/0}_{\bbS_{4N}} \big[ \sq(N) \leftrightarrow \text{ top of } \bbS_{4N} \big] \geq 1-\eps.
  \end{align}
  In~\cite{DumSidTas13}, Lemma~2~\footnote{Actually a slight adaptation of~\cite[Lem.~2]{DumSidTas13} is necessary to account for the rotation by $\frac{\pi}{4}$ of the lattice.} shows that 
  $$
  \phi^{1/0}_{\bbS_{4N}} \big[ \calC_h^*(-4N, 4N; N, 3N) \big]\geq c_1,
  $$
  for some constant $c_1>0$ not depending on $N$. 
  If $\calC_h^*(-4N, 4N; N, 3N)$ occurs, denote by $\Gamma^*$ the top-most dual crossing in its definition.
  Moreover, let $A$ be the event that $\Gamma^*$ is connected to the line $\bbZ \times \{0\}$ by two dually-open paths 
  contained in $\rect(-4N, -N; 0, 3N)$ and $\rect(N, 4N; 0, 3N)$, respectively (see Figure~\ref{fig:input_exp}).
  Then, using the comparison between boundary conditions and the self-duality of the model, we deduce the existence of $c_2 > 0$ such that 
  \begin{align}
    \phi^{1/0}_{\bbS_{4N}} 
    \big[ A\, 
    \big|\, \calC_h^*(-4N, 4N; N, 3N) \big]
    \geq c_2.
  \end{align}

  \begin{figure}[htb]
    \centering
    \includegraphics[width=0.65\textwidth]{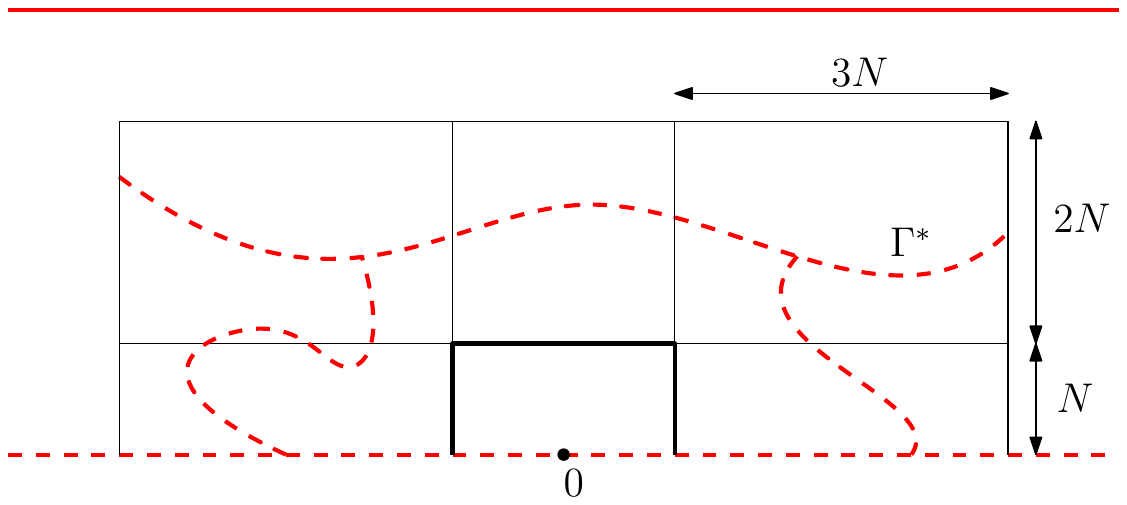}
    \caption{The strip $\bbS_{4N}$ with wired boundary conditions on the top and free on the bottom.
      If $\calC_h^*(-4N, 4N; N, 3N)\cap A$ occurs, then $\sq(N)$ is disconnected from the top of the strip.
      Due to the self-duality, both $\calC_h^*(-4N, 4N; N, 3N)$ 
      and $A$ conditionally on $\calC_h^*(-4N, 4N; N, 3N)$ occur with positive probability.} 
    \label{fig:input_exp}
  \end{figure}

  Notice that, if $\calC_h^*(-4N, 4N; N, 3N)$ and $A$ both occur, then $\sq(N)$ may not be connected to the top of $\bbS_{4N}$ by an open path. Thus
  \begin{align*}
    \phi^{1/0}_{\bbS_{4N}} \big[ \sq(N) \leftrightarrow \text{ top of } \bbS_{4N} \big] \leq 1 - c_1c_2.
  \end{align*}
  This contradicts \eqref{eq:ergo} for $\eps < c_1c_2$, and the proof is complete. 
\end{proof}

The proof of \eqref{eq:exp_decay2} is done in two stages, first it is proved for isoradial square lattices, then for arbitrary doubly-periodic isoradial graphs.  

\subsection{Isoradial square lattices} \label{sec:sq_lat_q>4}

The proof of~\eqref{eq:exp_decay3} for isoradial embeddings of square lattices follows the procedure of Section~\ref{sec:iso_square_RSW}. 
That is, two lattices with same transverse angles for the vertical tracks are glued along a horizontal track. 
Track exchanges are performed, and estimates as those of~\eqref{eq:exp_decay3} are transported from one lattice to the other. 

Transforming the regular lattice $\bbZ^2$ into an arbitrary isoradial one is done in two steps: 
first $\bbZ^2$ is transformed into a lattice with constant transverse angles for vertical tracks;
then the latter (or rather its rotation) is transformed into a general isoradial square lattice. 
For technical reasons, we will perform the two parts separately. 

We should mention that some significant difficulties arise in this step due to the long-range effect of boundary conditions. 
Indeed, recall that in order to perform track exchanges, the graph needs to be convexified. 
This completion affects boundary conditions in an uncontrolled manner, which in this case is crucial. 
Two special arguments are used to circumvent these difficulties; hence the two separate stages in the proof below. 

\medbreak

Recall the notation $\bbG_{\ba, \bb}$ for the isoradial square lattice with transverse angles $\ba = (\alpha_n)_{n \in \bbZ}$ for the vertical train tracks $(s_n)_{n\in \bbZ}$ and $\bb = (\beta_n)_{n \in \bbZ}$ for the horizontal train tracks $(t_n)_{n\in \bbZ}$.
Write $0$ (also written $x_{0,0}$) for the vertex of $\bbG_{\ba, \bb}$ just below track $t_0$ and just to the left of $s_0$. 
We will always assume that $\bbG_{\ba, \bb}$ is indexed such that $0$ is a primal vertex. 

The result of the first part is the following. 

\begin{prop} \label{prop:sq_lat_q>4}
  Let $\bbG_{0, \bb}$ be an isoradial square lattice in $\calG(\eps)$ for some $\eps > 0$, with transverse angles $0$ for all vertical tracks. 
  Then, there exist constants $C, c > 0$ depending on $\eps$ only such that
  \begin{align} \label{eq:exp_decay4}
    \phi_{\HR(N;N)}^{0} \big[ 0 \leftrightarrow \pd \sq(n) \big] \leq C \exp(-c n), \qquad \forall n < N.
  \end{align}
\end{prop}

\begin{proof}
  Fix a lattice $\bbG_{0, \bb}$  as in the statement. 
  For integers $2n < N$, let $G_\mix$ be the mixture of $\bbG_{0, \bb}$ and $\bbZ^2$, as described in Section~\ref{sec:switching_isos}. 
  Notice that here the order of the regular block (that of $\bbZ^2$) and the irregular one (that of $G_{0, \bb}$) is opposite to that in the previous section.

  In this proof, the mixed graph is only constructed above the base level; 
  it has $2N+1$ vertical tracks $(s_{i})_{-N \leq i \leq N}$ of transverse angle $0$ and $2N + 2$ horizontal tracks $(t_{j})_{0 \leq j \leq 2N+1}$, the first $N+1$ having transverse angles $\beta_0, \beta_1, \dots, \beta_N$, respectively 
  and the following $N+1$ having transverse angles $\frac\pi2$. 
  Finally, $G_\mix$ is a convexification of the piece of square lattice described above.
  
  Set $\tilde G_\mix$ to be the result of the inversion of the regular and irregular blocks of $G_\mix$ 
  using the sequence of transformations $\Sigma^{\uparrow}$.	
  Let $\phi_{G_\mix}$ and $\phi_{\tilde G_\mix}$ be the \rcms 
  with the free boundary conditions on $G_\mix$ and $\tilde G_\mix$ respectively. 
  The latter is then the push-forward of the former by the sequence of transformations $\Sigma^{\uparrow}$.
  
  Let $\delta_0 \in (0,1)$ be a constant that will be set below; 
  it will be chosen only depending on $\eps$ and $q$. 
  Write $\pd_L$, $\pd_R$ and $\pd_T$ for the left, right and top boundaries, respectively, of a rectangular domain $\HR(.;.)$. 

  Consider a configuration $\omega$ on $G_\mix$ such that the event $\{ 0 \leftrightarrow \pd \sq(n) \}$ occurs.
  Then, $0$ is connected in $\omega$ to either 
  $\pd_L \HR(n;\delta_0 n)$, $\pd_R \HR(n;\delta_0 n)$ or $\pd_T \HR(n;\delta_0n)$.
  Thus,
  \begin{align}
    \phi_{G_\mix} \big[ 0 \leftrightarrow \pd \sq(n) \big]
    \leq & \, \phi_{G_\mix} \big[ 0 \xleftrightarrow{\HR(n;\delta_0n)} \pd_L \HR(n;\delta_0n) \big] \nonumber \\
    + & \, \phi_{G_\mix} \big[ 0 \xleftrightarrow{\HR(n;\delta_0n)} \pd_R \HR(n;\delta_0n) \big] \nonumber\\
    + & \, \phi_{G_\mix} \big[ 0 \xleftrightarrow{\HR(n;\delta_0n)} \pd_T \HR(n;\delta_0n) \big].
    \label{eq:con1}
  \end{align}
  Moreover, since the graph $G_\mix$ and $G_{0, \bb}$ are identical in the ball of radius $N$ around $0$ for the graph-distance, 
  \begin{align} \label{eq:GG_mix}
    \phi_{\HR(N;N)}^{0} \big[ 0 \leftrightarrow \pd \sq(n) \big]
    \leq \phi_{G_\mix} \big[ 0 \leftrightarrow \pd \sq(n) \big],
  \end{align}
  where in the left-hand side $\HR(N;N)$ denotes the rectangular domain of $G_{0,\bb}$. 
  We used above the comparison between boundary conditions. 
  
  In conclusion, in order to obtain~\eqref{eq:exp_decay4}, it suffices to prove that the three probabilities of the right-hand side of~\eqref{eq:con1} are bounded by an expression of the form $C e^{-cn}$, uniformly in $N$. 
  We concentrate on this from now on. 
  
  \medbreak

  Let us start with the last line of~\eqref{eq:con1}.
  Recall Proposition~\ref{prop:vertical_transport}; a straightforward adaptation reads:
  \medskip
  
  \noindent \textbf{Adaptation of Proposition~\ref{prop:vertical_transport}.}
  {\em 
    There exist $\delta > 0$ and $c_n >0$ satisfying $c_n\to 1$ as $n \to \infty$ such that, for all $n$ and sizes $N \geq 4n$, 
    \begin{align}
      \phi_{\tilde G_\mix} \big[ 0 \xleftrightarrow{\HR(4n; \delta\delta_0 n)} \pd_T \HR(4n; \delta\delta_0 n) \big]
      \geq c_n \phi_{G_\mix} \big[ 0 \xleftrightarrow{\HR(n;\delta_0n)} \pd_T \HR(n;\delta_0n) \big].
      \label{eq:vertical_transport2}
    \end{align}
  }
  
  The proof of the above is identical to that of Proposition~\ref{prop:vertical_transport}.
  The constant $\delta$ and the sequence $(c_n)_n$ only depend on $\eps$ and $q$. 
  
  By the comparison between boundary conditions, 
  \begin{align*}
    \phi_{\tilde G_\mix} \big[ 0 \xleftrightarrow{\HR(4n; \delta\delta_0 n)} \pd_T \HR(4n; \delta\delta_0 n) \big]
    & \leq \phi_{\HR(N;N)}^{1/0} \big[ 0 \xleftrightarrow{} \pd \sq(\delta\delta_0 n) \big] \\
    & \leq C_0 \exp(-c_0 \delta\delta_0 n).
  \end{align*}
  The second inequality is due to Proposition~\ref{prop:input_exp} and to the fact that the rectangle $\HR(N;N)$ of $\tilde G_\mix$
  is fully contained in the regular block.
  Thus, from~\eqref{eq:vertical_transport2} and the above, we obtain,
  \begin{align}
    \phi_{G_\mix} \big[ 0 \xleftrightarrow{\HR(4n; \delta_0 n)} \pd_T \HR(4n; \delta_0 n)  \big] 
    \leq \frac{C_0}{c_n} \exp(-c_0\delta\delta_0 n).
    \label{eq:con2}
  \end{align}
  For $n$ large enough, we have $c_n > 1/2$, and the left-hand side of~\eqref{eq:con2} is smaller than $2C_0\exp(-c_0\delta\delta_0 n)$.
  Since the threshold for $n$ and the constants $c_0, \delta$ and $\delta_0$ only depend on $\eps$ and $q$, 
  the bound is of the required form. 
  \medbreak
  
  We now focus on bounding the probabilities of connection to the left and right boundaries of $\HR(n;\delta_0n)$. 
  Observe that, for a configuration in the event
  $\big\{ 0 \xleftrightarrow{\HR(n;\delta_0n)} \pd_R \HR(n;\delta_0n) \big\}$,
  it suffices to change the state of at most $\delta_0 n$ edges to connect $0$ to the vertex $x_{0,n}$ 
  (we will assume here $n$ to be even, otherwise $x_{0,n}$ should be replaced by $x_{0,n+1}$).
  By the finite-energy property, there exists a constant $\eta = \eta(\eps, q) > 0$ such that 
  \begin{align*}
    \phi_{G_\mix} \big[ 0 \xleftrightarrow{\HR(n;\delta_0n)} \pd_R \HR(n;\delta_0n) \big]
    \leq \exp ( \eta \delta_0 n ) \, \phi_{G_\mix} \big[ 0 \leftrightarrow x_{0, n} \big].
  \end{align*}
  The points $0$ and $x_{0,n}$ are not affected by the transformations in $\Sigma^{\uparrow}$, therefore 
  \begin{align*}
    \phi_{G_\mix} \big[ 0 \leftrightarrow x_{0, n} \big]
    & = \phi_{\tilde G_\mix} \big[ 0 \leftrightarrow x_{0, n} \big] \\
    & \leq \phi_{\tilde G_\mix} \big[ 0 \leftrightarrow \pd \sq(n) \big] \\
    & \leq \phi_{\HR(N;N)}^{1/0} \big[ 0 \leftrightarrow \pd \sq(n) \big]
    \leq C_0 \exp(-c_0 n),
  \end{align*}
  where in the last line, $\HR(N;N)$ is a subgraph of $\tilde G_\mix$, or equivalently of $\bbZ^2$ since these two are identical. 
  The last inequality is given by Proposition~\ref{prop:input_exp}.
  We conclude that, 
  \begin{align} \label{eq:con3}
    \phi_{G_\mix} \big[ 0 \xleftrightarrow{\HR(n;\delta_0n)} \pd_R \HR(n;\delta_0n) \big]
    \leq C_0 \exp \big[ - (c_0 - \delta_0 \eta) n \big].
  \end{align}
  The same procedure also applies to $\{ 0 \xleftrightarrow{\HR(n;\delta_0n)} \pd_L \HR(n;\delta_0n) \}$.
  
  Suppose now that $\delta_0 = \frac{c_0}{c_0\delta + \eta}$ is chosen such that
  $$
  c := c_0 - \delta_0 \eta =c_0\delta\delta_0 > 0.
  $$
  Note that $\delta_0 \in (0, 1)$ since $\eta \geq c_0 $ and that $c$ depends only on $\eps$ and $q$. 
  Then,~\eqref{eq:con1}, \eqref{eq:con2} and~\eqref{eq:con3} 
  imply that for $n$ larger than some threshold depending only on $\eps$, 
  \begin{align*}
    \phi_{G_\mix} \big[ 0 \leftrightarrow \pd \sq(n) \big]
    \leq 4 C_0 \exp(-cn).
  \end{align*}
  Finally, by~\eqref{eq:GG_mix}, we deduce~\eqref{eq:exp_decay4} for all $N \geq 2n$ and $n$ large enough. 
  The condition on $n$ may be removed by adjusting the constant $C$; 
  the bound on $N$ is irrelevant, since the left-hand side of~\eqref{eq:exp_decay4} is increasing in $N$. 
\end{proof}

The same argument may not be applied again to obtain~\eqref{eq:exp_decay4} for general isoradial square lattices 
since it uses the bound~\eqref{eq:input_exp}, which we have not proved for lattices of the form $\bbG_{0, \bb}$.
Indeed,~\eqref{eq:input_exp} is not implied by~\eqref{eq:exp_decay4} when no rotational symmetry is available. 
A different argument is necessary for this step. 

We draw the attention of the reader to the fact that the lattice of Proposition~\ref{prop:sq_lat_q>4} was not assumed to be doubly-periodic, neither will be the following one. 

\begin{prop} \label{prop:sq_lat2_q>4}
  Let $\bbG_{\ba, \bb}$ be an isoradial square lattice in $\calG(\eps)$ for some $\eps> 0$.
  Then, there exist constants $C, c > 0$ depending only on $\eps$, such that 
  \begin{align} \label{eq:exp_decay5}
    \phi_{\HR(N;N)}^{0} \big[ 0 \leftrightarrow \pd \sq(n) \big]
    \leq C \exp(-c n), \qquad \forall n < N.
  \end{align}
\end{prop}

\begin{proof}
  Fix a lattice $\bbG_{\ba, \bb}$  as in the statement. 
  The proof follows the same lines as that of Proposition~\ref{prop:sq_lat_q>4}, with certain small alterations. 
  
  Set $\theta = \frac12 ( \inf \{\beta_n : n \in \bbZ\} + \sup \{\alpha_n : n \in \bbZ\} )$
  and write $\bbG_{\ba, \theta}$ for the lattice with transverse angles $\ba$ for vertical tracks and constant angle $\theta$ for all horizontal tracks. 
  We will refer to this lattice as \emph{regular}. 
  
  For integers $2n \leq N < M$, define $G_\mix$ to be the mixture of $\bbG_{\ba, \bb}$ and $\bbG_{\ba, \theta}$ as in the previous proof, with the exception that, while both blocks have width $2N+1$ and the irregular block (that is that of $\bbG_{\ba, \bb}$) has height $N+1$, the regular block (that of $\bbG_{\ba, \theta}$) has height $M+1$. 
  Precisely, $G_\mix$ is the convexification of the lattice with $2N+1$ vertical tracks $(s_{i})_{-N \leq i \leq N}$ of transverse angles $(\alpha_{-N}, \dots, \alpha_N)$ and $M +N + 2$ horizontal tracks $(t_{j})_{0 \leq j \leq M + N+1}$, the first $N+1$ having transverse angles $\beta_0, \dots, \beta_N$ and the following $M+1$ having transverse angle $\theta$.
  
  Recall that $\tilde G_\mix$, which is the result of the inversion of the regular and irregular blocks of $G_\mix$ by $\Sigma^{\uparrow}$, 
  may be chosen to be an arbitrary convexification of the lattice 
  $G_{(\alpha_{-N}, \dots, \alpha_N), (\theta, \dots, \theta, \beta_0, \dots, \beta_N)}$.
  More precisely, once such a convexification $\tilde G_\mix$ is chosen, 
  a series of \stts $\Sigma^{\uparrow}$ may be exhibited. 
  This fact will be useful later. 
  
  Write as before $\phi_{G_\mix}$ and $\phi_{\tilde G_\mix}$ for the \rcms 
  with free boundary conditions on $G_\mix$ and $\tilde G_\mix$, respectively. 
  Then, by the comparison between boundary conditions, 
  \begin{align*}
    \phi_{\HR(N;N)}^0 \big[ 0 \leftrightarrow \pd \sq(n) \big]
    \leq \phi_{G_\mix} \big[ 0 \leftrightarrow \pd \sq(n) \big],
  \end{align*}
  where $\HR(N;N)$ refers to the domain in $\bbG_{\ba, \bb}$, or equivalently in $G_\mix$ since the two are equal. 
  Notice that the above inequality is valid for all $M$. 
  
  Let $\delta_0 \in (0,1)$ be a constant that will be set below.
  Using the same notation and reasoning as in the previous proof, we find
  \begin{align}
    \phi_{G_\mix} \big[ 0 \leftrightarrow \pd \sq(n) \big]
    & \leq 
    \tfrac{1}{c_n} \phi_{\tilde G_\mix} \big[ 0 \xleftrightarrow{\HR(2n; \delta \delta_0 n)} \pd_T \rect(2n; \delta \delta_0 n) \big]
    + 2 \exp (\eta \delta_0 n) \, \phi_{\tilde G_\mix} \big[ 0 \leftrightarrow x_{0,n} \big] \nonumber \\
    & \leq 2\phi_{\tilde G_\mix} \big[ 0 \leftrightarrow \pd \sq(\delta \delta_0 n) \big]
    + 2 \exp (\eta \delta_0 n) \, \phi_{\tilde G_\mix} \big[ 0 \leftrightarrow \pd \sq(n) \big], 
    \label{eq:exp_decay6}
  \end{align}
  where $\delta > 0$ and $\eta > 0$ are constants depending only on $\eps$ and $q$. 
  The latter inequality is only valid for $n$ above a threshold also only depending on $\eps$ and $q$. 
  
  At this point, the previous proof used~\eqref{eq:input_exp} to bound the right-hand side. 
  Since this is no longer available, we will proceed differently.

  As previously stated, we may choose the convexification for $\tilde G_\mix$. 
  Let it be such that the tracks with transverse angle $\theta$ are as low as possible.
  That is, $\tilde G_\mix$ is such that, 
  for any track $t$ with transverse angle $\theta$, any intersection below $t$ involves one track with transverse angle $\theta$. 
  The existence of such a convexification is easily proved; rather than writing a formal proof, 
  we prefer to direct the reader to the example of Figure~\ref{fig:sq_lat2_q>4}.
  \begin{figure}[htb]
    \centering
    \includegraphics[width=0.9\textwidth]{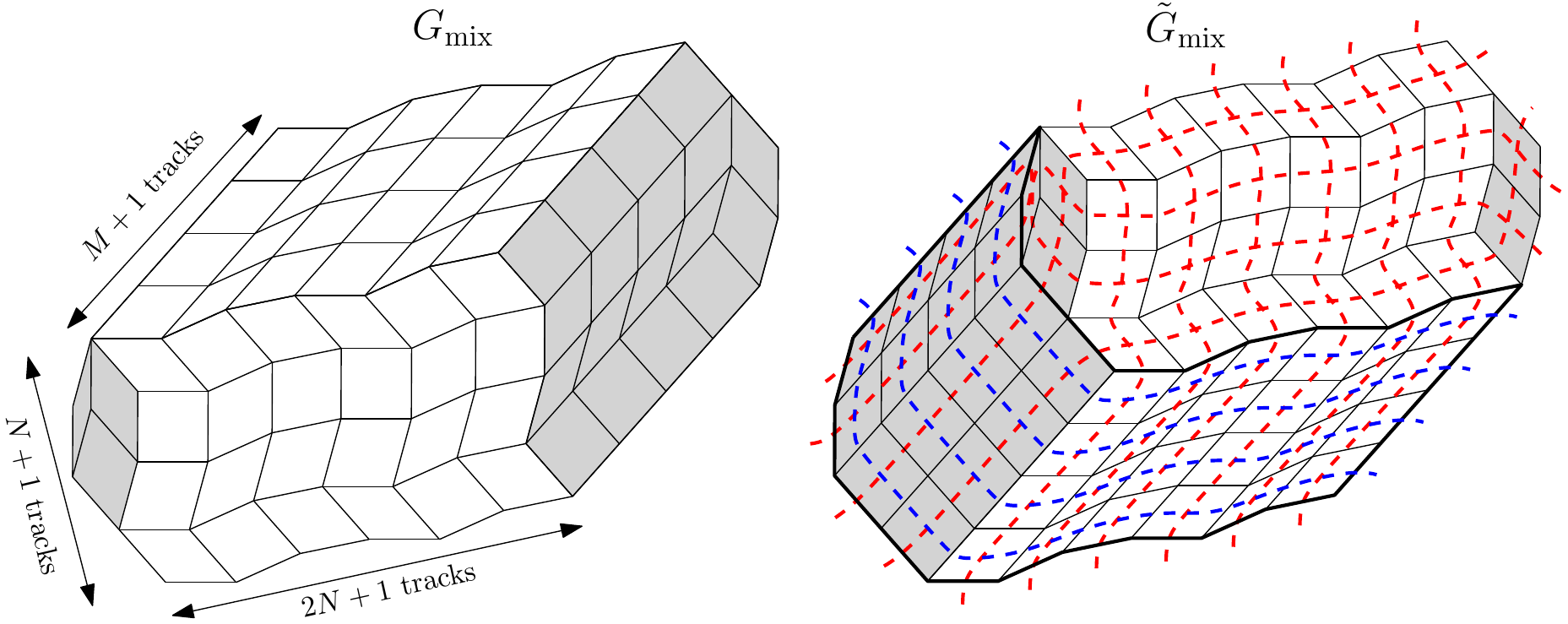}
    \caption{\textbf{Left:} The diamond graph $G_{\mix}^\diamond$ obtained by superposing a block of $\bbG_{\ba, \theta}^\diamond$ to one of $\bbG_{\ba, \bb}^\diamond$. The convexification is drawn in gray. 
      \textbf{Right:}
      The diamond graph $\tilde G_\mix^\diamond$ with convexification (gray) chosen such that the tracks $\tilde t_0, \dots, \tilde t_M$ (blue) are as low as possible. This ensures that the region below $\tilde t_{M}$ (delimited in bold) has a square lattice structure. }
    \label{fig:sq_lat2_q>4}
  \end{figure}

  Write $\tilde t_0, \dots, \tilde t_M$ for the tracks of transverse angle $\theta$ of $\tilde G_\mix$, indexed in increasing order. 
  Call $\tilde s_{-2N-1}, \dots, \tilde s_{N}$ the tracks intersecting $\tilde t_0$, ordered by their intersection points from left to right. 
  Denote by $\tilde \alpha_{-2N-1}, \dots, \tilde \alpha_{N}$ their transverse angles.
  
  The family $\tilde s_{-2N-1}, \dots, \tilde s_{N}$ contains all vertical tracks of the original graph $\bbG_{\ba,\bb}$ 
  (that is those denoted by $s_{-N}, \dots, s_N$) but also the horizontal tracks of $\bbG_{\ba,\bb}$ with transverse angles different from $\theta$.
  Since $\theta < \beta_j$ for all $0 \leq j \leq N$, the latter intersect $t_0$ left of the former. 
  Thus,  $\tilde s_i = s_i$ for $-N \leq i \leq N$, hence the indexing. 
  
  The region of $\tilde G_\mix$ contained below $\tilde t_{M}$ is a (finite part of a) square lattice. 
  Precisely, it is the square lattice $G_{(\tilde \alpha_i)_{-2N-1 \leq i \leq N}, (\theta)_{0 \leq j \leq M}}$.
  Complete the sequence $(\tilde \alpha_i)_{-2N-1\leq i \leq N}$ into a bi-infinite sequence $\tilde \ba = (\tilde \alpha_i)_{i \in \bbZ}$ 
  by declaring all additional terms equal to $\tilde \alpha_0$.
  Write $\tilde \rect(.,.;.,.)$ for the domains of $\bbG_{\tilde \ba,\theta}$ 
  defined in terms of the tracks $(\tilde s_i)_{i \in \bbZ}$ and $(\tilde t_j)_{j \in \bbZ}$.
  Also write $\rect(.,.;.,.)$ for the domains of the original lattice $\bbG_{\ba,\bb}$.
  
  By the comparison between boundary conditions, for any increasing event $A$ depending only on the region of $\tilde G_\mix$ below $\tilde t_{M}$,
  \begin{align*}
    \phi_{\tilde G_\mix} [A] \leq \phi_{\tilde \rect(-2N-1,N;0,M)}^\xi [A],
  \end{align*}
  where $\xi$ are the boundary conditions which are wired on the top of 
  $\tilde \rect(-2N-1,N;0,M)$ and free on the rest of the boundary. 
  Thus,~\eqref{eq:exp_decay6} implies that
  \begin{align*}
    & \phi_{\HR(N;N)}^0 \big[ 0 \leftrightarrow \pd \sq(n) \big] \\
    & \qquad \leq  2\phi_{\tilde \rect(-2N-1, N; 0, M)}^\xi \big[ 0 \leftrightarrow \pd \sq(\delta \delta_0 n) \big]
    + 2 \exp ( \eta \delta_0 n) \, \phi_{\tilde \rect(-2N-1, N; 0, M)}^\xi \big[ 0 \leftrightarrow \pd \sq(n) \big].
  \end{align*}
  Since the above is true for all $M$, we may take $M$ to infinity. 
  Then, the measures $\phi_{\tilde \rect(-2N-1, N; 0, M)}^\xi$ tend decreasingly to the measure 
  $\phi_{\tilde \rect(-2N-1, N; 0, \infty)}^0$ with free boundary conditions in the half-infinite strip.
  \footnote{This step is standard.
    Let $A$ be an increasing event depending only on the state of edges in $\tilde \rect(-2N-1, N; 0, M_0)$ for some $M_0$.
    Then, for any $M  > M_0$, denote by $\Gamma^*$ the highest dually-open horizontal crossing of $\tilde \rect(-2N-1, N; 0, M)$ 
    and set $H$ to be the event that $\Gamma^*$ does not intersect $\tilde \rect(-2N-1, N; 0, M_0)$.
    Then, $\Gamma^*$ may be explored from above, and standard arguments of comparison between boundary conditions imply that 
    $$
    \phi_{\tilde \rect(-2N-1, N; 0, M)}^\xi [A]
    \leq \phi_{\tilde \rect(-2N-1, N; 0, \infty)}^0 [A]
    \, \phi_{\tilde \rect(-2N-1, N; 0, M)}^\xi[H] + \phi_{\tilde \rect(-2N-1, N; 0, M)}^\xi [H^c].
    $$
    By the finite energy property and the fact that $\tilde \rect(-2N-1, N; 0, \infty)$ has constant width, 
    $\phi_{\tilde \rect(-2N-1, N; 0, M)}^\xi [H] \to 1$ as $M \to \infty$. 
    This suffices to conclude.}.
  We conclude that 
  \begin{align}
    & \phi_{\HR(N;N)}^0 \big[ 0 \leftrightarrow \pd \sq(n) \big] \nonumber \\
    & \qquad \leq 
    2 \phi_{\tilde \rect(-2N-1, N; 0, \infty)}^0 \big[ 0 \leftrightarrow \pd \sq(\delta \delta_0 n) \big]
    + 2 \exp ( \eta \delta_0 n) \, \phi_{\tilde \rect(-2N-1, N; 0, \infty)}^0 \big[ 0 \leftrightarrow \pd \sq(n) \big]	\nonumber\\
    & \qquad\leq 
    2 \phi_{\bbG_{\tilde \ba, \theta}}^0 \big[ 0 \leftrightarrow \pd \sq(\delta \delta_0 n) \big]
    + 2 \exp ( \eta \delta_0 n) \, \phi_{\bbG_{\tilde \ba, \theta}}^0 \big[ 0 \leftrightarrow \pd \sq(n) \big].
    \label{eq:exp_decay7}
  \end{align}
  Finally, the square lattice $\bbG_{\tilde \ba,\theta}$ has constant transverse angle $\theta$ for all its horizontal tracks 
  and, by choice of $\theta$, is in $\calG(\eps/2)$.
  Thus, Proposition~\ref{prop:sq_lat_q>4} applies to it (or rather to its rotation $\bbG_{0, (\tilde \alpha_{-j} - \theta + \pi)_j}$). 
  We conclude the existence of constants $C_0, c_0 > 0$ depending on $\eps$ only, such that 
  $$
  \phi_{\bbG_{\tilde \ba, \theta}}^0 \big[ 0 \leftrightarrow \pd \sq(k) \big]
  \leq C_0 \exp(-c_0k), \qquad \forall k \geq 1.
  $$
  Set $\delta_0 = \frac{c_0}{c_0\delta + \eta}$ and
  \begin{align*}
    c = c_0 - \delta_0 \eta =c_0\delta\delta_0 > 0. 
  \end{align*}
  Then $c$ only depends on $\eps$ and $q$
  and the right hand side of~\eqref{eq:exp_decay7} is bounded by $4C \exp(-cn)$, which provides the desired conclusion.
\end{proof}

The second proposition (Proposition~\ref{prop:sq_lat2_q>4}) appears to use a weaker input than the first. 
One may therefore attempt to use the same argument for Proposition~\ref{prop:sq_lat_q>4}, 
so as to avoid using the more involved bound~\eqref{eq:input_exp}. 
Unfortunately, this is not possible, as the sequence of angles $\tilde \ba$ in the second proof may never be rendered constant, since it contains all the horizontal and vertical tracks of $\bbG_{\ba, \bb}$. 

\subsection{Doubly-periodic isoradial graphs}

Let $\bbG$ be an arbitrary doubly-periodic isoradial graph in some $\calG(\eps)$, with grid $(s_n)_{n \in \bbZ}$, $(t_n)_{n \in \bbZ}$.
Denote by $0$ the vertex just below and to the left of the intersection of $t_0$ and $s_0$. We will assume that it is a primal vertex. 
The goal of this section is the following. 

\begin{prop} \label{prop:db_per_q>4}
  There exist constants $c,C > 0$ depending only on $\eps$ and on the size of the fundamental domain of $\bbG$, such that
  \begin{align} \label{eq:exp_decay8}
    \phi_{\HR(N;N)}^{0} \big[ 0 \leftrightarrow \pd \sq(n) \big] \leq C \exp(-c n), \qquad \forall n < N.
  \end{align}
\end{prop}

Again, some care is needed when handling boundary conditions. 
Rather than working with $\bbG$ and modifications of it, we will construct a graph that locally resembles $\bbG$, 
but that allows us to control boundary conditions. 

\begin{proof}
  For $\rho \in [0,\pi)$, write $\calT_\bbG^{(\rho)}$ for the set of tracks of $\bbG$ with asymptotic direction $\rho$ (recall the existence of an asymptotic direction from the proof of Lemma~\ref{lem:square_grid}).
  By periodicity, there exists a finite family $0 \leq \rho_0 < \dots < \rho_T < \pi$ such that 
  \begin{align*}
    \calT_\bbG = \bigsqcup_{\ell = 0}^T \calT_\bbG^{(\rho_\ell)}.
  \end{align*}
  Assume that the lattice is rotated such that the horizontal tracks $(t_n)_{n \in \bbZ}$ have asymptotic direction $\rho_0 = 0$. 
  Fix constants $n < N$.  

  Let $\tau_L$ be the right-most track in $\calT_{\bbG}^{(\rho_T)}$ that intersects $t_0$ left of $0$ and does not intersect $\HR(N;N)$.
  Similarly, define $\tau_R$ as the left-most track in $\calT_{\bbG}^{(\rho_1)}$ that intersects $t_0$ right of $0$ and does not intersect $\HR(N;N)$.  
  Denote by $\calD_0$ the domain of $\bbG$ bounded by $t_0$ below, above by $t_N$, to the left by $\tau_L$ and to the right by $\tau_R$. 
  One may imagine $\calD_0$ as a trapezoid with base $t_0$ and top $t_N$. 
  By definition of $\tau_L$ and $\tau_R$, 
  \begin{align} \label{eq:RcalD}
    \HR(N;N) \subset \calD_0.
  \end{align}

  Complete $\calD_0$ to form a bigger, finite graph $\calD$ as follows. 
  Let $\tilde s_{K_-},\dots, \tilde s_{K_+}$ be the tracks of $\calD_0$ that intersect $t_0$, ordered from left to right, with $\tilde s_0 = s_0$. 
  Orient these tracks upwards. Orient the remaining tracks $t_0, \dots, t_N$ from left to right. 

  In $\calD$, we will make sure that  $t_1, \dots, t_N$ intersect all tracks $(\tilde s_i)_{K_- \leq i \leq K_+}$, 
  but that no additional intersections between tracks $(\tilde s_i)_{K_- \leq i \leq K_+}$ are introduced. 
  One should imagine that the track $t_1$, after exiting $\calD_0$, ``slides down'' on the side of $\calD_0$; 
  $t_2$ does the same: it slides down along the side of $\calD_0$ until reaching $t_1$, then continues parallel to $t_1$, etc. 
  The same happens on the left side. 
  Finally, on top of the graph obtained above add a number of parallel tracks $t_{N+1},\dots, t_{M}$ adjacent to each other, with constant transverse angle, for instance, that of $t_N$, for some $M > N$. 
  Call the resulting graph $\calD$.
  In $\calD$, each track $t_j$ with $0\leq j \leq M$ intersects all tracks $(\tilde s_i)_{K_- \leq i \leq K_+}$. 
  We do not give a more formal description of the construction of $\calD$; we rather direct the readers attention to Figure~\ref{fig:calD} for an illustration.

  \begin{figure}[htb]
    \centering
    \includegraphics[width=0.8\textwidth]{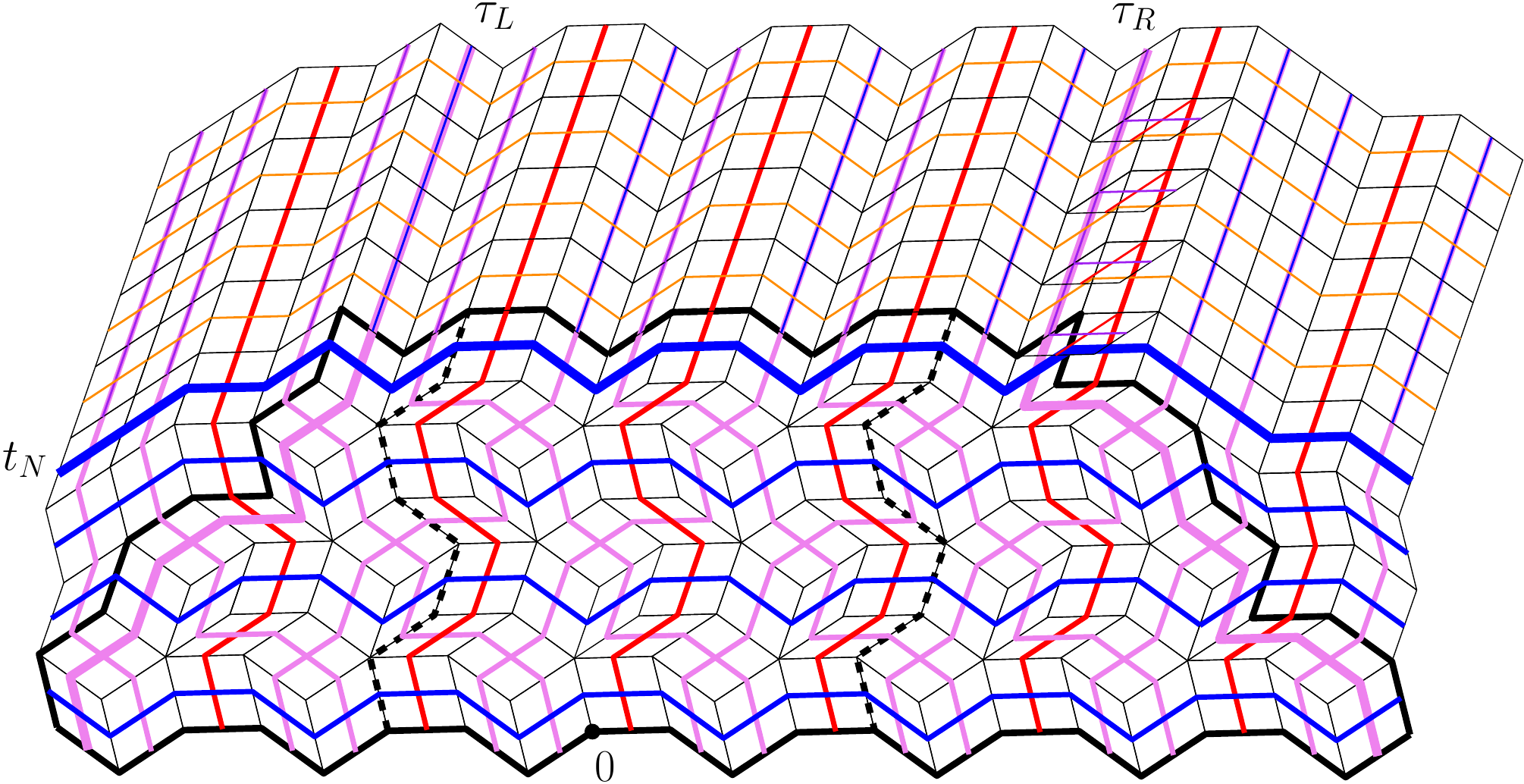}
    \caption{A graph $\calD_0$ (the delimited region) and the completion $\calD$ -- only the diamond graph is depicted. 
      The vertical tracks $(s_i)$ are red, horizontal tracks $(t_j)$ blue, and the others purple.
      The rectangle $\HR(N;N)$ (in reality it should be wider) is delimited by dotted lines. 
      The tracks delimiting $\calD_0$ are $\tau_L$, $\tau_R$ and $t_N$; they are marked in bold.} 
    \label{fig:calD}
  \end{figure}

  \medbreak

  In light of~\eqref{eq:RcalD}, 
  \begin{align} \label{eq:RcalD1}
    \phi_{\HR(N;N)}^{0} \big[ 0 \leftrightarrow \pd \sq(n) \big]
    \leq \phi_{\calD}^{0} \big[ 0 \leftrightarrow \pd \sq(n) \big],
  \end{align}
  where $\HR(N;N)$ refers to the region of the graph $\bbG$. 

  Next, we transform $\calD$ to create a square lattice. 
  Call a \emph{black point} of $\calD$ any intersection of two tracks $\tilde s_i$ and $\tilde s_j$. 
  Then, by a straightforward modification of Proposition~\ref{prop:bi_to_square}, there exist \stts $\sigma_1,\dots, \sigma_K$ applied to $\calD$ such that, in $(\sigma_K \circ \dots \circ \sigma_1)(\calD)$, there is no black point between $t_0$ and $t_1$. 
  Moreover, all transformations $\sigma_1, \dots, \sigma_K$ act between $t_0$ and $t_1$. 

  The existence of $\sigma_1, \dots, \sigma_K$ is proved by eliminating one by one the back points of $\calD$ between $t_0$ and $t_1$, starting with the top most. The main thing to observe is that, by the construction of $\calD$, any black point between $t_0$ and $t_1$ is the intersection of two tracks $\tilde s_i$ and $\tilde s_j$, both of which intersect $t_1$ above. 

  Set $\Sigma_1 = \sigma_K \circ \dots \circ \sigma_1$. 
  Then, one may define recurrently sequences of transformations $(\Sigma_j)_{1 \leq j \leq M}$ such that
  \begin{itemize}
  \item each $\Sigma_j$ acts on $(\Sigma_{j-1} \circ \dots \circ \Sigma_{1}) (\bbG)$ between $t_{j-1}$ and $t_j$;
  \item in $(\Sigma_{j} \circ \dots \circ \Sigma_{1}) (\bbG)$, there are no black points below $t_j$. 
  \end{itemize}

  Let $\overline \calD =  (\Sigma_M \circ \dots \circ \Sigma_{1}) (\bbG)$. 
  Below $t_M$, $\overline \calD$ is a rectangular part of a square lattice with width $K_{+} -K_{-}+1$ and height $M$. 

  Let $\tilde \ba = (\tilde \alpha_i)_{K_- \leq i \leq K_+}$ be the transverse angles of the tracks $(\tilde s_i)_{K_- \leq i \leq K_+}$.
  Also, denote by $\tilde \bb = (\tilde \beta_j)_{j \geq 0}$ the sequence of angles constructed as follows:
  for $0 \leq j \leq N$, $\tilde \beta_j = \beta_j$ which is the transverse angles of $t_j$; for $j > N$, set $\tilde \beta_j = \beta_N$.
  Then, the part of $\overline \calD$ below $t_M$ is a rectangular domain of the vertical strip of square lattice $G_{\tilde \ba, \tilde \bb}$.
  Moreover, also let us denote by $\bbG_{\tilde \ba, \tilde \bb}$ any completion of $G_{\tilde \ba, \tilde \bb}$ into a full plane square lattice.
  We note that the angles $\tilde \ba$ and $\tilde \bb$ also are transverse angles of tracks in $\bbG$, therefore $\bbG_{\tilde \ba, \tilde \bb}\in \calG(\eps)$ as does $\bbG$. 

  By the same argument as for~\eqref{eq:exp_decay7}, for any fixed $j \leq N$, 
  \begin{align} \label{eq:calD1}
    \limsup_{M \to \infty}\phi_{\overline \calD}^0 \big[ 0 \leftrightarrow \pd \sq(j) \big]
    \leq \phi_{\bbG_{\tilde \ba, \tilde \bb}}^0 \big[ 0 \leftrightarrow \pd \tilde \sq(j) \big]. 
  \end{align}
  In the above inequality, $\sq$ denotes a square domain defined in terms of tracks $(s_i)$ and $(t_j)$ whereas $\tilde \sq$ is defined in terms of tracks $(\tilde s_i)$ and $(t_j)$.
  We also notice that the box defined in terms of $(s_i)$ is larger than that of $(\tilde s_i)$.
  Proposition~\ref{prop:sq_lat2_q>4} applies to $\bbG_{\tilde \ba, \tilde \bb}$ and we deduce the existence of constants $c, C > 0$, depending only on $\eps$, such that, for all $j \leq N$, 
  \begin{align} \label{eq:calD2}
    \limsup_{M \to \infty} \phi_{\overline \calD}^0 \big[ 0 \leftrightarrow \pd \sq(j) \big]
    \leq \phi_{\bbG_{\tilde \ba, \tilde \bb}}^0 \big[ 0 \leftrightarrow \pd \sq(j) \big]
    \leq C \exp(-cj).
  \end{align}

  Let us now come back to connections in $\calD$. These can be transformed into connections in $\overline \calD$ via the sequence of \stts $(\Sigma_j)_{1 \leq j \leq M}$.
  We use the same decomposition as in the proofs of Section~\ref{sec:sq_lat_q>4} to transfer the exponential decay in $\overline \calD$ to that in $\calD$. 

  Let $\omega$ be a configuration on $\calD$ such that $0 \leftrightarrow \pd\sq(n)$. 
  Then, either $0$ is connected to the left or right sides of $\HR(n;\delta_0n)$ or it is connected to the top of $\HR(n;\delta_0n)$ inside $\HR(n;\delta_0n)$. The constant $\delta_0$ used in this decomposition will be chosen below and will only depend on $\eps$ and $q$. 
  Thus, we find, 
  \begin{align}
    \phi_{\calD}^0 \big[ 0 \leftrightarrow \pd \sq(n) \big]
    \leq & \, \phi_{\calD}^0 \big[0 \xleftrightarrow{\HR(n;\delta_0n)} \pd_L \HR(n;\delta_0n)\big] \nonumber \\
    + & \, \phi_{\calD}^0 \big[ 0 \xleftrightarrow{\HR(n;\delta_0n)} \pd_R \HR(n;\delta_0n) \big] \nonumber \\
    + & \, \phi_{\calD}^0 \big[ 0 \xleftrightarrow{\HR(n;\delta_0n)} \pd_T \HR(n;\delta_0n) \big].
    \label{eq:con4}
  \end{align}

  Let us now bound the three terms above separately. We start with the first two. 
  By the finite-energy property, there exist
  a constant $\eta > 0$ depending only on $\eps$ and the fundamental domain of $\bbG$
  and  primal vertices $x_-$ and $x_+$ just below $t_0$ (thus on the boundary of $\calD$)
  left of $s_{-n}$ and right of $s_n$, respectively,  such that 
  \begin{align}
    & \phi_{\calD}^0 \big[ 0 \xleftrightarrow{\HR(n;\delta_0n)} \pd_L \HR(n;\delta_0n) \big] 
    \leq \exp (\eta \delta_0 n) \, \phi_{\calD}^0 (0 \xleftrightarrow{} x_- ) \qquad \text{and} \nonumber \\
    & \phi_{\calD}^0 \big[ 0 \xleftrightarrow{\HR(n;\delta_0n)} \pd_R \HR(n;\delta_0n) \big] 
    \leq \exp (\eta \delta_0 n) \, \phi_{\calD}^0 (0 \xleftrightarrow{} x_+ ).
    \label{eq:con5}
  \end{align}
  Since the transformations $\Sigma_1, \dots, \Sigma_M$ preserve connections between points on the boundary of $\calD$, 
  \begin{align}
    \phi_{\calD}^0 \big[ 0 \xleftrightarrow{} x_- \big]
    = \phi_{\overline \calD}^0 \big[ 0 \xleftrightarrow{} x_- \big]
    \leq \phi_{\overline \calD}^0 \big[ 0 \leftrightarrow \pd \sq(n) \big].
    \label{eq:con6}
  \end{align}
  The same holds for  $\phi_{\calD}^0 \big[ 0 \xleftrightarrow{} x_+ \big]$. 

  Since each sequence of transformations $\Sigma_k$ only acts below $t_k$, an open path connecting $0$ to $\pd_T \HR(n;\delta_0n)$ in $\calD$ is transformed into an open path connecting $0$ to $t_{\delta_0 n -1}$ (Figure~\ref{fig:path_transformations}).
  \begin{align}
    \phi_{\calD}^0 \big[ 0 \xleftrightarrow{\HR(n;\delta_0n)} \pd_T \HR(n;\delta_0n) \big] 
    &\leq \phi_{\overline \calD}^0 \big[ 0 \xleftrightarrow{} t_{\delta_0 n -1} \big] \nonumber \\
    &\leq \phi_{\overline \calD}^0 \big[ 0 \xleftrightarrow{} \pd \sq(\delta_0 n -1) \big] \nonumber \\
    &\leq C \exp(-c(\delta_0 n -1)). 
    \label{eq:con7}
  \end{align}
  By injecting~\eqref{eq:con5},~\eqref{eq:con6} and \eqref{eq:con7} into~\eqref{eq:con4}, then further into~\eqref{eq:RcalD1}, we find
  \begin{align*}
    \phi_{\HR(N;N)}^{0} \big[ 0 \leftrightarrow \pd \sq(n) \big] \leq 
    2 \exp (\eta \delta_0 n) \, \phi_{\overline \calD}^0 \big[ 0 \leftrightarrow \pd \sq(n) \big]
    + \phi_{\overline \calD}^0 \big[ 0 \xleftrightarrow{} \pd \sq(\delta_0 n -1) \big]. 
  \end{align*}
  The above is true for all $M$, and we may take $M \to \infty$. Using~\eqref{eq:calD2}, we find
  \begin{align*}
    \phi_{\HR(N;N)}^{0} \big[ 0 \leftrightarrow \pd \sq(n) \big] \leq 
    2C \exp \big[ - (c - \eta \delta_0) n \big] + C \exp \big[ -c (\delta_0 n-1) \big].
  \end{align*}
  Set $\delta_0 = \frac{c}{c+\eta}$ so that $c' := c - \delta_0 \eta =c\delta_0> 0$.
  Then, we deduce that
  \begin{align*}
    \phi_{\HR(N;N)}^{0} \big[ 0 \leftrightarrow \pd \sq(n) \big] \leq 
    3C\, e^c\, \exp ( -c' n).
  \end{align*}
  Since $c$ and $\eta$ only depend on $\eps$, $q$ and the size of the fundamental domain of $\bbG$, we obtain the desired result. 
\end{proof}

\subsection{Conclusion} \label{sec:conclusions_q>4}

Below, we show how the previous two parts imply Theorem~\ref{thm:main2} and Corollary~\ref{cor:phase_transition} for $q > 4$. 
Fix a doubly-periodic isoradial graph $\bbG$ and one of its grids. 

\begin{proof}[Proof of Theorem~\ref{thm:main2}]
  Due to~\eqref{eq:exp_decay8} and Proposition~\ref{prop:exp_decay_half_full}, $\bbG$ satisfies~\eqref{eq:exp_decay2}. 
  Since $\bbG$ satisfies the bounded angles property for some $\eps > 0$ and due to its periodicity, 
  there exists a constant $\alpha > 0$ such that $\sq(n) \subseteq B_{\alpha n}$ for all $n \geq 1$.
  Then,~\eqref{eq:exp_decay2} implies 
  \begin{align}\label{eq:exp_decay}
    \phi_{\bbG}^0 [ 0 \leftrightarrow \partial B_n ]
    \leq \phi_{\bbG}^0 [ 0 \leftrightarrow \partial \sq( \tfrac{n}{\alpha} ) ]
    \leq C\exp \big(- \tfrac{c}{\alpha} n \big). 
  \end{align}
  This implies the second point of Theorem~\ref{thm:main2} with an adjusted value for $c$.


  Let us now consider the model with wired boundary conditions. 
  Recall that if $\omega$ is sampled according to $\phi_\bbG^1$, then its dual configuration follows $\phi_{\bbG^*}^0$.
  Since $\bbG^*$ is also a doubly-periodic isoradial graph, \eqref{eq:exp_decay} applies to it. 

  For a dual vertex $y \in \bbG^*$, let $C(y)$ be the event that there exists a dually-open circuit going through $y$ and surrounding the origin.
  The existence of such a circuit implies that the dual-cluster of $y$ has radius at least $|y|$. 
  Thus,
  \begin{align*}
    \phi_{\bbG}^1 \big[ C(y) \big]
    \leq \phi_{\bbG^*}^0 \big[ y \leftrightarrow y + \pd B_{|y|} \big]
    \leq C \exp (-c |y|),
  \end{align*}
  for some $c,C > 0$ not depending on $y$ \footnote{We implicitly used here that \eqref{eq:exp_decay} applies to $\phi_{\bbG^*}^0$ and to any translate of it. This is true due to the periodicity of $\bbG^*$. A multiplicative constant depending on the size of the fundamental domain is incorporated in $C$.}.
  Since the number of vertices in $\bbG^* \cap B_n$ is bounded by a constant times $n^2$, the Borel-Cantelli lemma applies and we obtain
  \begin{align*}
    \phi_{\bbG}^1 \big[ C(y) \text{ for infinitely many } y \in \bbG^* \big] = 0.
  \end{align*}
  The finite-energy property of $\phi_\bbG^1$ then implies $\phi_{\bbG}^1 \big[ 0 \leftrightarrow \infty \big] > 0$.
\end{proof}

\begin{proof}[Proof of Corollary~\ref{cor:phase_transition} for $q >4$]
  It is a well known fact that $\phi^0_{\bbG, \beta,q} = \phi^1_{\bbG, \beta,q}$ for all but countably many values of $\beta$ 
  (see for instance \cite[Thm 1.12]{HDC_pims} for a recent short proof that can be adapted readily to isoradial graphs).
  Thus, by the monotonicity of the measures $\phi^0_{\bbG,\beta,q}$, for any $\beta <1$, $\phi^0_{\bbG, 1,q}$ dominates $\phi^1_{\bbG, \beta,q}$. Theorem~\ref{thm:main2} then implies $\phi^1_{\bbG, \beta,q}(0\leftrightarrow \infty) = 0$ and \cite[Thm. 5.33]{Gri06} yields $\phi^1_{\bbG, \beta,q} = \phi^0_{\bbG, \beta,q}$. This proves the uniqueness of the infinite volume measure for all $\beta < 1$. The first point of the corollary follows directly from Theorem~\ref{thm:main2} by the monotonicity mentioned above. 
  
  Since measures with $\beta > 1$ are dual to those with $\beta<1$, the uniqueness of the infinite volume measure also applies when $\beta > 1$. 
  The second point of the corollary follows from Theorem~\ref{thm:main2} by monotonicity. 
  %
  %
  %
\end{proof}

\section{Proofs for the quantum random-cluster model} \label{sec:quantum}

\subsection{Discretisation}

Fix $\eps > 0$ and consider the isoradial square lattice $\bbG^\eps := \bbG_{\ba, \bb}$ 
where $\alpha_n = 0$ for all $n$ and $\beta_n = \eps$ if $n$ is even and $\beta_n = \pi - \eps$ if $n$ is odd.
See Figure~\ref{fig:flattened_lattice} for an illustration.

Recall the notation $x_{i,j}$ with $i+j$ even for the primal vertices of $\bbG^\eps$ (while $x_{i,j}$ with $i+j$ odd are the dual vertices). 
Also, recall the notation $\rect(i, j; k, \ell)$ for the domains of $\bbG^\eps$ contained between the vertical tracks $s_i$ and $s_j$ and the horizontal tracks $t_k$ and $t_\ell$. 

\begin{figure}[htb]
  \centering
  \includegraphics[scale=0.8, page=1]{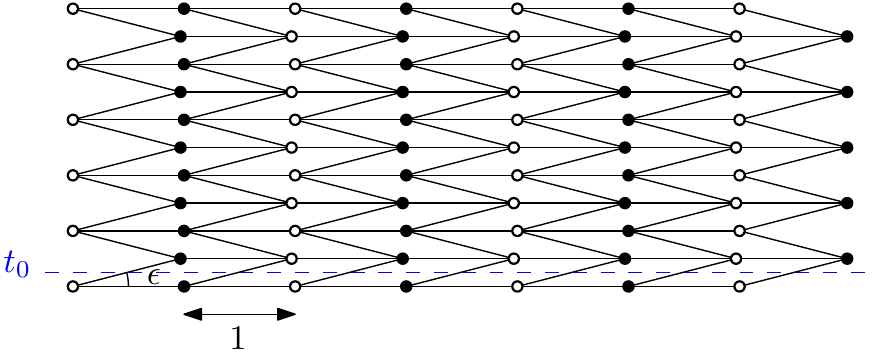}
  \includegraphics[scale=0.8, page=2]{images/flattened_lattice.pdf}
  \caption{A piece of the deformed square lattice $\bbG^\eps = \bbG_{\ba, \bb}$. \textbf{Left:} The diamond graph. \textbf{Right:} The primal lattice with dual vertices.}
  \label{fig:flattened_lattice}
\end{figure}

Observe that $\bbG^\eps$ contains two types of edges: those of length $2\sin(\frac\eps2)$ and those of length $2\cos(\frac\eps2)$. 
As we will take $\eps$ to $0$, we call the first \emph{short edges} and the latter \emph{long edges}. 
The parameters used to define $\phi_{\bbG^\eps}$ are given by~\eqref{eq:parameters}:
for short edges they are obtained by taking $\theta_e = \eps$, call the resulting value $p_\eps$; 
for long edges $\theta_e =\pi - \eps$, call the resulting value $p_{\pi-\eps}$.

When $\eps \rightarrow 0$, we have the following asymptotics,
$$
\def\arraystretch{2}
\begin{array}{>{\displaystyle}r>{\displaystyle}l>{\displaystyle}l}
  \text{if } 1 \leq q < 4, \quad
  &  \displaystyle
  1 - p_\eps \sim \frac{2r\eps}{\sqrt{q(4-q)}},
  & \quad p_{\pi-\eps} \sim \frac{2r \eps \sqrt q}{\sqrt{4-q}}; \\
  \text{if } q = 4, \quad 
  & 1- p_\eps \sim \frac{\eps}{2\pi},
  & \quad p_{\pi-\eps} \sim \frac{2}{\pi} \eps; \\
  \text{if } q > 4, \quad 
  & 1- p_\eps \sim \frac{2r\eps}{\sqrt{q(q-4)}},
  & \quad p_{\pi-\eps} \sim \frac{2r \eps \sqrt q}{\sqrt{q-4}}. \\
\end{array}
$$

Moreover, the length of a long edge converges to $2$, while that of short edges decreases as $\eps + o(\eps)$. 
Thus, in the limit $\eps \rightarrow 0$, the measure converges to the quantum FK model on a dilated lattice $2\bbZ \times \bbR$ with parameters 
$$
\def\arraystretch{2}
\begin{array}{>{\displaystyle}r>{\displaystyle}l>{\displaystyle}l}
  \text{if } 1 \leq q < 4, \quad
  &  \displaystyle
  \lambda_0 = \frac{2r}{\sqrt{q(4-q)}},
  & \quad \mu_0 = \frac{2r \sqrt q}{\sqrt{4-q}}; \\
  \text{if } q = 4, \quad 
  & \lambda_0 = \frac{1}{2\pi},
  & \quad \mu_0 = \frac{2}{\pi}; \\
  \text{if } q > 4, \quad 
  & \lambda_0 = \frac{2r}{\sqrt{q(q-4)}},
  & \quad \mu_0 = \frac{2r \sqrt q}{\sqrt{q-4}}. \\
\end{array}
$$
Note that $\lambda_0$ and $\mu_0$ are continuous in $q$: when $q$ goes to 4 either from above or from below, the common limits of $\lambda_0$ and $\mu_0$ are exactly the values given by $q = 4$.
A precise statement is given below in Proposition~\ref{prop:Qhv}.

For the rest of the section, unless otherwise stated, we consider the quantum random-cluster model $\phi_{\calQ}^\xi$ on $\bbZ \times \bbR$ with parameters $\lambda = 2 \lambda_0$ and $\mu = 2 \mu_0$
and boundary conditions $\xi = 0, 1$. This is simply the limiting model discussed above rescaled by a factor $1/2$. 
The infinite-volume measures with free and wired boundary conditions can be defined via weak limits as in the classical case.
The quantum model with these parameters enjoys a self-duality property similar to that of the discrete model on $\bbG^\eps$ with $\beta =1$. 

To distinguish the subgraphs of $\bbG^\eps$ from those of $\bbZ \times \bbR$, 
we shall always put a superscript $\eps$ for those of $\bbG^\eps$
and those of $\bbZ \times \bbR$ are always written in calligraphic letters.

For any subgraph $\calR$ of $\bbZ \times \bbR$, we write $\phi^\xi_{\calQ, \calR}$ for the quantum \rcm on $\calR$ with boundary conditions $\xi = 0, 1$.

\medbreak

For $1 \leq q \leq 4$, we will consider the counterparts of the horizontal and vertical crossing events given in Lemma~\ref{lem:equiv_RSW}. They are related to their discrete versions via the following convergence. 

\begin{prop} \label{prop:Qhv}
  For $a, b, c, d > 0$, let $\rect^\eps = \rect(2c, \tfrac{2d}{\eps})$ be a subgraph of $\bbG^\eps$ and $\calR = [-c, c] \times [-d, d]$ be a subgraph of $\bbZ \times \bbR$.
  Consider $\xi = 0, 1$, then
  \begin{align}
    & \phi^\xi_{\rect^\eps} \big[ \calC_h(2a; \tfrac{2b}{\eps}) \big] \xrightarrow[\eps \to 0]{}
    \phi^\xi_{\calQ, \calR} \big[ \calC_h(a; b) \big], \label{eq:Qh} \\
    & \phi^\xi_{\rect^\eps} \big[ \calC_v(2a; \tfrac{2b}{\eps}) \big] \xrightarrow[\eps \to 0]{}
    \phi^\xi_{\calQ, \calR} \big[ \calC_v(a; b) \big]. \label{eq:Qv}
  \end{align}
\end{prop}

For $q \geq 4$, the event to consider is that given by~\eqref{eq:exp_decay3}, which is related to the discrete models as follows.

\begin{prop} \label{prop:Qboundary}
  For any $N, n > 0$, let $\rect^\eps = \rect(2N, \tfrac{2N}{\eps})$ be a subgraph of $\bbG^\eps$ and $\calR = \Lambda(N) = [-N, N]^2$ be a subgraph of $\bbZ \times \bbR$.
  Then,
  \begin{align} \label{eq:Qboundary}
    \phi^\xi_{\rect^\eps} \big[ 0 \leftrightarrow \pd \rect(2n; 2\tfrac{n}{\eps}) \big]
    \xrightarrow[\eps \to 0]{}
    \phi^\xi_{\calQ, \calR} \big[ 0 \leftrightarrow \partial \Lambda(n) \big]
  \end{align}
\end{prop}

\begin{proof}[Proof of Propositions~\ref{prop:Qhv} and~\ref{prop:Qboundary}]
  As we described above, short edges in $\bbG^\eps$ are of length $2 \sin( \tfrac\eps2 )$, each of whom is closed with probability $\lambda_0 \eps$, where $\lambda_0 = \frac{2r}{\sqrt{q(4-q)}}$.
  Given $L > 0$, consider a collection of $N = \frac{L}{\eps}$ such consecutive edges.
  Consider $(X_i)_{1 \leq i \leq N}$ a sequence of i.i.d. Bernoulli random variables of parameter $\lambda_0 \eps$: $X_i = 1$ if the $i$-th edge is closed and $X_i = 0$ otherwise.
  Denote $S = \sum_{i=1}^N X_i$, which counts the number of closed edges in this collection of edges.

  Then, for any fixed $k \geq 0$ and $\eps \rightarrow 0$, we have that
  \begin{align*}
    \bbP [ S = k ] & = {N \choose k} (\lambda_0 \eps)^k (1- \lambda_0 \eps)^{N-k} \\
    & \sim \frac{N^k}{k!} (\lambda_0 \eps)^k e^{-N \lambda_0 \eps} \\
    & = e^{-L \lambda_0} \frac{(L \lambda_0)^k}{k!},
  \end{align*}
  where the quantity in the last line is the probability that a Poisson variable of parameter $L \lambda_0$ takes the value $k$.
  As a consequence, when $\eps \rightarrow 0$, the $N$ vertical short edges will converge to a vertical segment of length $L$, among which closed edges will give us cut points that can be described by a Poisson point process with parameter $\lambda_0$.

  The same reasoning applies to the long edges too.
  This shows that the measures $\phi^\xi_{R^\eps}$ converge weakly to $\phi^\xi_{\calQ, \calR}$, up to a scaling factor of $1/2$.
\end{proof}

Let us briefly discuss the strategy for proving Theorem~\ref{thm:quantum}, we restrict ourselves to the case $q \leq 4$ for illustration. 
In order to prove the RSW property for $\phi_\calQ$, 
one needs to bound uniformly the left hand sides of~\eqref{eq:Qh} and~\eqref{eq:Qv}
for $a = n$ and $b = \rho n$ for any fixed quantity $\rho$.
By duality, we may focus only on lower bounds.

Notice that for any fixed $\eps > 0$, the RSW property obtained in Theorem~\ref{thm:main} provides us with bounds for
$\phi^\xi_{R^\eps} \big[ \calC_h (2n; 2\rho \tfrac{n}{\eps} ) \big]$ and 
$\phi^\xi_{R^\eps} \big[ \calC_v (2n; \tfrac{2n}{\eps} ) \big]$ which are uniform in $n$. 
However, these are not necessarily uniform in $\eps$. 
Indeed, all estimates of Section~\ref{sec:q<4} crucially depend on angles being bounded uniformly away from $0$.

Removing this restriction in general is an interesting but difficult problem. 
However, in the simple case of the lattices $\bbG^\eps$, this is possible, and is done below. 

\subsection{The case $1 \leq q \leq 4$} \label{sec:details_hv_transport_quantum}


To show Theorem~\ref{thm:quantum} for $1 \leq q \leq 4$, it is enough to show the RSW property for the quantum model, the rest of the proof follows as in Section~\ref{sec:conclusion_q<=4}.
To this end, we proceed in the similar way as for isoradial graphs.
More precisely, the following proposition provides us with uniform bounds in $\eps \in (0, \pi)$ for crossing probabilities in $\bbG^\eps$.
Then Proposition~\ref{prop:Qhv} transfers these results to the quantum model and the same argument as in Lemma~\ref{lem:RSW_sq_cond} yields the RSW property for the quantum model.


\begin{prop} \label{prop:RSW_sq_cond_quantum}
  There exist $\delta > 0$, constants $a \geq 3$ and $b > 3a$  and $n_0$ such that, 
  for all $\eps \in (0,\pi)$ and $n\geq n_0$,  
  there exist boundary conditions $\xi$ on the region $\rect^\eps = \rect(bn; \tfrac{bn}{\eps})$ of $\bbG^{\eps}$ such that
  \begin{align}\nonumber
    \phi_{\rect^\eps}^\xi \big[ \calC_h(3an, bn; \tfrac{bn}{\eps}) \big] \geq 1- \delta/2 \quad 
    & \text{and} \quad \phi_{\rect^\eps}^\xi \big[ \calC_h^*( 3an , bn ; \tfrac{bn}{\eps}) \big] \geq 1- \delta/2, \nonumber \\
    \phi_{\rect^\eps}^\xi \big[ \calC_v(an; \tfrac{2n}{\eps}) \big] \geq \delta \quad
    & \text{and} \quad \phi_{\rect^\eps}^\xi \big[ \calC_v^*(an; \tfrac{2n}{\eps} ) \big] \geq \delta, \nonumber \\
    \phi_{\rect^\eps}^\xi \big[ \calC_h(an, 3an; \tfrac{n}{\eps} ) \big] \geq \delta \quad 
    & \text{and} \quad \phi_{\rect^\eps}^\xi \big[ \calC_h^*(an, 3an; \tfrac{n}{\eps} ) \big] \geq \delta.
    \label{eq:RSW_sq_cond_quantum}
  \end{align}
\end{prop}


For the rest of the section, we focus on proving Proposition~\ref{prop:RSW_sq_cond_quantum};
the rest of the arguments used to obtain the RSW property for $\phi_{\calQ}$ are standard. 

By symmetry, we may focus on $\eps \leq \pi/2$; thus, difficulties only appear as $\eps \to 0$.
Fix $\eps > 0$. We shall follow the same ideas as in Section~\ref{sec:iso_square_RSW}.
Recall the construction of the mixed lattice $G_\mix$:
for $M, N_1, N_2 > 0$, consider the graph obtained by superimposing 
a horizontal strip of $\bbG^\eps$ of height $N_2+1$ and width $2M+1$ over 
a horizontal strip of regular square graph $\bbG_{0, \frac\pi2}$ of height $N_1+1$ and same width.
Let the lower vertices of this graph be on the line $\bbR \times \{0\}$, with $x_{0,0}$ at $0$. 
Convexify this graph. The graph thus obtained, together with its reflection with respect to $\bbR \times \{0\}$, form $G_\mix$. 

Write $\tilde G_\mix$ for the graph with the regular and irregular blocks reversed. 
See Sections~\ref{sec:mix} and~\ref{sec:iso_square_RSW} for details on this construction and the track exchanging procedure
that allows us to transform $G_\mix$ into $\tilde G_\mix$. 
Figure~\ref{fig:quantum_gluing} contains an illustration of $G_\mix$ and $\tilde G_\mix$.

\begin{figure}[htb]
  \centering
  \includegraphics{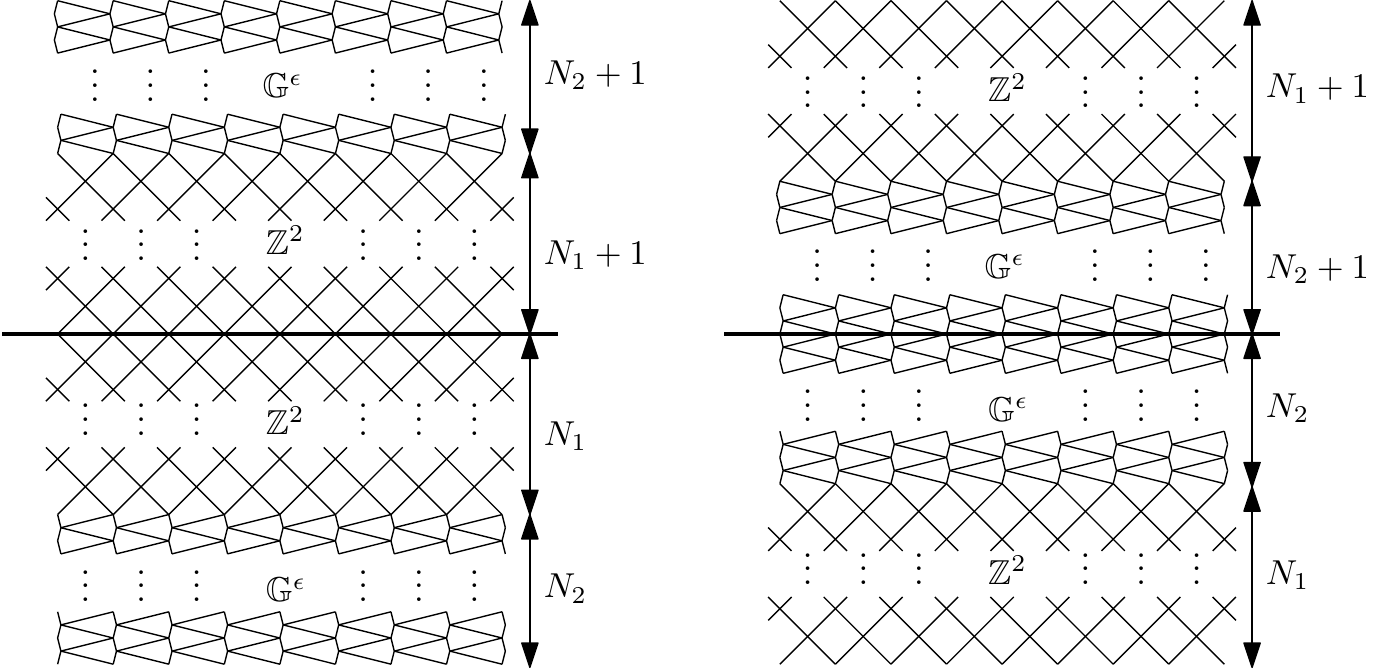}
  \caption{\textbf{Left: } The graph $G_\mix$. \textbf{Right: } The transformed graph $\tilde G_\mix$.}
  \label{fig:quantum_gluing}
\end{figure}


Write $\phi_{G_\mix}$ and $\phi_{\tilde G_\mix}$ for the random-cluster measures with $\beta=1$ on $G_\mix$ and $\tilde G_\mix$, respectively, with free boundary conditions. 
The following adaptations of Propositions~\ref{prop:horizontal_transport} and~\ref{prop:vertical_transport} imply Proposition~\ref{prop:RSW_sq_cond_quantum} as in Section~\ref{sec:path_transport}.

\begin{prop} \label{prop:Gepsilon_h}
  There exist $\lambda_0 := \lambda_0(q) >0$ and $n_0 \geq 1$ such that for all 
  $\lambda > \lambda_0$, $\eps \in (0,\pi/2]$, $\rho_{\rmout} > \rho_\rmin> 0$, $n \geq n_0$ 
  and sizes $M \geq (\rho_{\rmout} + \tfrac{\lambda}{\eps}) n$, $N_1 \geq n$ and $N_2 \geq \frac{\lambda}{\eps} n$, 
  \begin{align} \label{eq:Gepsilon_h}
    \phi_{\tilde G_\mix} \big[ 
    \calC_h(\rho_{\rmin} n, (\rho_{\rmout} + \tfrac{\lambda}{\eps}) n; \lambda \tfrac{n}{\eps}) \big] 
    \geq (1- \rho_{\rmout} e^{-n} ) 
    \phi_{G_\mix}\big[ \calC_h(\rho_{\rmin} n, \rho_{\rmout} n; n) \big].
  \end{align}
\end{prop}

The quantities $\lambda_0$ and $\lambda$ above have no relation to the intensity of the Poisson point process used in the definition of $\phi_\calQ$.

\begin{prop} \label{prop:Gepsilon_v}
  There exists $\eta > 0$ and a sequence $(c_n)_n \in (0,1]^\bbN$ with $c_n \to 1$ such that,
  for all $\eps \in (0,\pi/2]$, $n \geq1$ and sizes 
  $M\geq 3n$, $N_1 \geq N$ and $N_2 \geq \tfrac{n}{\eps}$, 
  \begin{align} \label{eq:Gepsilon_v}
    \phi_{\tilde G_\mix} \big[ \calC_v(3n; \eta \tfrac{n}{\eps}) \big] \geq 
    c_n \phi_{G_\mix} \big[ \calC_v(n; n) \big].
  \end{align}
\end{prop}

Proposition~\ref{prop:Gepsilon_h} controls the upward drift of a crossing: 
it claims that with high probability (independently of $\eps$), this drift (in the graph distance) is bounded by a constant times $\frac{1}{\eps}$.
As a result, due to the particular structure of $\bbG^\eps$, the upward drift in terms of the Euclidean distance is bounded by a constant independent of $\eps$.
The proof follows the same idea as that of Proposition~\ref{prop:horizontal_transport} 
with the difference that it requires a better control of~\eqref{eq:Hbound}, which is obtained by a coarse-graining argument.

Proposition~\ref{prop:Gepsilon_v} controls the downward drift of a vertical crossing.
The proof follows the same lines as that of Proposition~\ref{prop:vertical_transport}, 
with a substantial difference in the definition \eqref{eq:H_Delta} of the process $H$ 
which bounds the decrease in height of a vertical crossing when performing a series of track exchanges.
In the proof of Proposition~\ref{prop:vertical_transport}, $H$ was a sum of Bernoulli random variables; 
here the Bernoulli variables are replaced by geometric ones.
This difference may seem subtle, but is essential in obtaining a bound on the Euclidean downward drift which is uniform in $\eps$.

\begin{proof}[Proof of Proposition~\ref{prop:Gepsilon_h}]
  We keep the same notations as in the proof of Proposition~\ref{prop:horizontal_transport}.
  We remind that the process $(H^k)_{0 \leq k \leq n}$ is coupled with the evolution of $(\gamma^{(k)})_{0 \leq k \leq n}$ in such a way that all vertices $x_{i, j}$ visited by $\gamma^{(k)}$ have $(i, j) \in H^k$.
  Moreover, $(H^k)$ can be seen as a growing pile of sand which grows laterally by 1 at each time step and vertically by 1 independently at each column with probability $\eta$.
  The goal here is to estimate $\eta$ in the special case of $G_\mix$ as illustrated in Figure~\ref{fig:quantum_gluing}.
  More precisely, we want to improve the bound given in~\eqref{eq:eta_sandpile}.
  Let $\lambda > 0$ denote a (large) value, we will see at the end of the proof how it needs to be chosen.

  \begin{figure}[htb]
    \centering
    \includegraphics[scale=1]{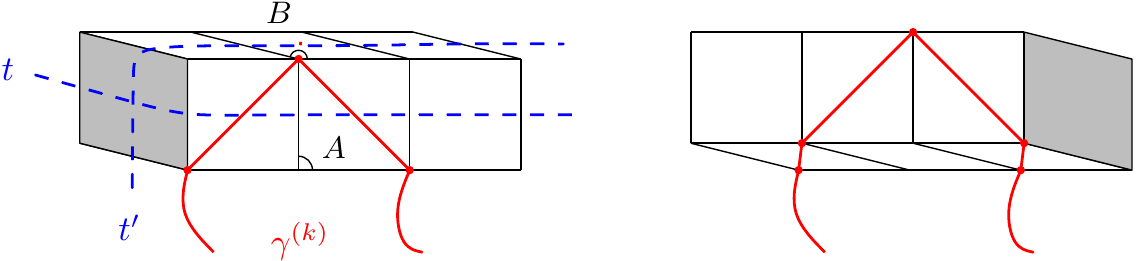}
    \caption{Figure~\ref{fig:sandpile} adapted to the case of $G_\mix$ defined above.
      When we perform \stts, we exchange two tracks, one of transverse angle $\frac\pi2$ and the other $\pi-\eps$ or $\eps$.
      We assume that we are in the case $A = \frac\pi2$ and $B = \pi-\eps$.}
    \label{fig:quantum_sandpile}
  \end{figure}

  In the special case of $G_\mix$ described above, a track exchange is always between tracks with transverse angles $A = \frac\pi2$ and $B = \pi-\eps$ or $\eps$.
  This is illustrated in Figure~\ref{fig:quantum_sandpile}.
  The parameter $\eta_{A, B}$ is the probability that the path $\gamma^{(k)}$, as shown in the figure, drifts upwards by 1 when the dashed edge on the left of the figure is open.
  This can be estimated as follows for $1 \leq q < 4$,
  \begin{align*}
    \eta_{A, B} & = \frac{y_{\pi - A} y_{\pi - (B-A)}}{q} =
    \left\{
      \begin{array}{ll}
        \frac{\sin(r(\tfrac\pi2 - \eps))}{\sin(r(\tfrac\pi2 + \eps))} & \text{if } 1 \leq q < 4, \\
        \frac{\tfrac\pi2 - \eps}{\tfrac\pi2 + \eps} & \text{if } q = 4.
      \end{array}
    \right.
  \end{align*}

  A quick computation then shows that we have
  \begin{align*}
    \eta := \sup_{A, B \in [\eps, \pi-\eps]} \eta_{A, B} =  1 - \zeta(q)  \eps,
    \quad \text{ where } \zeta(q) = \frac{2 \sqrt{2+\sqrt q}}{\pi^2}.
  \end{align*}
  Thus, this value $\eta$ may be used in the process $(H^k)$ bounding the evolution of $(\gamma^{(k)})$.
  Hence, we obtain, for $0 \leq k \leq \frac{\lambda}{\eps}$,
  \begin{align*}
    \bbP \big[  h(\gamma^{(k)}) < \tfrac{\lambda}{\eps} n \big]
    \geq \bbP \big[ \max \{ j : (i, j) \in H^{k} \} < \tfrac{\lambda}{\eps} n \big].
  \end{align*}

  A straightforward application of~\eqref{eq:Hbound} is not sufficient to conclude, 
  as it would provide a value of $\lambda$ of order $\log(\frac{1}{\eps})$ rather than of constant order. 
  We will improve~\eqref{eq:Hbound} slightly by revisiting its proof (given in~\cite[Lem.~3.11]{GriMan13}).
  
  We are interested in the time needed to add a neighboring block from those which are already included by $H^k$.
  To be more precise, with each edge $e$ of $\bbZ \times \bbN$, we associate a time $t_e$: if $e$ is horizontal, set $t_e = 1$; if $e$ is vertical, set $t_e$ to be a geometric random variable with parameter $\zeta(q) \eps$.
  Moreover, we require that the random variables $(t_e)$ are independent.

  For $x, y \in \bbZ \times \bbN$, define $\calP(x, y)$ to be the set of paths going from $x$ to $y$, containing no downwards edge.
  For a path $\gamma \in \calP(x, y)$, write $\tau(\gamma) = \sum_{e \in \gamma} t_e$, 
  which is the total time needed to go through all the edges of $\gamma$.
  Also write $\tau(x, y) = \inf \{ \tau(\gamma) : \gamma \in \calP(x, y) \}$.
  As such, the sets $(H^k)_{k \geq 0}$ can be described by
  \begin{align} \label{eq:Hk_paths}
    H^k = \{ y \in \bbZ \times \bbN : \exists x \in H^0, \tau(x, y) \leq k \},
  \end{align}
  where we recall the definition of $H^0$:
  $$
  H^0 = \{ (i, j) \in \bbZ \times \bbN : -(\rho_{\rmout} + 1) n \leq i-j \text{ and } i+j \leq (\rho_{\rmout} + 1) n \text{ and } j \leq n \}.
  $$
  Thus, our goal is to prove that the time needed to reach any point at level $\frac{\lambda}{\eps} n$ 
  is greater than $\frac{\lambda}{\eps} n$ with high probability.

  Let $\calP_n$ be set of all paths of $\calP(x, y)$ with $x \in H^0$, $y = (i, j)$ with $j = \tfrac\lambda\eps n$
  and which contain at most $n$ horizontal edges.
  Then,~\eqref{eq:Hk_paths} implies
  \begin{align} \label{eq:Hbound_quantum}
    \bbP \big[ \max \{j : (i, j) \in H^{\frac{\lambda}{\eps} n} \} \geq \tfrac{\lambda}{\eps} n \big]
    = \bbP \big[ \exists \gamma \in \calP_n : \tau(\gamma) \leq \tfrac{\lambda}{\eps} n \big].
  \end{align}

  Next comes the key ingredient of the proof.
  We define the notion of \emph{boxes} as follows.
  For $(k, \ell) \in \bbZ \times \bbN$, set 
  $$
  B(k, \ell) = \{ (k, j) \in \bbZ \times \bbN : \tfrac{\ell-1}{\eps} \leq j < \tfrac{\ell}{\eps} \}.
  $$
  We note that different boxes are disjoint and each of them contains $\frac1\eps$ vertical edges.
  A sequence of adjacent boxes is called a \emph{box path}.
  Note that such a path is not necessarily self-avoiding.
  Set $\tilde \calP_n$ to be the set of box paths from some $B(k, n)$ to some $B(\ell, \lambda n)$, where $k$ and $\ell$ are such that $-(\rho_{\rmout} +1) n \leq k \leq (\rho_{\rmout} +1)n$ and $|k - \ell| \leq n$, and in which at most $n$ pairs of consecutive boxes are adjacent horizontally.

  With any path $\gamma \in \calP_n$, we associate the box path $\tilde \gamma \in \tilde \calP_n$ of boxes visited by $\gamma$ (above level $n/\eps$).
  Notice that, since $\gamma$ has at most $n$ horizontal edges, so does $\tilde \gamma$.


  Given a box path $\tilde \gamma = (\tilde \gamma_i) \in \calP_n$, call $\tilde \gamma_i$ a \emph{vertical} box if $\tilde \gamma_{i-1}, \tilde \gamma_i$ and $\tilde \gamma_{i+1}$ have the same horizontal coordinate.
  Since any path $\tilde \gamma \in \tilde \calP_n$ can only have at most $n$ pairs of consecutives boxes that are adjacent horizontally, there are at least $(\lambda-3)n$ vertical boxes in $\tilde \gamma$.
  See Figure~\ref{fig:box_path} for an illustration of the above notions.

  \begin{figure}[htb]
    \centering
    \includegraphics[scale=1.2]{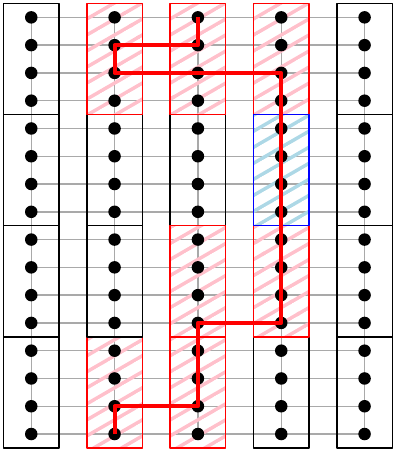}
    \caption{A path $\gamma \in \calP_n$ and its associated box path $\tilde \gamma \in \tilde \calP_n$.
      Note that the box path $\tilde \gamma$ might not be self-avoiding.
      The blue box in the figure is a vertical box for $\tilde \gamma$.}
    \label{fig:box_path}
  \end{figure}

  A box is called \emph{bad} if it contains a vertical edge $e$ such that $t_e \geq 2$.
  We can estimate the probability that a box is bad:
  \begin{align*}
    \bbP[ \text{a given box is bad} ] = 1 - (1- \zeta \eps)^{1/\eps} \to 1 - e^{-\zeta}
    \quad \text{ as } \eps \rightarrow 0.
  \end{align*}
  Notice that for a path $\gamma \in \calP_n$ such that $\tau(\gamma) \leq \frac\lambda\eps n$, there are at most $n$ vertical bad boxes in $\tilde \gamma$.
  Indeed, due to the definition, in any bad vertical box of $\tilde \gamma$, $\gamma$ crosses an edge $e$ with $t_e \geq 2$.
  Therefore,
  \begin{align} 
    \bbP \big[ \exists \gamma \in \calP_n : \tau(\gamma) \leq \tfrac{\lambda}{\eps} n \big]
    & \leq 
    \bbP \big[ \exists \tilde \gamma \in \tilde \calP_n : \text{ there are at most } n \text{ vertical bad boxes in } \tilde \gamma \big] \nonumber \\
    & \leq \sum_{\tilde \gamma \in \tilde \calP_n} \bbP \big[ \text{there are at most } n \text{ vertical bad boxes in } \tilde \gamma \big]\label{eq:Hbound_rewrite}
  \end{align}
  
  Let us now bound the above. 
  First, note that any path $\tilde \gamma \in \tilde\calP_n$ has length at most $\lambda n$ 
  and can have at most $n$ horizontal displacements, which gives
  \begin{align} \label{eq:path_bound}
    |\tilde \calP_n | \leq 2(\rho_\rmout + 1) n 2^n { \lambda n \choose n }
    \leq \rho_\rmout (c \lambda)^n, 
  \end{align}
  for a constant $c > 0$ independent of all the other parameters.
  Secondly, recall that any $\tilde\gamma\in \tilde\calP_n$ has at least $(\lambda-3)n$ vertical boxes. 
  Let $X_1,\dots, X_{(\lambda-3)n}$ be i.i.d. Bernoulli random variables with parameter $\delta = 1 - (1- \zeta \eps)^{1/\eps} $
  that indicate whether the $(\lambda-3)n$ first vertical boxes of $\tilde \gamma$ are bad 
  ($X_i = 1$ if the $i^{\text{th}}$ vertical box of $\tilde \gamma$ is bad and $X_i=0$ otherwise). 
  Then,
  %
  %
  \begin{align}
    \bbP \big[ \text{there are at most } n \text{ vertical bad boxes in } \tilde \gamma \big]
    & \leq \bbP \big[ X_1 + \cdots + X_{(\lambda-3)n} \leq n \big] \nonumber \\
    & \leq \Big[ 
    \Big( \frac{(\lambda-3) \delta}{\lambda-4} \Big)^{\lambda-4}
    (\lambda-3) (1-\delta)
    \Big]^n.
    \label{eq:proba_vb}
  \end{align}
  The last inequality is obtained by large deviation theory.

  Finally, put~\eqref{eq:Hbound_quantum}--\eqref{eq:proba_vb} together, as in \cite[Lem.~3.11]{GriMan13}.
  It follows that if $\lambda$ is chosen larger than some threshold $\lambda_0 > \frac{4-3\delta}{1-\delta}$ that only depends on $\delta$, 
  then
  \begin{align*}
    \bbP \big[ \max \{j : (i, j) \in H^{\frac{\lambda}{\eps} n} \} \geq \tfrac{\lambda}{\eps} n \big]
    \leq \rho_\rmout e^{-n}.
  \end{align*}
  Recall that $\delta \xrightarrow[\eps\to0]{} 1 - e^{-\zeta}$ is uniformly bounded in $\eps > 0$, 
  hence $\lambda_0$ may also be chosen uniform in $\eps$. 

  \begin{rmk}
    We point out that the coarse-graining argument above is essential to the proof due to the reduced combinatorial factor~\eqref{eq:path_bound}.
    In effect, the computation in~\cite[Lem.~3.11]{GriMan13} would have given us a combinatorial factor $(c \lambda / \eps)^n$, and due to the additional $\eps$ in the denominator, one can only show that $\lambda$ should grow as $\log(\frac{1}{\eps})$.
    This improvement is made possible because when, $\eps$ goes to 0, paths in the directed percolation take $\frac1\eps$ more vertical edges than horizontal ones, and bad edges (those with passage-time greater than 2) are of density proportional to $\eps$.
    We can therefore ``coarse grain'' a good number of paths to a unique one, which improves the bound.
  \end{rmk}
\end{proof}


\begin{proof}[Proof of Proposition~\ref{prop:Gepsilon_v}]
  We will adapt the proof of Proposition~\ref{prop:vertical_transport} to our special setting.
  The goal is to have a better control of the downward drift of paths when track exchanges are performed.
  There are two significant differences: 
  (i) a better description of the regions $D^k$ in which vertical paths are contained;
  (ii) a (stochastic) lower bound on $h^k$ by a sum of geometric random variables rather than a sum of Bernoulli variables as in the aforementioned proof.

  In this proof we are only interested in events depending on the graph above the base level, and we will only refer to the upper half-plane henceforth.

  Fix $\eps > 0$ and $n \in \bbN$.
  Let $M \geq 4 n$, $N_1 \geq n$ and $N_2 \geq \frac{n}{\eps}$ where, as illustrated in Figure~\ref{fig:quantum_gluing}, 
  $2M+1$ is the width of the blocks of $\graph{1} = \bbZ^2$ and $\graph{2} = \bbG^\eps$, $N_1+1$ is the height of the block of $\bbZ^2$ and $N_2$ that of $\bbG^\eps$.
  Recall that the sequence of \stts we consider here is $\Sigma^\uparrow$, which consists of pulling up tracks of $\graph{1}$ one by one above those of $\graph{2}$, from the top-most to the bottom-most.
  We write $\bbP$ for the measure taking into account the choice of a configuration $\omega$ according to the \rcm $\phi_{G_\mix}$ as well as the results of the \stts in $\Sigma^\uparrow$ applied to the configuration $\omega$.

  For $0 \leq i \leq N_1$, recall from Section~\ref{sec:mix} the notation 
  \begin{align*}
    \Sigma^\uparrow_{i} = \Sigma_{t_{i}, t_{N_1+N_2+1}} \circ \cdots \circ \Sigma_{t_{i}, t_{N_1+1}}, 
  \end{align*}
  for the sequence of \stts moving the track $t_{i}$ of $\graph{1}$ above $\graph{2}$.
  Then, $\Sigma^\uparrow  = \Sigma^\uparrow_0 \circ \cdots \circ \Sigma^\uparrow_{N_1}$.

  We note that $\omega \in \calC_v(n; n)$ if and only if $\Sigma^\uparrow_{n+1} \circ \cdots \circ \Sigma^\uparrow_{N_1} (\omega) \in \calC_v(n; n)$, 
  since the two configurations are identical between the base $t_0$ and $t_n$.
  Thus, we can assume that $\Sigma^\uparrow_{N_1}, \cdots, \Sigma^\uparrow_{n+1}$ are performed 
  and look only at the effect of $\Sigma^\uparrow_{n}, \cdots, \Sigma^\uparrow_0$ on such a configuration.
  Let us define for $0 \leq k \leq n+1$,
  \begin{align*}
    G^k & = \Sigma^\uparrow_{n-k+1} \circ \cdots \circ \Sigma^\uparrow_{N_1} (G_\mix), \\
    \omega^k & = \Sigma^\uparrow_{n-k+1} \circ \cdots \circ \Sigma^\uparrow_{N_1} (\omega), \\
    D^k & = \{ x_{u, v} \in G^k : |u| \leq n +2k, \,0 \leq v \leq N_2+n \}, \\
    h^k & = \sup \{ h \leq N_2 + n-k : \exists u, v \in \bbZ \text{ with } x_{u, 0} \xleftrightarrow{D^k, \omega^k} x_{v, h}  \}.
  \end{align*}
  That is, $h^k$ is the highest level that may be reached by an $\omega^{k}$-open path lying in the rectangle~$D^k$.
  These notions are illustrated in Figure~\ref{fig:Dk_quantum}

  \begin{figure}[htb]
    \centering
    \includegraphics[width=1\textwidth]{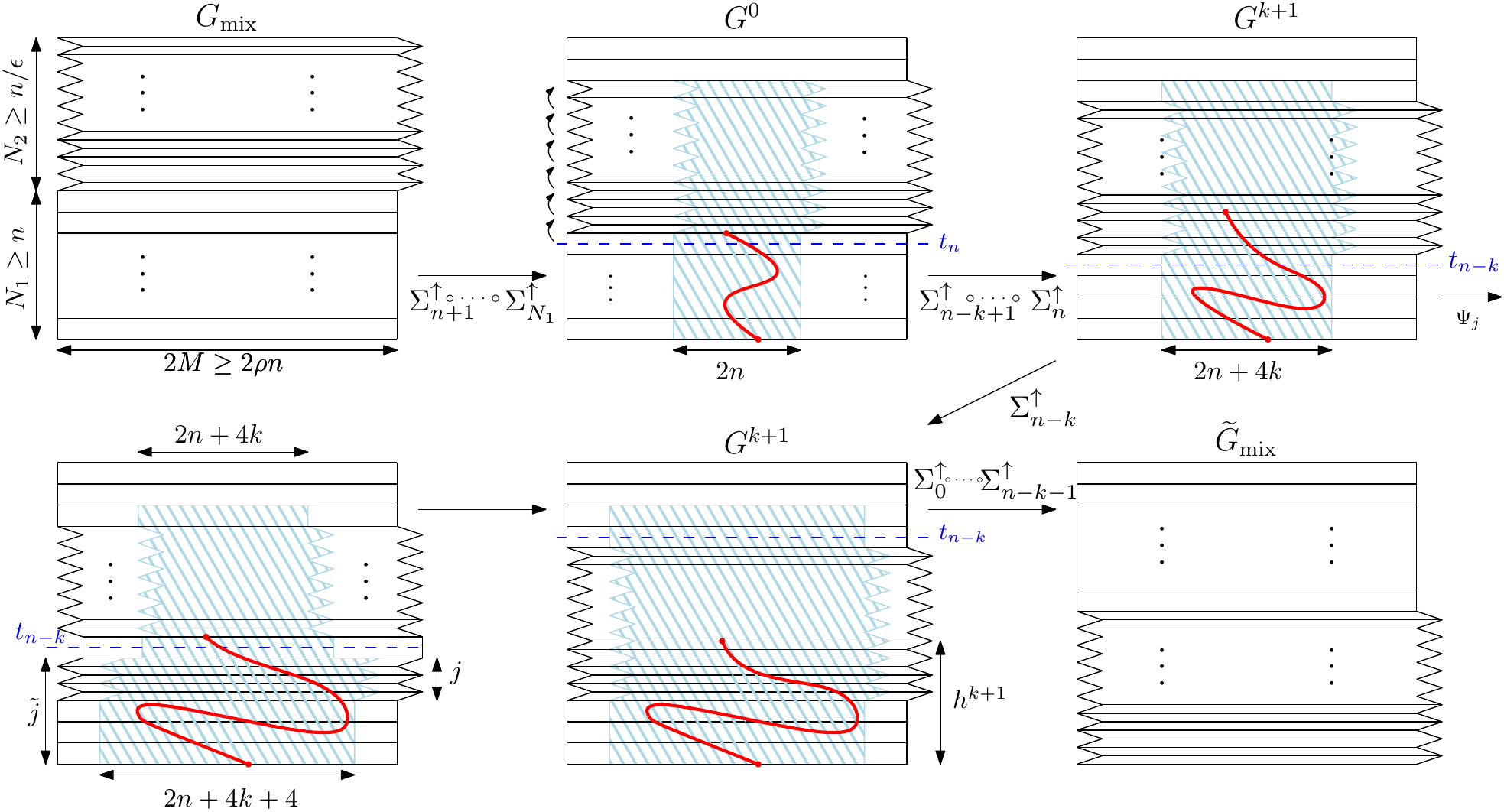}
    \caption{Several stages in the transformation of $G_\mix$ (only the outlines of the diamond graphs are depicted). 
      Pulling up the top $N_1 - n$ tracks of the regular lattice does not affect the event $\calC_v(n; n)$.
      The red vertical crossing is then affected by the track exchanges.
      However, it remains in the hashed domains $(D^k)_{0 \leq k \leq n+1}$, and $(D^k_j)_{0\leq j\leq N_2}$. 
      Its height evolves according to~\eqref{eq:hjk1}--\eqref{eq:hjk3},~\eqref{eq:hjk22} and~\eqref{eq:hjk42}.
      Notice the asymmetric shape of $D^k_j$ in the fourth diagram, where $j$ is even. }
    \label{fig:Dk_quantum}
  \end{figure}

  Due to the above definitions, if $\omega^0 \in \calC_v(n; n)$, then $h^0 \geq n$.
  Hence,
  \begin{align*}
    \bbP [ h^0 \geq n ] \geq \bbP [\omega^0 \in \calC_v(n; n)] =  \phi_{G_\mix}[\calC_v(n; n)].
  \end{align*}
  Moreover, using the fact that $\omega^{n+1}$ follows the law of $\phi_{\tilde G_\mix}$ and the definitions of $D^{n+1}$ and $h^{n+1}$ above, we obtain
  \begin{align*}
    \phi_{\tilde G_\mix}[ \calC_v(3n; \eta \tfrac{n}{\eps}) ] \geq \bbP(h^{n+1} \geq \eta \tfrac{n}{\eps}).
  \end{align*}
  Therefore, it is enough to show
  \begin{align}\label{eq:goal_Gepsilon_v}
    \bbP \big[ h^{n+1} \geq \eta \tfrac{n}{\eps} \big]  \geq c_n \bbP \big[ h^0 \geq n \big],
  \end{align}
  for some $\eta \in (0, \frac12)$ to be specified below and constants $c_n$ with $c_n \rightarrow 1$ as $n \rightarrow \infty$, 
  all independent of $\eps$.

  \medskip

  Fix $0 \leq k \leq n$ and let us examine the $(N_1-(n-k)+1)^{th}$ step of $\Sigma^\uparrow$, that is $\Sigma^\uparrow_{n-k}$.
  Write $\Psi_j := \Sigma_{t_{n-k}, t_{N_1+j}} \circ \cdots \circ \Sigma_{t_{n-k}, t_{N_1+1}}$ for $0 \leq j \leq N_2+1$.
  In other words, $\Psi_j$ is the sequence of \stts that applies to $G^k$ and moves the track $t_{n-k}$ 
  above $j$ tracks of $\graph{2}$, namely $t_{N_1+1}, \dots, t_{N_1+j}$.
  Moreover, $\Psi_{N_2} = \Sigma^\uparrow_{n-k}$; hence, $\Psi_{N_2} (G^{k}) = G^{k+1}$ and $\Psi_{N_2} (\omega^k) = \omega^{k+1}$.




  For $0 \leq j \leq N_2$ write $\sfj := n-k+j$ and define $D^k_j$ as the subgraph of $\Psi_j(G^k)$ induced by vertices $x_{u, v}$ with $0 \leq v \leq N_2 + n$ and
  \begin{align*}
    \left\{
      \begin{array}{rcll}
        & |u| &\leq n+2k+2 \qquad\qquad& \text{ if } v \leq \sfj ,\\
        -(n+2k) \leq & u &\leq n+2k+1 & \text{ if } v = \sfj+1 \text{ and } j \text{ odd},\\
        -(n+2k+1) \leq & u &\leq n+2k & \text{ if } v = \sfj+1 \text{ and } j \text{ even},\\
        & |u| &\leq  n+2k & \text{ if } v > \sfj+1.
      \end{array}
    \right.
  \end{align*}
  We note that $D^k \subseteq D^k_0 \subseteq \cdots \subseteq D^k_{N_2} \subseteq D^{k+1}$.
  Let $\omega^k_j = \Psi_j(\omega^k)$ and
  $$
  h^k_j = \sup \{ h \leq N_2+n-k : \exists u, v \in \bbZ \text{ with } x_{u, 0} \xleftrightarrow{D^k_j, \omega^k_j} x_{v, h} \}.
  $$
  Due to inclusions between the domains, we have $h^k \leq h^k_0$ and $h^k_{N_2} \leq h^{k+1}$.
  Next, we aim to obtain similar equations to~\eqref{eq:hjk1}--\eqref{eq:hjk4}.

  \medskip

  Fix $0 \leq j \leq N_2$ and let ${\mathsf \Sigma} := \Sigma_{t_{n-k}, t_{N_1+j+1}}$ be the track exchange to be applied to $\Psi_j(G^{k})$.
  Moreover, let $\calP_j$ be the set of paths $\gamma$ of $\Psi_j(G^{k})$, contained in $D_j^k$, 
  with one endpoint at height $0$, the other at height $h(\gamma)$, and all other vertices with heights between $1$ and $h(\gamma)-1$.
  
  First we claim that, if $\gamma$ is an $\omega^k_j$-open path of $\calP_j$, then ${\mathsf \Sigma} (\gamma)$ is $\omega_{j+1}^k$-open and contained in $D^k_{j+1}$ (hence contains a subpath of $\calP_{j+1}$ reaching the same height as ${\mathsf \Sigma} (\gamma)$).
  Due to the specific structure of $\bbG^\eps$, we prove this according to the parity of $j$.
  For $j$ even, the transverse angle of the track $t_{N_1+j+1}$ is $\pi-\eps$.
  Thus, as shown by the blue points in Figure~\ref{fig:path_transformations}, ${\mathsf \Sigma}$ induces a possible horizontal drift of $\gamma$ of $+2$ at level $\sfj$ 
  and $+1$ at level $\sfj+1$.
  By its definition, $D^k_{j+1}$ indeed contains ${\mathsf \Sigma} (\gamma)$.
  For $j$ odd, the figure is symmetric, thus we get horizontal drifts of $-2$ and $-1$ at levels $\sfj$ and $\sfj+1$, respectively.

  Let us briefly comment on the differences between the above and the general case appearing in Proposition~\ref{prop:vertical_transport}.
  In Proposition~\ref{prop:vertical_transport}, since the directions of the track exchanges are not necessarily alternating, 
  we may repeatedly obtain horizontal drifts of the same sign.
  This is why the domains $D^k_j$ in Proposition~\ref{prop:vertical_transport} grow with slope 1 (see Figure~\ref{fig:Dk_classical}), 
  and eventually induce a different definition of $D^k$ than the one above.
  In the present case, due to alternating transverse angles which create alternating positive and negative drifts, 
  $D^k$ may be chosen with vertical sides and $D^{k+1}$ is obtained by adding two columns on the left and right of $D^k$.
  While this may seem an insignificant technicality, it allows to bound the horizontal displacement of the vertical crossing
  by $2n$ rather than a quantity of order $\frac{n}{\eps}$, and this is essential for the proof. 

  \medskip

  As a consequence of the discussion above, equations~\eqref{eq:hjk1}--\eqref{eq:hjk3} hold as in the classical case.
  For this proof, we will improve~\eqref{eq:hjk2} and~\eqref{eq:hjk4} to
  \begin{align}
    \bbP [ h_{j+1}^k \geq h +1 \,|\, h_j^k = h ] & \geq 1-C\eps & \text{ if } h = \sfj, \label{eq:hjk22} \\
    \bbP [ h_{j+1}^k \geq h \,|\, h_j^k = h ] & \geq 1-C\eps & \text{ if } h = \sfj +1, \label{eq:hjk42}
  \end{align}
  for some constant $C> 0$ that does not depend on $\eps$, only on $q$. 
  
  Before going any further, let us explain the meaning of~\eqref{eq:hjk1}--\eqref{eq:hjk3}, \eqref{eq:hjk22} and~\eqref{eq:hjk42} through a non-rigorous illustration.
  In applying $\Sigma^\uparrow_{n-k}$, the track $t_{n-k}$ (of transverse angle $\pi/2$) is moved upwards progressively. 
  Let us follow the evolution of a path $\gamma$ reaching height $h^k$ throughout this process. 
  As long as the track $t_{n-k}$ does not reach height $h^k$, the height reached by $\gamma$ is not affected. 
  When $t_{n-k}$ reaches height $h^k$ (as in Figure~\ref{fig:quantum_growth4}; left diagram),
  the height of $\gamma$ may shrink by $1$ or remain the same; \eqref{eq:hjk42} indicates that the former arrises with probability bounded above by~$C\eps$.
  If $\gamma$ shrinks, the following track exchanges do not influence $\gamma$ any more, and we may suppose $h^{k+1} = h^k -1$. 
  Otherwise, the top endpoint of $\gamma$ at the following step is just below $t_{n-k}$ (as in Figure~\ref{fig:quantum_growth4}; centre diagram).
  In the following track exchange, $\gamma$ may increase by $1$ in height.
  By~\eqref{eq:hjk22}, this occurs with probability $1-C\eps$. 
  If the height of $\gamma$ does increase, then it is again just below $t_{n-k}$, and it may increase again. 
  In this fashion, $\gamma$ is ``dragged'' upwards by $t_{n-k}$.
  This continues until $\gamma$ fails once to increase. 
  After this moment, $\gamma$ is not affected by any other track exchange of~$\Sigma^\uparrow_{n-k}$.

  The reasoning above would lead us to believe that $h^{k+1} \geq h^k - 2 +Y$ stochastically, where $Y$ is a geometric random variable with parameter $C\eps$.
  This is not entirely true since the conditioning in~\eqref{eq:hjk22} and~\eqref{eq:hjk42} is not on $\omega_j^k$, but only on $h_j^k$. 
  However, this difficulty may be avoided as in the proof of Proposition~\ref{prop:vertical_transport}.
  Let us render this step rigorous and obtain the desired conclusion~\eqref{eq:goal_Gepsilon_v}, before proving~\eqref{eq:hjk22} and~\eqref{eq:hjk42}.

  Let $(Y_k)_{0 \leq k\leq n}$ be i.i.d. geometric random variables of parameter $C\eps$. 
  Define the Markov process $(H^k)_{0 \leq k\leq n+1}$ by $H^0 =h^0$ and $H^{k+1} =  \min\{H^k +Y_k-2, n-k + N_2\}$.
  Then, the comparison argument of~\cite[Lem.~3.7]{GriMan13} proves that $h^k$ dominates $H^k$ stochastically for any $k$. 
  Precisely, for any $k$, the processes $(H^j)_{0 \leq j \leq n+1}$ and $(h^j)_{0 \leq j \leq n+1}$ may be coupled such that $H^k \leq h^k$ a.s.. 
  We insist that we do not claim that there exists a coupling that satisfies the above inequality simultaneously for all $k$. 
  We do not provide details on how to deduce this inequality from~\eqref{eq:hjk1}--\eqref{eq:hjk3}, \eqref{eq:hjk22} and~\eqref{eq:hjk42}, 
  since this step is very similar to the corresponding argument in~\cite[Lem.~6.9]{GriMan14}. 
  Let us simply mention that the cap of $n-k + N_2$ imposed on $H^k$ 
  comes from the fact that a path may not be dragged upwards above the highest track of the irregular block. 

  By comparing $h^{n+1}$ and $H^{n+1}$, we find
  \begin{align*}
    c_n:=\frac{\bbP \big[ h^{n+1} \geq \eta \tfrac{n}{\eps} \big]}{\big[ h^0 \geq n \big]}
    \geq \bbP \big[ H^{n+1} \geq \eta \tfrac{n}{\eps} \,\big|\,H^0 \geq n \big] 
    \geq \bbP \big[ Y_0 + \dots + Y_n -(n+2) \geq \eta \tfrac{n}{\eps} \big].
  \end{align*}
  The last inequality is due to that, if $H^k +Y_k-2 = n-k + N_2$ at any point during the process, then $h^{n+1} \geq \eta \tfrac{n}{\eps}$ is sure to arrises. 
  Finally notice that $\bbE[ Y_k - 1 ] = \frac{1 - C\eps}{C\eps} > \eta/\eps$ for $\eta < 1/C$ and $\eps$ small enough. 
  The same large deviation argument as in the final step of the proof of Proposition~\ref{prop:vertical_transport} allows us to conclude that $c_n \to 1$, uniformly in $\eps$.
  Thus, we are only left with proving~\eqref{eq:hjk22} and~\eqref{eq:hjk42}, which we do next. 
  \medskip 
  
  \noindent{\bf First we prove~\eqref{eq:hjk42}}. 
  This is similar to the argument proving~\eqref{eq:hjk4}, with a slight improvement on the estimate of the parameter $\delta$. 
  Fix $0 \leq j \leq N_2$ and use the notation introduced above. 
  Without loss of generality, assume also that $j$ is even so that the track exchange ${\mathsf \Sigma} = \Sigma_{t_{n-k}, t_{N_1+j+1}}$ is performed from left to right.
  Denote by $\Gamma = \Gamma(\omega^k_j)$ the $\omega^k_j$-open path of $\calP_j$ that is the minimal element of 
  $\{ \gamma \in \calP_j : h(\gamma) = h^k_j, \gamma \text{ is } \omega^k_j \text{-open} \}$ as in Proposition~\ref{prop:vertical_transport}.
  
  \begin{figure}[htb]
    \centering
    \includegraphics[scale=0.9]{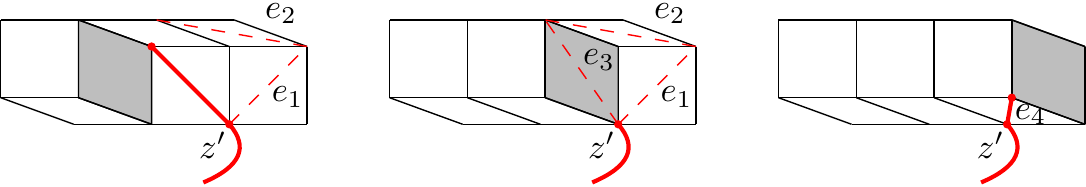}
    \caption{Three \stts contributing to ${\mathsf \Sigma}$ slide the gray rhombus from left to right.
      If the three edges in the middle diagram are all closed, then $e_4$ is open with probability $\frac{y_{e_1}y_{e_3}}{q}$.}
    \label{fig:quantum_growth4}
  \end{figure}
  
  Fix some $\gamma \in \calP_j$ of height $\sfj +1$. 
  Let $z = x_{u, \sfj+1}$ denote the upper endpoint of $\gamma$ and let $z'$ denote the other endpoint of the unique edge of $\gamma$ leading to $z$. Then either $z' = x_{u+1, \sfj}$ or $z' = x_{u-1, \sfj}$.
  
  Conditioning on $\Gamma = \gamma$.
  If $z' = x_{u-1, \sfj}$, then it is always the case that $h({\mathsf \Sigma}(\Gamma) ) \geq \sfj+1$.
  Assume that $z' = x_{u+1, \sfj}$ as in Figure~\ref{fig:quantum_growth4} and consider the edges $e_1, \dots, e_4$ depicted in the image. 
  If $e_1$ is open in $\omega^k_j$ then it is easy to see that $h_{j+1}^k= \sfj +1$, for any outcome of the star-triangle transformations. 
  The same is valid for the edge $e_3$ appearing in the second diagram of Figure~\ref{fig:quantum_growth4}.
  Assume that both $e_1$ and $e_3$ are closed in the second diagram. 
  Then, if in addition $e_2$ is also closed, by the randomness appearing in the star-triangle transformation leading to the fourth diagram,
  \begin{align*}
    \bbP [ e_4 \text{ is open} \,|\, e_1, e_2, e_3 \text{ are closed} ] \geq \frac{y_{e_1}y_{e_3}}{q}.
  \end{align*}
  This is due to the transition probabilities of Figure~\ref{fig:simple_transformation_coupling}.
  Finally, if $e_4$ is open in the last diagram, then the height of $\Gamma$ remains at least $\sfj+1$ for the rest of ${\mathsf \Sigma}$.
  In conclusion,
  \begin{align*}
    \bbP [ h_{j+1}^k \geq \sfj +1 \,|\, \Gamma = \gamma ]
    \geq \frac{y_{e_1}y_{e_3}}{q}
    \bbP [ e_2 \text{ is closed} \,|\, \Gamma = \gamma \, \text{ and } e_1, e_3 \text{ closed} ]
  \end{align*}
  Notice that the edge $e_2$ is above level $\sfj$, hence the conditioning $\Gamma = \gamma$ and $e_1$, $e_3$ closed affects it negatively. 
  Thus, $\bbP [ e_2 \text{ is closed} \,|\, \Gamma = \gamma\, \text{ and } e_1, e_3 \text{ closed} ] \geq 1-p_{e_2}$.
  Using the fact that $y_{e_1} = \sqrt q$, $y_{e_3} \to \sqrt q$ and $p_{e_2} \sim C' \eps$ as $\eps \rightarrow 0$, with $C' > 0$ depending only on $q$, and summing over all possibilities on $\gamma$, we obtain
  \begin{align*}
    \bbP [ h_{j+1}^k \geq \sfj +1 \,|\, h_j^k = \sfj +1 ] & \geq 1 - C\eps,
  \end{align*}  
  for some constant $C$ depending only on $q$. 
  
  \medskip  

  \noindent{\bf Now let us prove~\eqref{eq:hjk22}}.
  The argument is very similar to the above.
  Assume again that the track exchange ${\mathsf \Sigma} = \Sigma_{t_{n-k}, t_{N_1+j+1}}$ is performed from left to right.
  Let $\Gamma$ be defined as above and call $z$ its top endpoint. 
  Condition on $h^k_j = \sfj$. 
  Then, as the vertical rhombus is slid through $t_{n-k}, t_{N_1+j+1}$, it arrives above $z$ as in the second diagram of Figure~\ref{fig:quantum_growth4}. 
  Let $e_1$, $e_2$, $e_3$ and $e_4$ be defined as in Figure~\ref{fig:quantum_growth4}. 
  If $e_1$ or $e_3$ are open in the second diagram, then $h^k_{j+1} = \sfj+1$ for any outcome of the star-triangle transformations 
  \footnote{Due to the conditioning, $e_1$ or $e_3$ may only be open if their top endpoint lies outside $D^k_j$.}. 
  Assume that $e_1$ and $e_3$ are both closed at this stage. 
  Moreover, since the conditioning only depends on edges below level $\sfj+1$, it influences the state of $e_2$ only via boundary conditions.
  Hence, $\bbP [ e_2 \text{ is closed} \,|\, \Gamma \text{ and } e_1, e_3 \text{ closed} ] \geq 1-p_{e_2}$. 
  As discussed above, when $e_1$, $e_2$ and $e_3$ are all closed, 
  the final star-triangle transformation of Figure~\ref{fig:quantum_growth4} produces an open edge $e_4$ with probability bounded below by $1 - C\eps$. We conclude as above. 

\end{proof}

\subsection{The case $q > 4$}

We will adapt the proof of the exponential decay of Section~\ref{sec:q>4} to the quantum case.
More precisely, we only need to do so for the case of isoradial square lattices, that is Section~\ref{sec:sq_lat_q>4}. 
The argument is very similar to that of Section~\ref{sec:sq_lat_q>4}, with the exception that 
Propositions~\ref{prop:Gepsilon_h} and~\ref{prop:Gepsilon_v} are used instead of Propositions~\ref{prop:horizontal_transport} and~\ref{prop:vertical_transport}.

We recall the notation $\HR$ for half-plane rectangles and the additional subscript $\eps$ for domains defined in $\bbG^\eps$.
The key result is the following.

\begin{prop} \label{prop:quantum_sq_lat_q>4}
  There exist constants $C, c > 0$ depending only on $q$ such that, for any $\eps>0$ small enough,
  \begin{align} \label{eq:quantum_exp_decay4}
    \phi_{\HReps (N; \frac{N}{\eps} )}^{0} \big[ 0 \leftrightarrow \pd \HReps (n; \tfrac{n}{\eps}) \big]
    \leq C \exp(-c n), \qquad \forall n < N.
  \end{align}
\end{prop}

The above has the following direct consequences.

\begin{cor} \label{cor:quantum1_q>4}
  There exist constants $C, c > 0$ depending only on $q$ such that for $\eps$ small enough,
  \begin{align} \label{eq:quantum_exp_decay5}
    \phi_{\recteps(N; \frac{N}{\eps} )}^{0} \big[ 0 \leftrightarrow \pd \recteps(n; \tfrac{n}{\eps}) \big] 
    \leq C \exp(-c n), \qquad \forall n < N.
  \end{align}
\end{cor}

\begin{cor} \label{cor:quantum2_q>4}
  There exist constants $C, c > 0$ depending only on $q$ such that,
  \begin{align} \label{eq:quantum_exp_decay6}
    \phi_{\calQ, \Lambda(N)}^{0} \big[ 0 \leftrightarrow \pd \Lambda(n) \big]
    \leq C \exp(-c n).
  \end{align}
\end{cor}

Corollary~\ref{cor:quantum1_q>4} is an straightforward adaptation of Proposition~\ref{prop:exp_decay_half_full}.
Corollary~\ref{cor:quantum2_q>4} is a consequence of the fact that the constants in~\eqref{eq:quantum_exp_decay5} are uniform, thus we can take $\eps \rightarrow 0$ and apply Proposition~\ref{prop:Qboundary}.

To conclude, as in Section~\ref{sec:conclusions_q>4}, Corollary~\ref{cor:quantum2_q>4} implies Theorem~\ref{thm:quantum} for $q > 4$.

We will not give more details on the proofs of Corollaries~\ref{cor:quantum1_q>4} and~\ref{cor:quantum2_q>4} and Theorem~\ref{thm:quantum}. For the rest of the section, we focus on showing Proposition~\ref{prop:quantum_sq_lat_q>4}.

\begin{proof} [Proof of Proposition~\ref{prop:quantum_sq_lat_q>4}]
  We follow the idea of the proof of Proposition~\ref{prop:sq_lat_q>4}.

  Fix $\eps > 0$ as in the statement. 
  For $N > n$, let $G_\mix$ be the mixture of $\graph 1 = \bbG^\eps$ and $\graph 2 = \bbZ^2$, as described in Section~\ref{sec:switching_isos}. 
  In this proof, the mixed lattice is only constructed above the base level; 
  it has $2N+1$ vertical tracks $(s_{i})_{-N \leq i \leq N}$ of transverse angle $0$ and $\tfrac{N}{\eps} + N + 2$ horizontal tracks $(t_{j})_{0 \leq j \leq \frac{N}{\eps} + N+1}$, the first $\frac{N}{\eps} + 1$ having alternate angles $\eps$ and $\pi-\eps$ (we call this the irregular block)
  and the following $N+1$ having transverse angle $\frac\pi2$ (we call this the regular block).
  Finally, $G_\mix$ is a convexification of the piece of square lattice described above.
  
  Set $\tilde G_\mix$ to be the result of the inversion of the regular and irregular blocks of $G_\mix$ 
  using the sequence of transformations $\Sigma^{\uparrow}$.	
  Let $\phi_{G_\mix}$ and $\phi_{\tilde G_\mix}$ be the \rcms 
  with free boundary conditions on $G_\mix$ and $\tilde G_\mix$, respectively. 
  The latter is then the push-forward of the former by the sequence of transformations $\Sigma^{\uparrow}$.
  
  Let $\delta_0 \in (0,1)$ be a constant that will be set below; 
  it will be chosen only depending on $q$. 
  Write $\pd_L$, $\pd_R$ and $\pd_T$ for the left, right and top boundaries, respectively, of a rectangular domain $\HReps(.;.)$. 

  Consider a configuration $\omega$ on $G_\mix$ such that $0 \xleftrightarrow{} \pd \HReps(n; \tfrac{n}{\eps})$.
  Then, as in~\eqref{eq:con1}, we have
  \begin{align}
    \phi_{G_\mix} \big[ 0 \leftrightarrow \pd \HReps(n; \tfrac{n}{\eps}) \big]
    \leq & \, \phi_{G_\mix} \big[ 0 \xleftrightarrow{\HReps(n; \delta_0 \frac{n}{\eps})} \pd_L \HReps(n; \delta_0 \tfrac{n}{\eps}) \big] \nonumber \\
    + & \, \phi_{G_\mix} \big[ 0 \xleftrightarrow{\HReps(n; \delta_0 \frac{n}{\eps})} \pd_R \HReps(n; \delta_0 \tfrac{n}{\eps}) \big] \nonumber\\
    + & \, \phi_{G_\mix} \big[ 0 \xleftrightarrow{\HReps(n; \delta_0 \frac{n}{\eps})} \pd_T \HReps(n; \delta_0 \tfrac{n}{\eps}) \big].
    \label{eq:quantum_con1}
  \end{align}
  Moreover, since the graphs $G_\mix$ and $\bbG^\eps$ are identical in $\HReps(N; \frac{N}{\eps})$, we obtain, 
  \begin{align} \label{eq:quantum_GG_mix}
    \phi_{\HReps(N; \frac{N}{\eps})}^{0} \big[ 0 \leftrightarrow \pd \HReps(n; \tfrac{n}{\eps}) \big]
    \leq \phi_{G_\mix} \big[ 0 \leftrightarrow \pd \HReps(n; \tfrac{n}{\eps}) \big],
  \end{align}
  where we use the comparison between boundary conditions. 
  
  In conclusion, in order to obtain~\eqref{eq:quantum_exp_decay4} it suffices to prove that the three probabilities of the right-hand side of~\eqref{eq:quantum_con1} are bounded by an expression of the form $C e^{-cn}$, where the constants $C$ and $c$ depend only on $q$. 
  We concentrate on this from now on. 
  
  \medbreak

  Let us start with the last line of~\eqref{eq:quantum_con1}.
  Recall Proposition~\ref{prop:Gepsilon_v}; a straightforward adaptation reads:
  \medskip
  
  \noindent {\textbf{Adaptation of Proposition~\ref{prop:Gepsilon_v}.}}
  {\em 
    There exist $\tau > 0$ and $c_n >0$ satisfying $c_n\to 1$ as $n \to \infty$ such that, for all $n$ and sizes $N \geq 4n$, 
    \begin{align}
      \phi_{\tilde G_\mix} \big[ 0 \xleftrightarrow{\HReps(4n; \delta_0 \tau n)} \pd_T \HReps(4n; \delta_0 \tau n) \big]
      \geq c_n \phi_{G_\mix} \big[ 0 \xleftrightarrow{\HReps(n; \delta_0 \frac{n}{\eps} )} \pd_T \HReps(n; \delta_0 \tfrac{n}{\eps} ) \big].
      \label{eq:quantum_Gepsilon_v}
    \end{align}
  }
  
  Indeed, the proof of the above is identical to that of Proposition~\ref{prop:Gepsilon_v} with the only difference that the position of the two graphs are switched, thus the factor $\eps^{-1}$ becomes $\eps$. 
  The constant $\tau$ and the sequence $(c_n)_n$ only depend on $q$. 
  
  Observe that, in $\tilde G_\mix$, the domain $\HR(N; N)$ is fully contained in the regular block and contains $\HR(4n; \delta_0 \tau n )$ if $\delta_0 \tau \leq 1$. 
  Thus, by comparison between boundary conditions, 
  \begin{align*}
    & \phi_{\tilde G_\mix} \big[ 0 \xleftrightarrow{\HR(4n; \delta_0 \tau n )} \pd_T \HR(4n; \delta_0 \tau n ) \big] \\
    & \qquad \leq \phi_{\HR(N; N)}^{1/0} \big[ 0 \xleftrightarrow{\HR(4n; \delta_0 \tau n )} \pd_T \HR(4n; \delta_0 \tau n ) \big] \\
    & \qquad \leq \phi_{\HR(N; N)}^{1/0} \big[ 0 \xleftrightarrow{} \pd \sq(\delta_0 \tau n) \big]
    \leq C_0 \exp(-c_0 \delta_0 \tau n).
  \end{align*}
  where $\HR(N; N)$ is the subgraph of $\bbZ^2$ and the last inequality is given by Proposition~\ref{prop:input_exp}.
  Note that $c_0$ and $\tau$ depend only on $q$.
  Thus, from~\eqref{eq:quantum_Gepsilon_v} and the above we obtain,
  \begin{align}
    \phi_{G_\mix} \big[ 0 \xleftrightarrow{\HR(4n; \delta_0 \tau n)} \pd_T \HR(4n; \delta_0 \tau n)  \big] 
    \leq \frac{1}{c_n} C_0 \exp(-c_0 \delta_0 \tau n).
    \label{eq:quantum_con2}
  \end{align}
  For $n$ large enough, we have $c_n > 1/2$, and the left-hand side of~\eqref{eq:quantum_con2} is smaller than $2C_0\exp(-c_0\delta_0 \tau n)$.
  Since the threshold for $n$ and the constants $c_0, \tau$ and $\delta_0$ only depend on $q$, 
  the bound is of the required form. 
  
  \medbreak
  
  We now focus on bounding the probabilities of connection to the left and right boundaries of $\HReps (n; \delta_0 \tfrac{n}{\eps})$. 

  Observe that, for a configuration such that the event $\{ 0 \xleftrightarrow{\HReps(n; \delta_0 \frac{n}{\eps})} \pd_R \HReps(n; \delta_0 \tfrac{n}{\eps} ) \}$ occurs,
  it suffices to change the state of at most $\delta_0 \tfrac{n}{\eps}$ edges to connect $0$ to the vertex $x_{0,n}$ 
  (we will assume here $n$ to be even, otherwise $x_{0,n}$ should be replaced by $x_{0,n+1}$).
  Moreover, these edges can be chosen to be vertical ones in the irregular block, thus they are all ``short'' edges with subtended angle $\eps$.
  By the finite-energy property, there exists a constant $\tau = \tau(\eps, q) \in (0, 1)$ such that 
  \begin{align*}
    \phi_{G_\mix} \big[ 0 \xleftrightarrow{\HReps(n; \delta_0 \frac{n}{\eps})} \pd_R \HReps(n; \delta_0 \tfrac{n}{\eps}) \big]
    \leq \tau^{- \delta_0 \tfrac{n}{\eps} } \phi_{G_\mix} \big[ 0 \leftrightarrow x_{0,n} \big],
  \end{align*}
  where $\tau$ can be estimated as follows,
  \begin{align*}
    \tau = \frac{p_\eps}{ p_\eps + (1 - p_\eps) q}
    = \frac{y_\eps}{ q + y_\eps} \geq 1 -c_1 \eps,
  \end{align*}
  where $c_1 > 0$ is a constant depending only on $q$.

  The points $0$ and $x_{0,n}$ are not affected by the transformations in $\Sigma^{\uparrow}$, therefore 
  \begin{align*}
    \phi_{G_\mix} \big[ 0 \leftrightarrow x_{0,n} \big]
    & = \phi_{\tilde G_\mix} \big[ 0 \leftrightarrow x_{0,n} \big] \\
    & \leq \phi_{\tilde G_\mix} \big[ 0 \leftrightarrow \pd \sq(n) \big] \\
    & \leq \phi_{\HR(N; N)}^{1/0} \big[ 0 \leftrightarrow \pd \sq(n) \big] \\
    & \leq C_0 \exp(-c_0 n),
  \end{align*}
  where in the last line, $\HR(N; N)$ is the subgraph of $\tilde G_\mix$, or equivalently of $\bbZ^2$ since these two are identical. 
  The last inequality is given by Proposition~\ref{prop:input_exp}.
  We conclude that, 
  \begin{align}
    \phi_{G_\mix} \big[ 0 \xleftrightarrow{\HReps(n; \delta_0 \frac{n}{\eps})} \pd_R \HReps(n; \delta_0 \tfrac{n}{\eps}) \big]
    &\leq C_0 \exp \big[ -(c_0 + \tfrac{\delta_0}{\eps} \log \tau) n \big]  \nonumber \\
    &\leq C_0 \exp \big[ -(c_0 + \tfrac{\delta_0}{\eps} \log (1 -c_1 \eps) ) n \big] . \label{eq:quantum_con3}
  \end{align}
  The same procedure also applies to the event $\big\{ 0 \xleftrightarrow{\HReps(n; \delta_0 \frac{n}{\eps})} \pd_L \HReps(n; \delta_0 \tfrac{n}{\eps}) \big\}$.
  
  Now let $\delta_1 = \frac{c_0 \eps}{c_0 \tau \eps - \log (1 - c_1 \eps) }$ and $\delta_0 = \min \{ \delta_1, \frac12 \}$.
  Notice that $\delta_1 \rightarrow \frac{c_0}{c_0 \tau + c_1} > 0$ when $\eps \to 0$, which gives the following relation,
  \begin{align*}
    c_0 + \frac{\delta_0}{\eps} \log (1 - c_1 \eps)
    \geq c_0 + \frac{\delta_1}{\eps} \log (1 - c_1 \eps)
    = c_0 \delta_1 \tau
    \longrightarrow \frac{c_0^2 \tau}{c_0 \tau + c_1} > 0,
  \end{align*}
  as $\eps \rightarrow 0$.
  Thus, for $\eps$ small enough, we can pick a uniform constant $\delta_0$ such that
  \begin{align*}
    c_0 + \frac{\delta_0}{\eps} \log(1 -c_1 \eps) \geq \frac{1}{2} c_0 \delta_1 \tau =: c.
  \end{align*}
  Then, Equations~\eqref{eq:quantum_con1}, \eqref{eq:quantum_con2} and~\eqref{eq:quantum_con3} 
  imply that for $n$ larger than some threshold depending only on $q$, 
  \begin{align*}
    \phi_{G_\mix} \big[ 0 \leftrightarrow \pd \Lambda(n) \big]
    \leq 4 C_0 \exp(-cn).
  \end{align*}
  Finally, by~\eqref{eq:quantum_GG_mix}, we deduce~\eqref{eq:quantum_exp_decay4} for all $N \geq 4n$ and $n$ large enough. 
  The condition on $n$ may be removed by adjusting the constant $C$. 
\end{proof}

\newpage

\bibliographystyle{acm}
\bibliography{fk}

\end{document}